\documentclass[nofootinbib,amsmath,amsthm,amssymb,dvipsnames,18pt]{revtex4-2}
\usepackage{graphicx}
\usepackage[latin1,utf8]{inputenc}
\usepackage{dcolumn}
\usepackage{hyperref}
\usepackage{bm}
\usepackage{epsfig}
\usepackage{xcolor}
\usepackage{framed}
\usepackage{amsmath}
\usepackage{amssymb}
\usepackage{amsthm}
\usepackage{cancel}
\usepackage{mathtools}
\graphicspath{{Images_Paper/}}
\usepackage[english]{babel}
\usepackage{placeins}
\usepackage{float}
\usepackage{caption}
\usepackage{ulem}
\usepackage[font=footnotesize,labelsep=quad,justification = justified]{caption}
\newtheorem{theorem}{Theorem}
\newtheorem{conjecture}{Conjecture}
\newtheorem{corollary}[theorem]{Corollary}
\newtheorem*{theorem*}{Theorem}
\newtheorem{lemma}[theorem]{Lemma}

\renewcommand{\qedsymbol}{$\blacksquare$}

\theoremstyle{plain}

\theoremstyle{plain}

\theoremstyle{plain}

\begin{document}

%\title{Haros graphs: an exotic representation of real numbers}

%\shorttitle{Haros graphs: an exotic representation of real numbers} %%%for recto running head
%\shortauthorlist{J. Calero-Sanz, B. Luque, L. Lacasa} %%% for verso running head

%\author{
%\name{Jorge Calero-Sanz$^{1,2}$}
%\address{$^2$Signal and Communications Theory and Telematic Systems and Computing, Rey Juan Carlos University, Madrid, Spain \email{jorge.calero@urjc.es}}
%\name{Bartolo Luque$^{1}$}
%\address{$^1$Departamento de Matem\'atica Aplicada a la Ingenier\'ia Aeroespacial, ETSIAE, Universidad Polit\'ecnica de Madrid, Madrid, Spain \email{bartolome.luque@upm.es}}
%\and
%\name{Lucas Lacasa$^{3}$}
%\address{$^3$Institute for Cross-Disciplinary Physics and Complex Systems IFISC (CSIC-UIB), Palma de Mallorca, Spain.\email{lucas@ifisc.uib-csic.es}}}

\title{Haros graphs: an exotic representation of real numbers}\
\author{Jorge Calero-Sanz$^{1,2}$}
\email{jorge.calero@urjc.es}
\author{Bartolo Luque$^{1}$}
\email{bartolome.luque@upm.es}
\author{Lucas Lacasa$^{3}$ }
\email{lucas@ifisc.uib-csic.es}
\affiliation{$^1$Departamento de Matem\'atica Aplicada a la Ingenier\'ia Aeroespacial, ETSIAE, Universidad Polit\'ecnica de Madrid, Madrid, Spain;\\$^2$Signal and Communications Theory and Telematic Systems and Computing, Rey Juan Carlos University, Madrid, Spain \\$^3$Institute for Cross-Disciplinary Physics and Complex Systems IFISC (CSIC-UIB), Palma de Mallorca, Spain.}

\date{\today}

\begin{abstract}
{This paper introduces Haros graphs, a construction which provides a graph-theoretical representation of real numbers in the unit interval reached via paths in the Farey binary tree. We show how the topological structure of Haros graphs yields a natural classification of the reals numbers into a hierarchy of families. To unveil such classification, we introduce an entropic functional on these graphs and show that it can be expressed, thanks to its fractal nature, in terms of a generalised de Rham curve. We show that this entropy reaches a global maximum at the reciprocal of the Golden number and otherwise displays a rich hierarchy of local maxima and minima that relate to specific families of irrationals (noble numbers) and rationals, overall providing an exotic classification and representation of the reals numbers according to entropic principles. We close the paper with a number of conjectures and outline a research programme on Haros graphs.}
%{Keywords here}
\end{abstract}
%\pacs{1}% 
\maketitle

\section{Introduction}

The structure of real numbers has been a traditional topic of study in mathematics \cite{Hardy,Knuth, Angell}, tackled from several angles, both from a computational perspective --e.g. the quest of exact arithmetic for the approximation of rationals and reals \cite{Niqui, Vuillemin} -- and from a theoretical one, e.g. the characterisation of different families of irrationals such as transcendental or normal numbers \cite{Nivenetal, Hardy}. Among other techniques \cite{Angell}, representing numbers in terms of a continued fraction expansion \cite{AdamczewskiII, Khinchin} has been a very fruitful avenue. A classical approach is to consider the Farey binary tree, where rationals acquire a natural ordering \cite{Bonnano}. Some works \cite{Singerman, Bonnano} present the relationship between Farey sequences and continued fractions. Moreover, other objects can be related such as the modular group \cite{Vepstas}, and the dynamic and thermodynamic perspectives are highlighted \cite{Isola}.\\

\noindent In parallel and also recently, a new technique for representing sequences of numbers in terms of graphs has been proposed \cite{FromTime, HVG}: Visibility Graphs--either Natural Visibility Graphs (VGs) or Horizontal Visibility Graphs (HVGs)-- are combinatorial representations of sequences of numbers, and a number of works \cite{Irreversibility, Detectingperiodicity, Barabasi}  have made it clear that the structure of such graphs inherit several properties of the associated sequences, thereby making them a useful technique for time series analysis and feature-based classification of complex signals. In addition to their applicability for signal processing, a range of theoretic works have studied how the trajectories of dynamical systems are mapped in graph space. In this sense, different routes to low-dimensional chaos have been explored, including the Feigenbaum scenario \cite{FeigenbaumGraphs}, the Pomeau-Manneville scenario \cite{intermittency}, and the quasiperiodic route \cite{QuasiperiodicGraphs}, and a number of dynamical properties such as Lyapunov exponents or renormalization-group features have been shown to be inherited in the associated graphs. Interestingly, the Horizontal Visibility Graph extracted from the circle map certifies that the mode locking frequencies --that relate to specific rationals-- have a very concrete graph-theoretical footprint, and similarly some selected irrationals --i.e. the Golden number, and others-- acquire a natural graph-theoretic representation. \\

\noindent Motivated by these facts, in this work, we build on the dynamical association between rationals and HVGs and proceed to construct a graph-theoretical representation of rational and irrational numbers in the Farey binary tree. Our contention is that the topological structure of such graphs inherits combinatorially important properties of numbers. We will show that such an exotic graph-theoretical representation allows for a classification of rationals and irrationals numbers, as well as for finer classifications inside the set of irrationals, which unveils an extremely rich and self-similar structure. \\

\noindent The rest of the paper goes as follows: In Section II we present the necessary background on Farey sequences, the Farey binary tree, and continued fractions. In Section III, we introduce Haros graphs, which are graph-theoretical representations of rationals and irrationals numbers. Note at this point that Haros graphs are different from the so-called Farey graphs \cite{Zhang, Singerman}, hence our labelling \footnote{Interestintly, the classical concept of Farey graph \cite{Zhang} is indeed similar to the concept of Feigenbaum graph, HVGs extracted from the Feigenbaum scenario in the period-doubling bifurcation cascade \cite{FeigenbaumGraphs}}. In this section, we also explore both analytically and numerically the degree distribution of these graphs in relation to the number they represent, and we also study properties of the degree sequence in relation to Khinchin's constant. In Section IV we proceed to study the entropy of the degree distribution, which we find to have a fractal shape. We show that such entropy has a closed-form shape in terms of a generalised de Rham curve: a self-affine function, continuous but not differentiable at any point. By further exploring the extrema of such curve, we find that it has an infinite number of local minima, which can be classified into families of rational numbers, and in turn, the infinite number of local maxima are related to irrationals that can be naturally parametrised into families of noble numbers. In Section V, we conclude, state some conjectures, and outline several open questions that can form the basis for a research programme on Haros graphs.

\section{Preliminaries} 
%English geologist and essayist John Farey (1766-1826) wrote articles on topics as diverse as geology, music, coins, wagon wheels or kites \cite{Motif}. 

\subsection{Farey sequences}
In 1816, the geologist John Farey \cite{Farey}
%published a curious result on fractions in Edinburgh magazine and Dublin Philosophical Magazine \cite{Farey}. Farey 
proposed to order from the smallest to the largest the irreducible fractions with a denominator less than or equal to $n$ belonging to the interval $[0,1]$, hence defining the so-called {\it Farey sequence} of order $n$, ${\cal F}_n$, as the ordered set of all irreducible fractions $0<p/q<1, \ p,q\in \mathbb{N}^+$ whose denominators do not exceed $n$. The first three sequences are:\\
$$\mathcal{F}_{1}=\left\lbrace\frac{0}{1},  \frac{1}{1}\right\rbrace; \ \mathcal{F}_{2}=\left\lbrace\frac{0}{1},\frac{1}{2}, \frac{1}{1}\right\rbrace; \ \mathcal{F}_{3}=\left\lbrace\frac{0}{1}, \frac{1}{3},  \frac{1}{2}, \frac{2}{3},\frac{1}{1}\right\rbrace; \  \mathcal{F}_{4}=\left\lbrace\frac{0}{1}, \frac{1}{4}, \frac{1}{3},  \frac{1}{2}, \frac{2}{3},\frac{3}{4},\frac{1}{1}\right\rbrace$$\\
and in general:
 $\mathcal{F}_{n} = \left\lbrace p/q \in [0,1]:\ 0\leq p \leq q \leq n, \ (p,q) = 1 \right\rbrace$.\\
 \\
\noindent The mediant $\oplus$ of two rational numbers $p/q$ and $p'/q'$ is defined as: $$\frac{p}{q}\oplus\frac{p'}{q'}:=\frac{p+p'}{q+q'}$$ For two consecutive rationals in a Farey sequences ${\cal F}_n$, the mediant of these terms is indeed another irreducible fraction: $$\frac{p}{q} < \frac{p+p'}{q+q'} < \frac{p'}{q'}$$
that does not appear in ${\cal F}_n$ but, provided that $q+q'\leq n+1$, appears in ${\cal F}_{n+1}$  \cite{Knuth, Hardy}.
Thus, one can iteratively construct Farey sequences, using ${\cal F}_n$ and the mediant operator to construct ${\cal F}_{n+1}$.  
For instance, using the mediant sum over $\mathcal{F}_{1}$ we get: $$\frac{0}{1}\oplus\frac{1}{1}= {\frac{1}{2}},$$ 
which is the middle term in ${\cal F}_{2}$. Similarly: 
%${\cal F}_{2}=\left\lbrace \mathbf{\frac{0}{1}},\frac{0}{1}\oplus\frac{1}{1}= \mathbf{\frac{1}{2}}, \mathbf{\frac{1}{1}}\right\rbrace; \quad 
$${\cal F}_{3}=\left\lbrace \mathbf{\frac{0}{1}},\frac{0}{1}\oplus\frac{1}{2}= \mathbf{\frac{1}{3}}, \mathbf{ \frac{1}{2} }, \frac{1}{2}\oplus\frac{1}{1}= \mathbf{\frac{2}{3}},  \mathbf{\frac{1}{1}} \right\rbrace.$$ Observe, however, that some of the terms resulting from the mediant sum over $\mathcal{F}_3$ ($\frac{1}{3}\oplus\frac{1}{2}=\frac{2}{5}$ and $\frac{1}{2}\oplus\frac{2}{3}=\frac{3}{5}$) are not in $\mathcal{F}_4$ as the denominator exceeds $n=4$, in this sense $\mathcal{F}_n$ is a subset of the set formed by sequentially applying the mediant sum to adjacent pairs of terms in $\mathcal{F}_{n-1}$.\\

\noindent Each Farey sequence $\mathcal{F}_n$ induces a partition of $[0,1]$ into $n$ subintervals (generally of different sizes), which we will call Farey subintervals of order $n$. For example, $\mathcal{F}_3$ induces the partition: $$[0,1] = \left[0, \frac{1}{3}\right) \bigcup \left[\frac{1}{3},  \frac{1}{2}\right) \bigcup \left[\frac{1}{2}, \frac{2}{3}\right) \bigcup \left[\frac{2}{3}, 1\right].$$ 
As $n$ grows, it is easy to see that ${\cal F}_n$ provides finer approximations for the ordering of rational numbers, so that all rational numbers will eventually be ordered and enumerated in the limit set ${\cal F}:=\lim_{n\to \infty}{\cal F}_n$. Therefore, all rational numbers in the unit interval are contained in ${\cal F}$ \cite{Knuth}, something conjectured by Farey and independently proved by Haros and Cauchy \cite{Motif}.
At this point, we should emphasise that, while popularly known as Farey sequences, ${\cal F}_n$ were originally and independently studied by Charles Haros, publishing similar results (including some proofs) some 15 years before Farey \cite{beiler1964recreations}. In tribute, our graph-theoretic construction (Section III) is called Haros graphs.

\subsection{Farey binary tree, binary sequences and continued fractions}

\noindent A classical way of visualising the different elements of ${\cal F}_n$ is using the so-called {\it Farey binary tree} (see Fig. \ref{fig:FareyBinaryTree} for an illustration). 
This is an infinite complete binary tree (adding the nodes $0/1$ and $1/1$ as special cases) where each node corresponds to a given rational, computed using the mediant sum over specific numbers above in the tree. This tree is hierarchically organised in levels $\ell_n$, such that the first 4 levels are:
$$\ell_1=\left\lbrace\frac{0}{1},  \frac{1}{1}\right\rbrace; \ \ell_{2}=\left\lbrace\frac{1}{2}\right\rbrace; \ \ell_{3}=\left\lbrace \frac{1}{3},  \frac{2}{3}\right\rbrace, \  \ell_{4}=\left\lbrace\frac{1}{4}, \frac{2}{5}, \frac{3}{5},  \frac{3}{4} \right\rbrace$$
For $n\leq 3$, we have ${\cal F}_n=\cup_{i=1}^n \ell_i$, i.e. the Farey sequence corresponds to projecting all ancestors in the Farey binary tree up to that level. However, this is no longer true for $n>3$, for which we have: 

$${\cal F}_n \subset \bigcup_{i=1}^n \ell_i$$

For instance: $\cup_{i=1}^4 \ell_i= {\cal F}_4 \cup \{2/5,3/5\} $ (the mismatch is due to the definition of the Farey sequence, which imposes the denominator $q\leq n$). While the offset: 
$$\vert \bigcup_{i=1}^n \ell_i \vert-\vert {\cal F}_n \vert$$

grows without bounds, in the limit $n\to \infty$, one can prove that:

$$\text{card}\left(\bigcup_{i=1}^\infty \ell_i\right)=\text{card}({\cal F})$$

%that $\lim_{n\to \infty} |{\cal F}_n|/|\cup_{i=1}^n \ell_i|=1$, that any element in $\cup_{i=1}^\infty\ell_n$ is also in $\cal F$, 
and, therefore, the Farey binary tree orders all real numbers \cite{Niqui}.\\

In this work, we consider the ordering as provided by the Farey binary tree rather than by the Farey sequences. However, for computational reasons, the numerical calculus used to represent some figures has been computed for the Farey sequence $\mathcal{F}_{1000}$. There are exponentially many different downstream paths of a certain depth in the Farey tree starting from $1/1$ and one can enumerate these by assigning to each path a unique binary sequence of symbols $L/R$ indicating descending to the Left/Right, respectively. For example, it can be observed in Fig. \ref{fig:FareyBinaryTree} that the rational number $8/13$ follows the symbolic path $LRLRL$ in the Farey Binary Tree. Since any finite path has an ending element $p/q$ and such rational is reached by a single finite path in the Farey binary Tree, one can establish a bijection between finite binary sequences and rational numbers in the unit interval \cite{Bonnano,Isola}. \\

\begin{figure}[htb]
\includegraphics[width=1\textwidth, trim=2cm 4cm 2cm 4cm,clip=true]{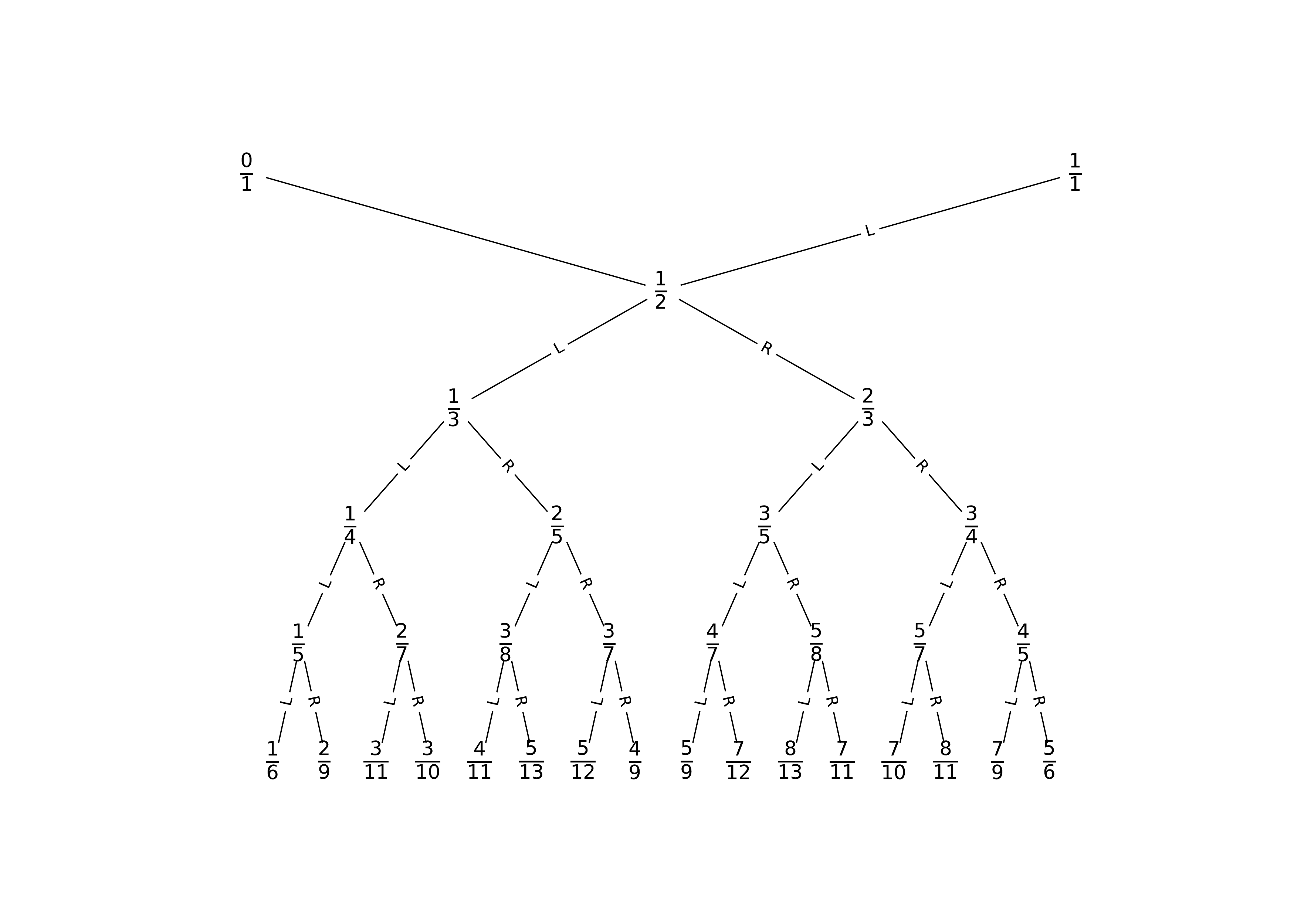}
\caption{Six levels of the Farey binary tree, where the special nodes $0/1$ and $1/1$ form the first level, $\ell_1$. Every irreducible rational number appears only once in the tree and can be reached following a path which is encoded in a symbolic binary sequence Left/Right. By convention, paths start in $1/1$; hence, the first symbol in every path will be $L$. The Farey sequence $\mathcal{F}_n$ is a subset of the rationals appearing in $\cup_{i=1}^n\ell_i$.}
\label{fig:FareyBinaryTree}
\end{figure}

\noindent Another way of representing rationals is via continued fractions. It is well known that irrational numbers in $[0,1]$ admit a description in terms of an infinite continued fraction, whose finite truncations (called convergents) are rational numbers. For instance, the reciprocal of the Golden number $\phi^{-1} = 0.618...$ is a solution of the quadratic equation $x^2 + x - 1 = 0$. Therefore, it fulfils $\phi^{-1} = \frac{1}{1 + \phi^{-1}}$ which, when interpreted recursively, leads us to: $$\phi^{-1} = \frac{1}{1 + \frac{1}{1 + \frac{1}{1 + ...}}}$$
The above equation determines that the reciprocal of the Golden number admits an expression as a continued fraction as: $\phi^{-1} = [1, 1, ...] = [\bar{1}]$. The successive convergents of this continued fraction are: $[1] = 1$, $[1,1] = 1/2$, $[1,1,1] = 2/3$, $[1,1,1,1] = 3/5$, and so on.
Interestingly, note that the Fibonacci sequence is generated by the recurrence equation $\mathfrak{F}_n = \mathfrak{F}_{n-1} + \mathfrak{F}_{n-2}$ and the initial conditions $\mathfrak{F}_0 = \mathfrak{F}_1 = 1$. The Fibonacci quotients $\left\lbrace \mathfrak{F}_{n-1}/\mathfrak{F}_n \right\rbrace_{n \geq 1}$ correspond to the convergents of $\phi^{-1}$ and: $\lim_{n\to \infty} \mathfrak{F}_{n-1}/\mathfrak{F}_n =  \phi^{-1}$.\\
\\
\noindent Now, for any given rational (or irrational) number, there exists a close relationship between its continued fraction expansion and its representation in terms of a binary sequence in the Farey binary tree \cite{Bonnano}. Indeed, a continued fraction $[a_1, a_2 ,... ]$ has a correspondent symbolic path in the Farey binary tree as $L^{a_1} R^{a_2} L^{a_3} ... $, where $L^q$ is interpreted as a sequence of $q$ Ls (in the rational case, the last symbol has an index $a_n -1$). For instance, $\phi^{-1} = [1, 1, 1,1,1,...] = LRLRLR...=\overline{LR}$, that is, an infinite zigzag.\\
\\
 \noindent As a result, and since any two subsequent convergents can be reached by a specific path in the Farey binary tree, it turns out that any irrational number in $[0,1]$ can be reached asymptotically via a particular infinite path. In other words, any irrational can be approximated in $[0,1]$ with arbitrary precision via convergents \cite{Nivenetal}, in such a way that in the limit the irrationals are reached asymptotically and therefore $\cup_{i=1}^\infty \ell_i={\cal F} = [0,1]$.

\section{Haros graphs}
\subsection{Definition and basic properties}
We start by considering a graph concatenation operator $\oplus$ (note that this is different from the mediant operator $\oplus$ defined in the previous section, but the abuse of notation is convenient and will become evident later). This operator was originally defined with a different notation (a so-called inflation operator) in \cite{flanagan2019spectral}, and we refer the reader to that paper for a precise definition. Here, instead, we provide a visual description, which we believe is more intuitive than the formal definition. Visually, the concatenation of two graphs consists on merging two extreme nodes (so-called boundary nodes) and connect the first and last nodes of the new graph with a new link, see Fig. \ref{fig:Fig2_FirstFarey} for an illustration. The boundary nodes of the two graphs that are involved in the merging are called merging nodes. It is easy to see that $\oplus$ is not commutative.\\
We are now ready to define Haros graphs:\\

\begin{figure}[htb]
\includegraphics[width=0.8\textwidth , trim=0cm 17cm 0cm 5cm,clip=true]{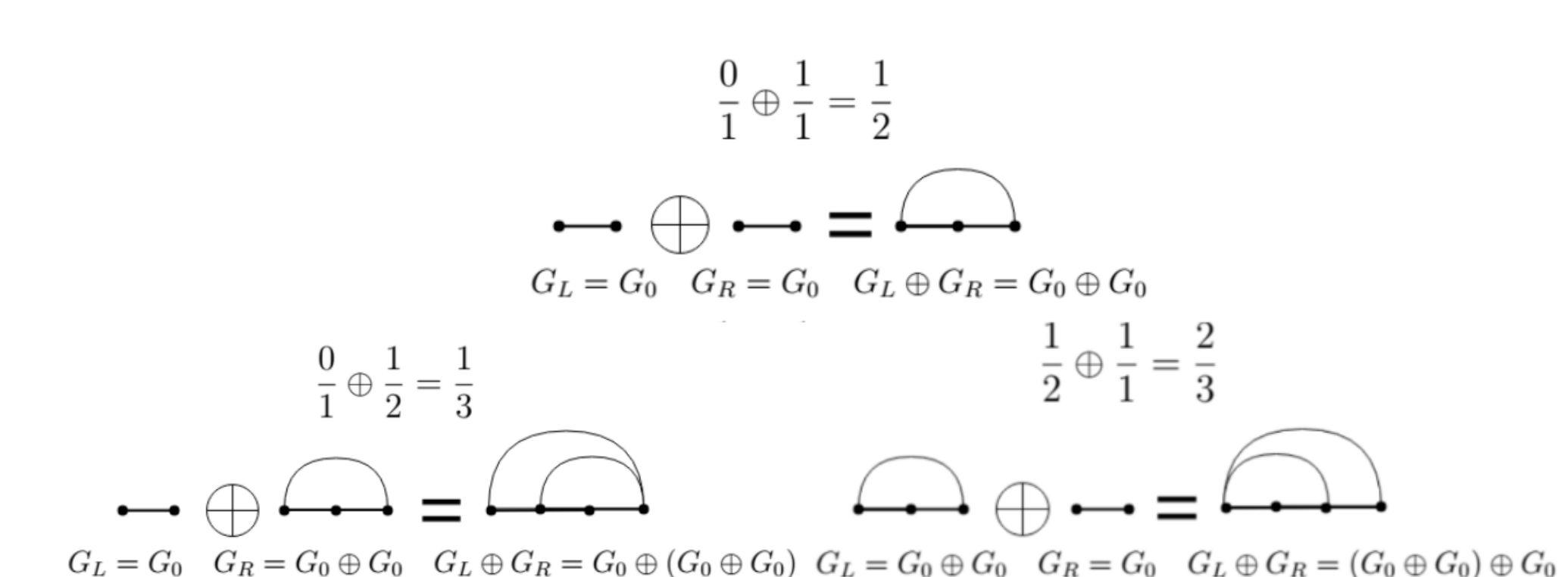}
\caption{(a) First line: Construction of $G_1 = G_0 \oplus G_0$. The graph concatenation operator merges the boundary nodes: labelling the nodes of $G_0$ as $[1,2]$, the boundary nodes are $2$ for $G_L$ and $1$ for $G_R$. The graph obtained by the concatenation has three nodes and the external nodes ($1$ and $3$ in $G_L \oplus G_R$) are linked by a new vertex.  \ (b) Second line: Construction of the two possible concatenations between $G_0$ and $G_1$.}
\label{fig:Fig2_FirstFarey}
\end{figure}

\noindent \textbf{Definition }(Haros graph)
\textit{Haros graphs are built by iteratively applying the concatenation operator $\oplus$ on a previous (ancestor) Haros graph, where the oldest ancestor $G_0$ consists of two nodes linked by an edge (see Fig. \ref{fig:FareyGraphTree} for an illustration).  Since $\oplus$ is not commutative, one can apply $\oplus$ in two ways to produce two offspring Haros graphs from the same ancestor Haros graph $G$, by either `left' concatenating  ($X_\ell\oplus  G$) and `right' concatenating ($G\oplus X_r$). $X_\ell$ is the closest `left' ancestor of $G$ and $X_r$ is the closest `right' ancestor of $G$ (Fig. \ref{fig:FareyGraphTree} clarifies this).}\\

\noindent The successive applications of $\oplus$ thus induce a binary sequence (left/right) very much like the one which enumerates paths in the Farey binary tree, actually Fig. \ref{fig:FareyGraphTree} is in one-to-one correspondence to the Farey binary tree. Since binary sequences describing paths in such tree are in a one-to-one correspondence with rationals, we can thus make a one-to-one correspondence between rationals and Haros graphs.\\

% We built these graphs, namely Haros graphs, recursively using a concatenation operation which acts over graphs as a mediant sum between rational. 
 
\noindent We illustrate now how to relate a given rational $x=p/q$ (where $p/q$ is an irreducible fraction) to its Haros graph $G_x$ (see again Figs. \ref{fig:Fig2_FirstFarey} and \ref{fig:FareyGraphTree} for an illustration): by definition we let $G_0 = G_{0/1} = G_{1/1}$ to be associated at the same time to rational numbers $0/1$ and $1/1$. Then we build $G_0\oplus G_0 $ and associate it with $G_{1/2}$, the Haros graph related to the rational number $1/2$.
At this point, we can now build two new possible Haros graphs, depending on how we concatenate again: $G_0\oplus G_{1/2}$,  which is identified with $G_{1/3}$, or $G_{1/2}\oplus G_0$, which is identified with $G_{2/3}$. 
In other words, we have the following equation: 
\begin{equation}
G_{p/q} \oplus G_{p'/q'}= G_{p/q \ \oplus \ p'/q'},
\label{eq:conca_eq}
\end{equation}

i.e. the concatenation operator between Haros graphs is identified with the mediant sum between rationals.
%In similar way to the mediant sum described above for Farey sequences. As in the case of the Farey sequences, we can generate new elements using a parent graphs that add up on the left and right, which we denote respectively $G_L$ and $G_R$ (see Fig. \ref{fig:Fig2_FirstFarey}).
%So, we build and define the set of Haros graphs taking $G_0$ as an Haros graph and applying the concatenation graph operation as the mediant sum in Farey sequences. 
As a result, we obtain a one-to-one correspondence between the set of Haros graphs and the set of Farey sequences, and accordingly every number $x \in [0,1]$ can now be represented by a unique Haros graph $G_x$.
If $x$ is rational, there exists a finite binary sequence in the Farey binary tree that reaches $x$, and equivalently there is a finite sequence of concatenations that reaches the associated Haros graph; thus, $G_x$ is a finite graph that we denote a rational Haros graph. On the contrary, if $x$ is irrational, then $G_x$ is a countable infinite graph that denote an irrational Haros graph.\\

\begin{figure}[htb]
\includegraphics[width=0.85\textwidth]{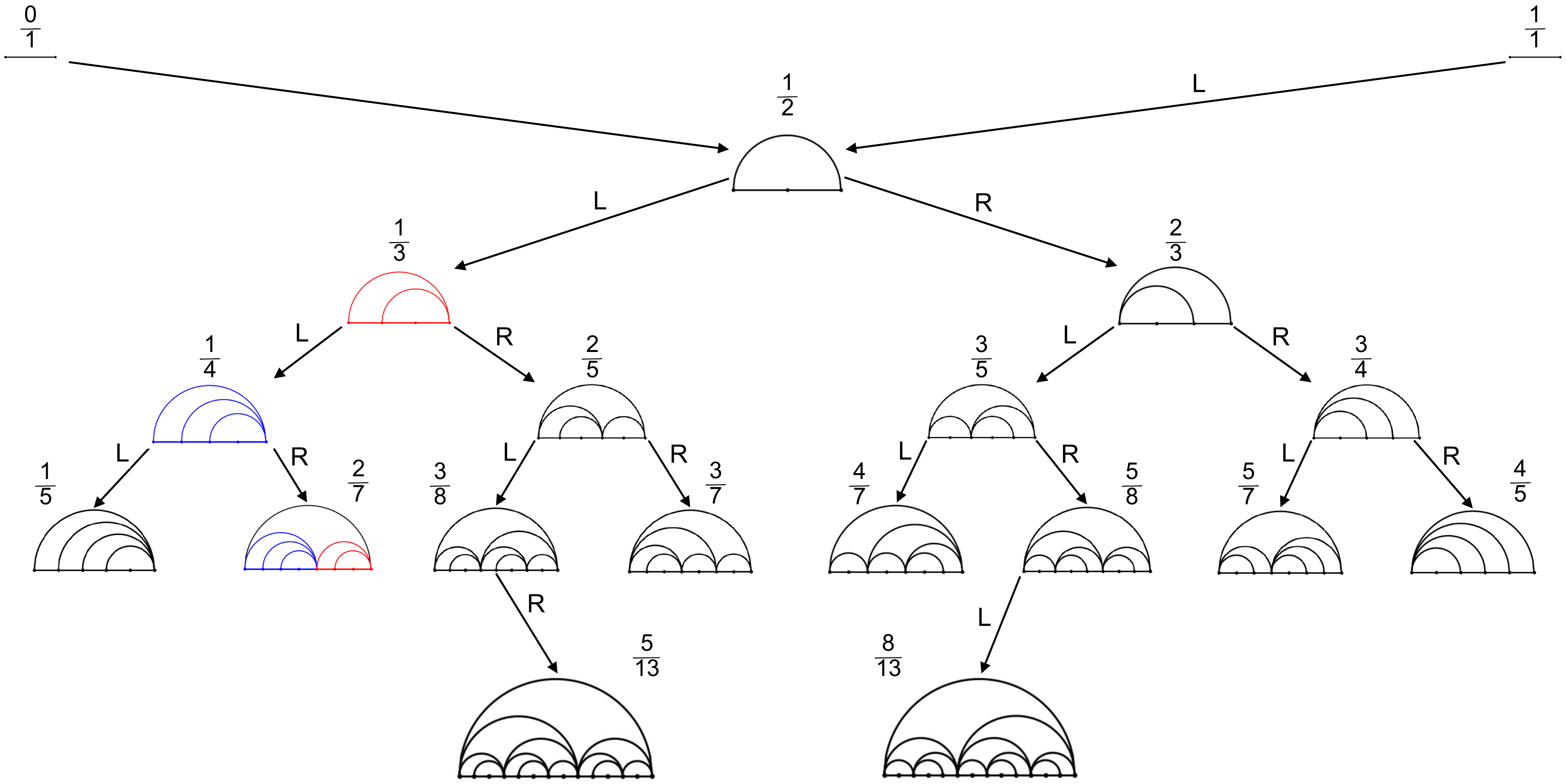}
\caption{Six levels of the Haros graph binary tree with the graphs associated with the corresponding rational fractions $p/q$ (for space reasons only two of these are shown at the sixth level). On the left, we highlight the Haros graph $G_{2/7}$, built as a concatenation of $G_{1/4}$ (blue) and $G_{1/3}$ (red).}
\label{fig:FareyGraphTree}
\end{figure}

\noindent \textit{Boundary node convention. } While by construction $G_{p/q}$ is an Haros graph of $q+1$ nodes, for convenience and to make calculations cleaner in what follows, we assume that the two extreme nodes of an Haros graph (linked by construction) can be identified as a single `boundary' node, while maintaining both incident edges (respectively, the rest of the nodes are called \textit{inner} nodes). For instance, if the first and last nodes have degree 2, then our convention is to say that the boundary node has degree 4. Such a boundary node is depicted at the end of the node sequence. 
%We will define the sequence of connectivity of an Haros graph to the sequence of connectivity of its successive nodes. 
For example, the degree sequence of $G_{1/2}$ is: $$[2,2,2] \to [2, 2 + 2] = [2, 4]$$ 
and for: $$G_{1/3} = [2,3,2,3] \to [3,2,2+3] = [3,2,5].$$ 
Similarly: $$G_{1/3} \oplus G_{1/2} = G_{2/5},$$ with degree sequences: $$[2,3,2,3]\oplus [2,2,2]= [2+1,3,2,3+2,2,2+1]=[3,3,2,5,2,3],$$ which is then presented as $[3,2,5,2,6]$.

%In Fig. \ref{fig:FareyGraphTree}, we show in the Farey binary tree a graph associated with every rational.

\subsection{Degree distribution}
\label{Ddistr}
While trivially rational numbers $x\in \mathbb{Q}$ have $G_{x}$ with a finite vertex set and irrational numbers $x\in \mathbb{I}$ have an infinite vertex set, we wonder if the structure of these graphs allows us to further distinguish among classes of irrational numbers. 
Among the plethora of different graph properties at our disposal, we shall concentrate on the degree sequence, i.e. the ordered sequence whose i-th term provides the degree (number of edges) of the i-th node in the graph. This choice is indeed informed by the fact that Haros graphs form a subset of Horizontal Visibility Graphs (HVG), and a recent theorem \cite{canonical} proved that HVGs are unigraphs, and thus they are uniquely determined by their degree sequence, i.e., the degree sequence is a `maximally informative' property of the graph \cite{LacasaVisibility}.
In this section, we consider the probability degree distribution $P(k,x)$, that provides the probability that a node chosen at random in $G_x$ has degree $k$.\\ 

By construction, Haros graphs do not have isolated nodes, so $P(0,x)=0, \ \forall x$. Similarly, and with the exception of $G_0$ and $G_1$, $P(1,x)=0, \ \forall x$, so the subsequent analysis is on $k\geq 2$. We start by focusing on the interval $x\in [0,1/2)$ and $P(k,x)$ for $k\leq 4$. We can state the following theorem:\\

\begin{theorem}
\label{Tma:P234}
$\forall x \in(0,1/2)$, the first values of the degree distribution $P(k,x)$ of $G_x$ are:
\begin{equation}
P(k,x)=\left\{
\begin{array}{ll}
x,  & k=2 \\
1 - 2x, & k=3 \\
0, & k=4 \\
\end{array}%
\right.
\label{eq1}
\end{equation}
\end{theorem}

See Appendix \ref{appendix:A} for a proof. Now, an important but obvious observation is that there exists a mirror symmetry with respect to $1/2$ ($x\to 1-x$) in the construction of Haros graphs (see Fig. \ref{fig:FareyGraphTree}) which directly implies $P(k,x)=P(k,1-x)$. Accordingly, for $x\in (1/2,1)$ the values in Eq. \ref{eq1} must be replaced by $P(2,x) = 1 - x$, $P(3,x) = 2x - 1$, and $P(4,x)=0$.\\ 

Note that, conditioned on a generic but fixed degree $k=\kappa$, $P(\kappa,x)$ is a real-valued function with domain $[0,1]$. In this sense, by construction, the boundary cases $G_{0/1}, G_{1/2}$ and $G_{1/1}$ are associated with isolated discontinuities of this function in $k=2,3,4$. In Fig. \ref{fig:P234} we show the numerical values for the functions $P(2,x), P(3,x)$ and $P(4,x)$ computed for all Haros graphs $G_x$ with $x \in {\cal F}_{1000}$, in full agreement with the previous theorem and highlighting the above mentioned mirror symmetry. Consequently, from now on, we will restrict the analysis to the interval $x\in [0,1/2]$.\\ 

\begin{figure}[htb]
\includegraphics[width=0.4\columnwidth]{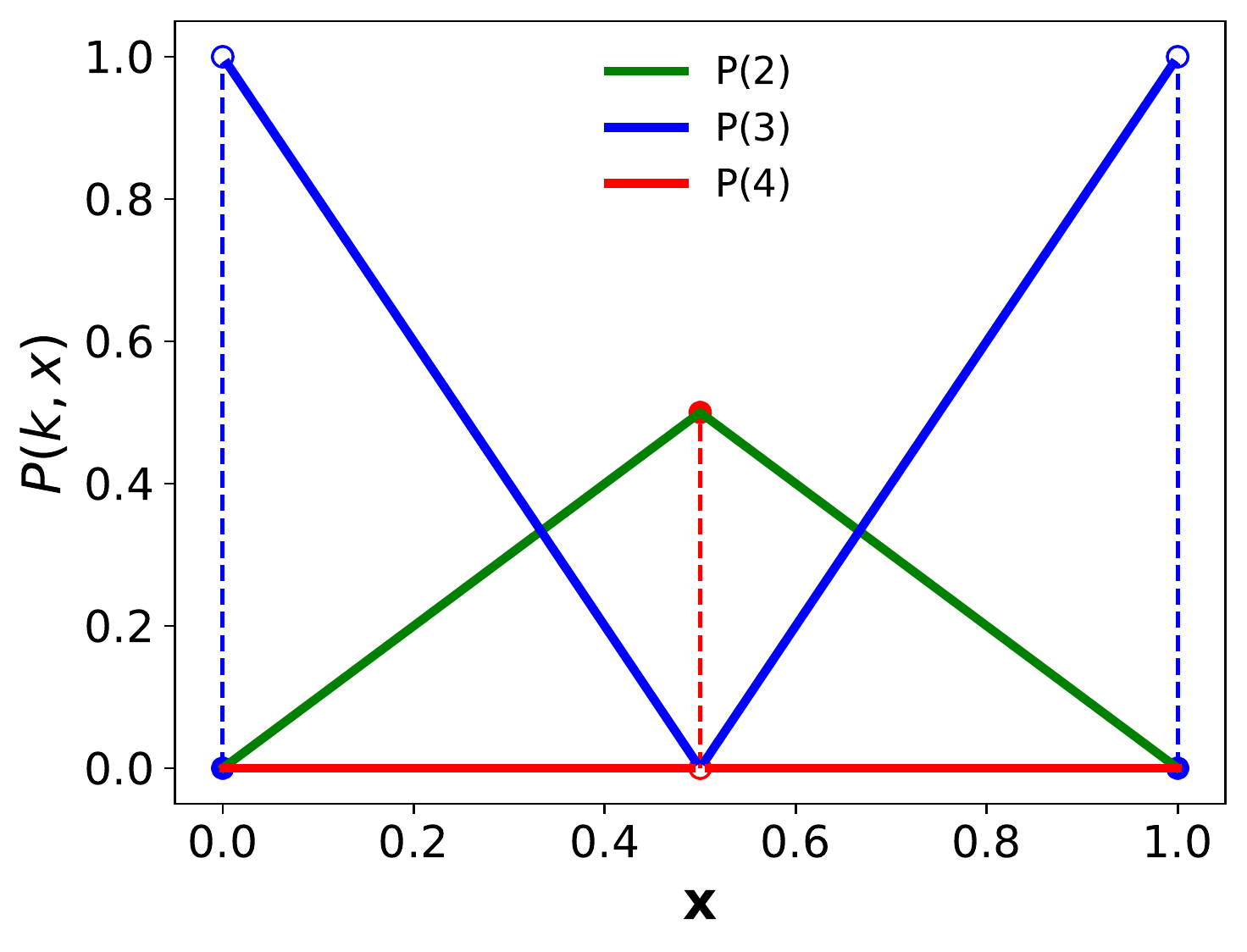}
\caption{Numerical plot of the degree distribution $P(2,x)$ (green), $P(3,x)$ (blue) and $P(4,x)$ (red) of $G_x$ as a function of $x$, for all $x \in {\cal F}_{1000}$, showing a perfect good agreement with Eq. \ref{eq1}. Note that $P(3,x)$ presents two discontinuities in the graphs $G_{0/1}$ and $G_{1/1}$ (blue dots), and likewise $P(4,x)$ presents a discontinuity at $G_{1/2}$.}
\label{fig:P234}
\end{figure}

%\subsubsection{$P(k,x)$ for $k\geq 5$: patterns of holes}\label{P(k,x)5} \
%\subsubsection{\texorpdfstring{$P(k,x)$}{Pkx} for \texorpdfstring{$k\geq 5$}{k5} : patterns of holes}

It is not easy to find in closed form an expression for $P(k,x)$ for a generic $k,x$. Let us consider the case $k=5$. One can prove by induction that:
\begin{equation}
P(5,x)=\left\{
\begin{array}{ll}
3x-1,  & x\in(1/3,1/2) \\
-3x+2,  & x\in(1/2,2/3) \\
0, & \text{otherwise.} \\
\end{array}%
\right.
\label{eqk5}
\end{equation}
Notice that the degree $k=5$ appears with a non-null probability for graphs $G_x$ where $x$ belongs to the subintervals induced by $\cup_{i=1}^3 \ell_i=\mathcal{F}_3$: $(1/3,1/2)$ and $(1/2,2/3)$. It turns out that a generic degree $k$ only `emerges' (i.e. leading to a non-null probability in the degree distribution) after a sufficiently long and specific descent in Farey's binary tree, i.e., for specific subintervals of $x$. In this sense, the following theorem can be stated: \\

\begin{theorem}
\label{Tma:Rep}
For all $\kappa \geq 5, \ P(k=\kappa,x) = 0$ if and only if there is a symbol repetition in the binary sequence of $G_x$ (either $LL$ or $RR$) at the position $\kappa-2$, or equivalently, a $LL$ or $RR$ on the downstream path at the level $\kappa - 2$ of the Haros graph binary tree.
\end{theorem}
\noindent Proof of this theorem is available in Appendix \ref{appendix:B}. This theorem has several implications:

\begin{itemize}
\item  The emergence of a 'new' degree $\kappa$ occurs for the first time at a certain level of the Haros graph binary tree (see Fig. \ref{fig:P5_aparicion} for an illustration). It can then be observed that if the last downstream step was `L' (analogously `R'), a subsequent step `L' will not generate an Haros graph with inner nodes with degrees larger than $\kappa$, as its only effect is to replicate the degree sequences (except in the boundary node). On the contrary, if the subsequent downstream step makes a symbol change ($L\to R$ or $R\to L$), then the resulting Haros graph inherits an inner node with degree ($k>\kappa$).\\
An example will help us illustrate this observation: Consider $G_{1/3}$, generated as $G_{0/1} \oplus G_{1/2}$. Its associated path in the Haros graph binary tree is $LL$ and its degree sequence is $[3,2,2 + 3 = 5]$, i.e., disregarding the boundary node of degree $2+3=5$, only the degrees $k=2$ and $k=3$. A further descent to the left ($LLL$) generates the graph $G_{1/4}$ whose degree sequence is $[3,3,2,2 + 4 = 6]$, where the degree $k = 5$ does not appear, whereas if the descent was towards the right ($LLR$), the resulting graph is $G_{2/5}$ with degree sequence $[3,2,{\bf 5},2,3 + 3 = 6]$, having an inner node with $k=5$.\\ 

In figure \ref{fig:P5_aparicion}, we can see how all Haros graphs whose binary sequence starts with $LLR$ have inner nodes with $k = 5$, and this is not true for those graphs with binary sequence starting with $LLL$. In short, the degree $k = 5$ will not appear if it has not been generated in the descent starting from $G_{1/3}$. This fact is confirmed in Fig. \ref{fig:Fig_xvsP(k)}, which illustrates that degree $k = 5$ has a non-zero probability only for Haros graphs in the interval $[1/3, 2/3]$ (except the case $x = 1/2$), which are those whose path in the tree starts with $LLR$ or $LRL$.

\begin{figure}[t]
\includegraphics[width=1\columnwidth]{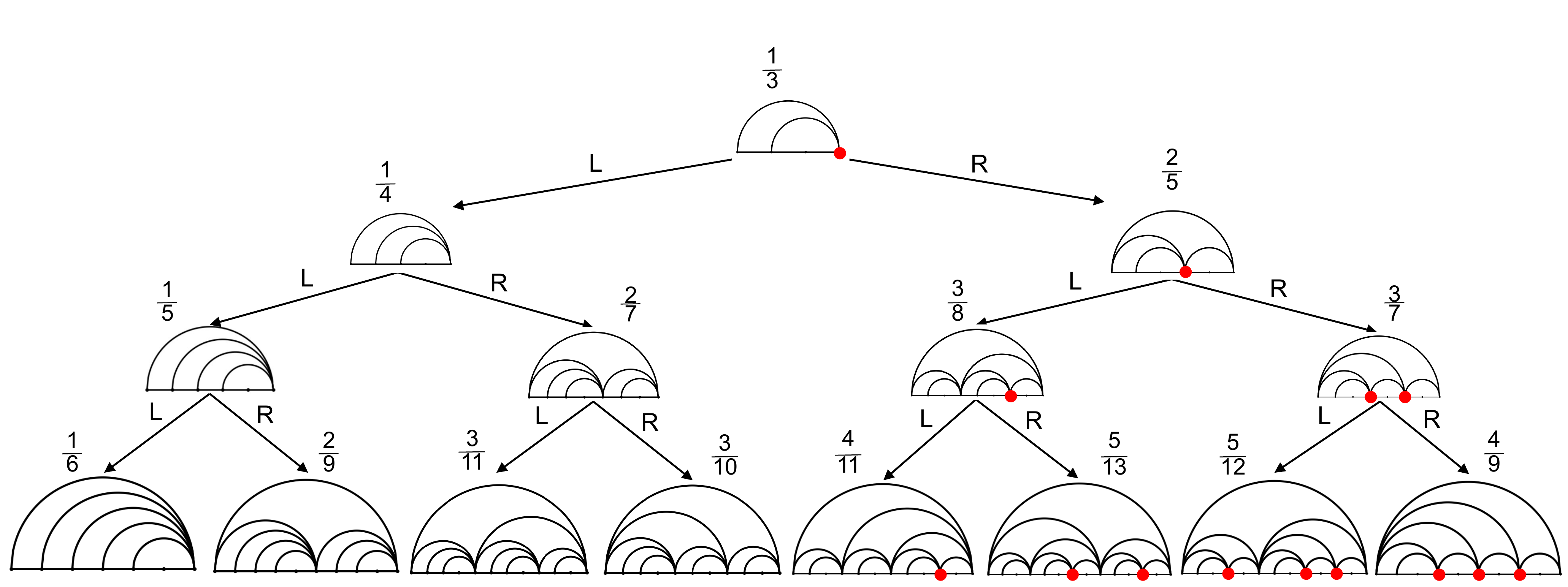}
\caption{Further visual aid of paths in the Haros graph binary tree. The symbolic path in the Haros graph tree of $G_{1/3} = G_{0/1} \oplus G_{1/2}$ is $LL$ and its degree sequence is $[3, 2, 2 + 3 = 5]$, i.e., disregarding the extreme node of eventual connectivity $2 + 3 = 5$, then only degrees $k = 2, 3$ emerge. A further descent to the left $L^3$ generates the graph $G_{1/4}$ whose degree sequences is $[3, 3, 2, 2 + 4 = 6]$, where degree $k  = 5$ does not appear; while if the descent is to the right $LLR$, the graph obtained is $G_{2/5}$ whose degree sequence is $[3, 2, 5, 2, 3 + 3 = 6]$. In the subtree starting at $LLR$ (level $n = 3$), there are Haros graphs with nodes with degree $k = n + 2 = 5$ (red dots). However, in the Haros graphs whose paths start at $LLL$, no nodes with degree $k = 5$ appear. Nodes with degree $k = 5$ will not appear if they have not been generated in the descent starting at $1/3$.}
\label{fig:P5_aparicion}
\end{figure}

\item The hole patterns (the null values in $P(k,x)$) can be completely described by the binary sequence in the Farey binary tree, or equivalently, by the continued fraction of the number $x$. Hence, the construction of Haros graphs is intimately related to the subintervals induced by $\cup_{i=1}^n \ell_n$, see also Figure \ref{fig:Fig_xvsP(k)} for an illustration.
\end{itemize}

One can now use theorem \ref{Tma:Rep} and its implications to discuss the properties of $P(k,x)$ for selected families of irrational numbers. For instance, consider the reciprocal of the Golden number. Its symbolic path in the Farey binary tree consists of an infinite zigzag $\overline{LR}$. Theorem \ref{Tma:Rep} establishes that the degree distribution of its associated Haros graph is necessarily holeless, and, in fact, this is the only irrational number with such property.  Furthermore, two classical results --the Lagrange's theorem and Galois' theorem-- establish that the periodic continued fractions are, precisely, the quadratic irrationals \cite{Hardy}. Therefore, by virtue of theorem  \ref{Tma:Rep}, the Haros graphs of quadratic irrationals will have a degree distribution with a periodic pattern of zeros.

\subsubsection{A conjecture for a closed expression of \texorpdfstring{$P(k,x)$}{Pkx}}
%\subsubsection{A conjecture for a closed expression of $P(k,x)$} \

Observe that we started the preceding section stating that a closed-form expression for a generic $P(k,x)$ is difficult to obtain. However, based on theorem \ref{Tma:Rep} and on the triangular-like dependence on $x$ observed for some values of $k$ (Fig. \ref{fig:Fig_xvsP(k)}) we can now state the following conjecture.\\

\begin{conjecture}
 \textit{Let $\ell_n$ be the set of Farey fractions emerging at level $n \geq 1$ of the Farey binary tree. For instance, $\ell_3 = \left\{ 1/3, 2/3\right\}$ and $\ell_4 = \left\{ 1/4, 2/5, 3/5, 3/4 \right\}$, and, in general, at level $\ell_n$ there are $\textnormal{Card}( \ell_n ) = 2^{n-2}$ new fractions added. Then, for all real number $x$ and connectivity $k\geq 5$ in $G_x$, we have:}
\begin{equation}
P(k,x)=\left\{
\begin{array}{ll}
q_i \cdot x - p_i, \  \textnormal{if}\  x\in\left(\frac{p_i}{q_i},\frac{a}{b}\right), \\ 
-q_{i+1} \cdot x + p_{i+1}, \ \textnormal{if}\ x\in\left(\frac{a}{b}, \frac{p_{i+1}}{q_{i+1}} \right) ,\\  
1/q_i,\ \textnormal{if}\ x =\frac{p_i}{q_i},\\ 
 0,\ \textnormal{otherwise},
\end{array}%
\right.
\label{eqkgeneral}
\end{equation}
\end{conjecture} 

\noindent \textit{with:} $$\frac{p_i}{q_i}, \frac{p_{i+1}}{q_{i+1}} \in \ell_{k-2}, \frac{a}{b} \in \ell_{k-3} \; \textnormal{ and } \; \frac{p_i}{q_i} < \frac{a}{b} <\frac{p_{i+1}}{q_{i+1}}.$$

We have numerically verified the correctness of this conjecture up to $k\leq 60$, and Fig. \ref{fig:Fig_xvsP(k)} illustrates it for $k\leq 8$. Equation \ref{eqkgeneral} implies that the degree distribution is a piecewise-linear function in some subintervals defined by the elements at a certain level of the Farey binary tree. Moreover, the figure also points out several important aspects of the degree distribution:

\begin{itemize}
\item The emergence of different degrees is related to whether $G_x$ is located in a specific subinterval defined by Farey fractions, in agreement with theorem \ref{Tma:Rep} and its implications discussed above. 
    
\item $P(k,x)$ seems to have a self-similar structure. 
    
\item $P(k,x)$ seems to be continuous except on a null-measure set of points (red points), where we have removable discontinuities.
\end{itemize}

The next two subsections address these last two observations.

\begin{figure}[h!]
\includegraphics[width=0.5\columnwidth, trim=1cm 4cm 2cm 5cm,clip=true]{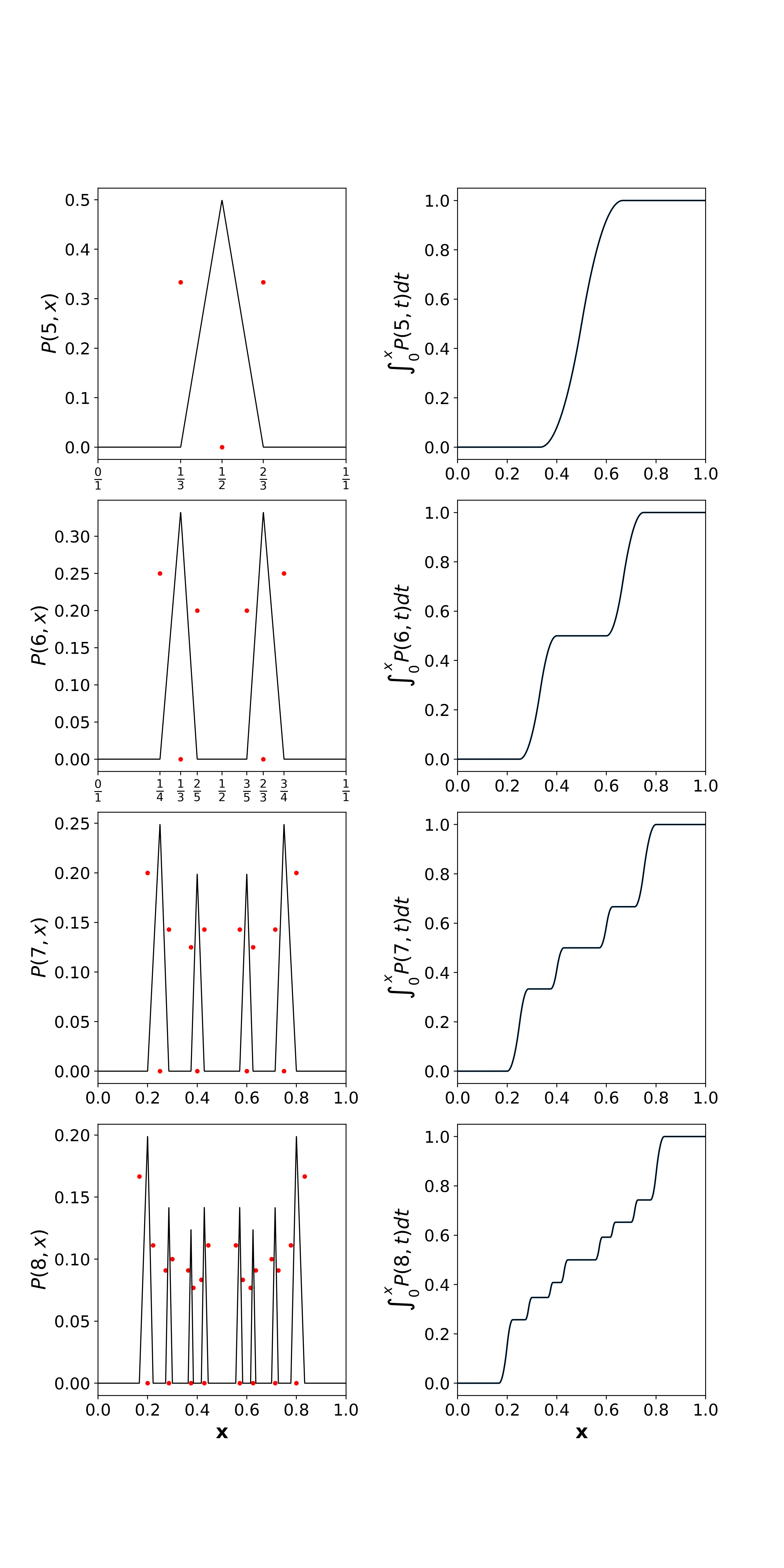}
\caption{(Left panels) Degree distribution $P(k,x)$ as a function of $x$, computed numerically for $k = 5, 6, 7, 8$, for all Haros graphs $G_x$ with $x \in \mathcal{F}_{1000}$. Red points represent removable discontinuities. Solid lines are in perfect agreement with Eq.\ref{eqkgeneral}. (Right panels): Cumulative distributions. As $k$ grows, the cumulative distributions approach a Cantor's staircase function, suggesting self-similarity of $P(k,x)$. This figure supports the study of the behaviour of $P(k,x)$ established in Subsection \ref{Ddistr}. The 'peak' or triangle where the degree distribution is positive is located between two Farey fractions. These fractions are the nodes that form the level $n = \kappa - 2$ of the Farey binary tree (e.g. the nodes $1/3$ and $2/3$ are the level $3$). For example, see the plot of $P(k = 5, x)$. The degree sequence is zero for Haros graphs $G_x$ with $x < 1/3$ or $x > 2/3$, that is, Haros graphs whose symbolic path starts with $LLL$ or $LRR$. In the subinterval $(1/3, 1/2)$, $P(5,x) = 3x + 1$ whereas $P(5,x) = -3x + 2$ in $(1/2, 2/3)$. Moreover, the degree distribution $P(k+1, x)$ can be obtained as a duplicated and scaled version of $P(k,x)$ through the function $F(x)$ defined in Eq. \ref{eq:F}.  }
\label{fig:Fig_xvsP(k)}
\end{figure}

\subsubsection{Self-similarity of \texorpdfstring{$P(k,x)$}{Pkx}}
%\subsubsection{Self-similarity of $P(k,x)$} \

    To dig into the particular scaling properties of $P(k,x)$, it can be observed that the right panel of Fig. \ref{fig:Fig_xvsP(k)} depicts the cumulative degree distributions, showing Cantor's staircase function shapes. This further suggests that $P(k,x)$ has a self-similar structure, where $P(k=\kappa+1, x)$ is found by adequately scaling and shifting the base of the `triangle' in $P(k=\kappa,x)$ (left panel of the same figure).\\ 
    To explore the scaling equations, let us consider the Farey binary tree and define $T_n$ to be the subtree whose asymptotic end nodes densely populate the interval ${\cal I}_n=(1/(n+1),1/n]$ (see Fig.\ref{fig:Arbol_coloreado} for an illustration of $T_2$, $T_3$ and $T_4$). Each $T_n$ has a root node at the rational number $2/(2n+1)$ obtained as the mediant sum of $1/(n+1)  \oplus 1/n$.  In terms of symbolic paths in the Farey binary tree, the roots of $T_n$'s have symbolic paths $L^nR$ (Fig. \ref{fig:Arbol_coloreado} highlights the root nodes of $T_2$, $T_3$ and $T_4$ in green colour). The descendants of such roots have a generic symbolic path $L^nR{\cal P}$, where $\cal P$ is an arbitrary binary sequence. For illustration, in Fig.\ref{fig:Arbol_coloreado} we depict in magenta the symbolic paths $L^nRL$, for $n\geq 2$, associated to rational numbers $3/(3n + 2)$; and we depict in orange the symbolic paths $L^nR^2$, for $n\geq 2$, associated to rational numbers $3/(3n + 1)$ .\\
    In terms of continued fractions, the subintervals ${\cal I}_n$, or equivalently the subtree $T_n$  starting in the root nodes $2/(2n+1)$, are formed by the real numbers which continued fractions have as the first element $a_1 = n$. The families of nodes described above have the following continued fraction expansions: blue nodes $\to [n]$, green nodes $\to [n,2]$, magenta nodes $\to [n,1,2]$, and orange nodes $\to [n,3]$. Now, since $[0,1/2]=\cup_{i=2}^\infty {\cal I}_n$, any arbitrary (real) number $x$ in $[0,1/2]$  belongs to a given subtree $T_k$. Suppose for the sake of argument that $x\in {\cal I}_k$, and consider the function:
 \begin{equation}
     F(x)=\frac{x}{1+x}
     \label{eq:F}
 \end{equation} 
 It is easy to see that the function $F(x)$ acts on $x=[a_1 = k,a_2,a_3,...]$, resulting in $F(x)=[k + 1,a_2,a_3,...]$ \cite{IsolaMaps}, hence, $F(x) \in {\cal I}_{k+1}$ . In general, composing such a function $n>0$ times, $F^{(n)}(x)\in {\cal I}_{k+n}$. Moreover, if a number $x$ is reached by the symbolic path $L^nR{\cal P}$, the number $F(x)$ has the symbolic path $L^{n+1}R{\cal P}$. Hence, the families depicted in Fig.\ref{fig:Arbol_coloreado} can be described as $\left\lbrace x, F(x), F^{(2)}(x),\dots,F^{(n)}(x),\dots \right\rbrace$.With a little abuse of notation, we can then write $F(T_n) = T_{n+1}$, i.e. the subtree $T_n$ is completely mapped into $T_{n+1}$ by the action of $F$ \cite{FareyTree}. \\
 
 \begin{figure}[h!]
\includegraphics[scale=0.5, trim=2cm 3cm 4cm 4cm,clip=true]{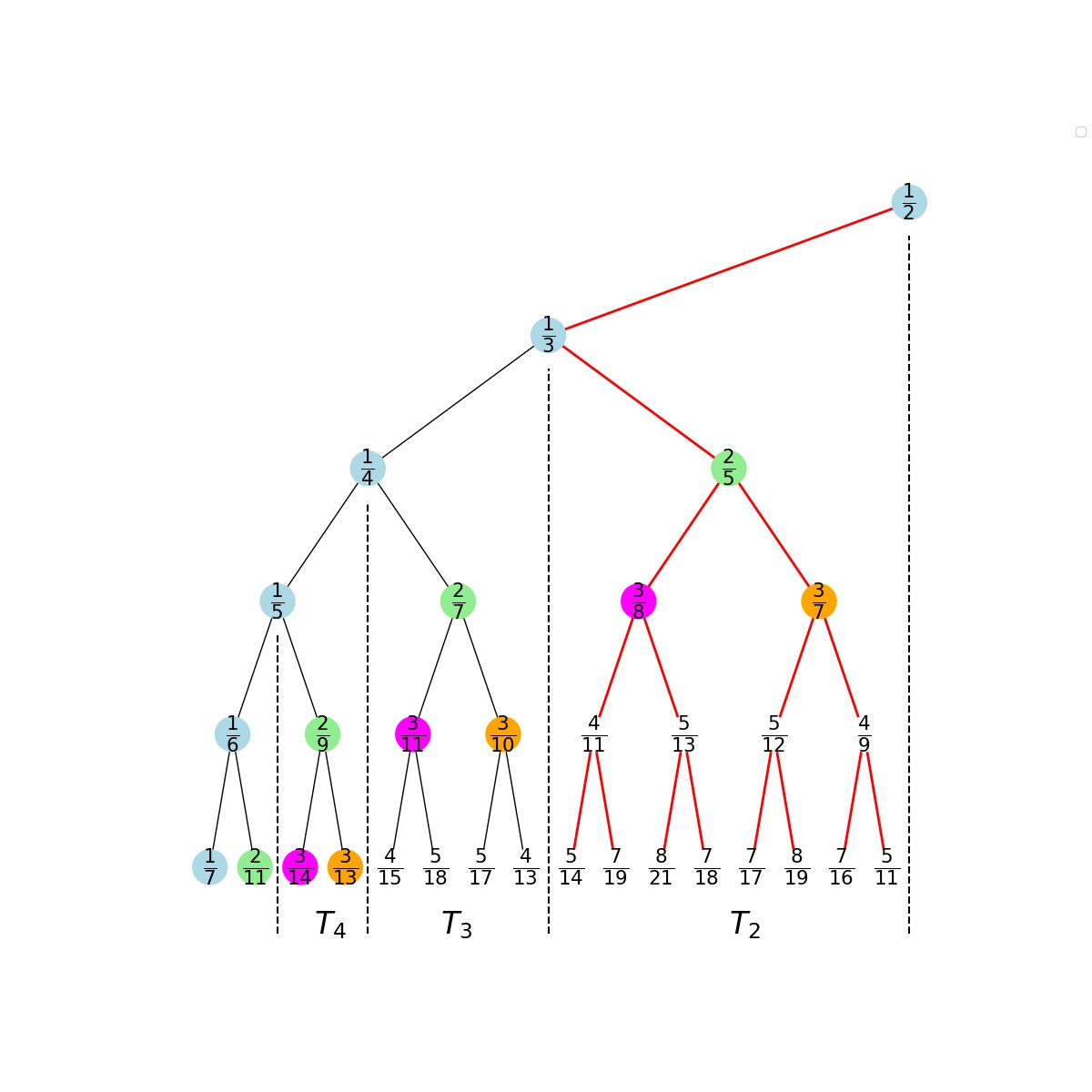}
\caption{Farey binary subtree starting at $1/3$. The coloured nodes are the dots in figures \ref{fig:Fig12_Entropy} and \ref{fig:Fig_entropiaSuelo}. The blue nodes are the extremes of the intervals $\mathcal{I}_n$. The green nodes are the roots of the subtrees restricted to two consecutive blue nodes. The magenta and orange dots form different rational families: in magenta the symbolic paths $L^nRL$ and in orange the paths $L^nR^2$. Equivalently, in continued fractions, the magenta dots are $[n,1,2]$ and the orange dots are $[n,3]$.  }%The sub-tree between $\mathcal{I}_2$ are presented with red lines in correspondence with the red-coloured pattern in Fig \ref{fig:Fig_racniveles}.}
\label{fig:Arbol_coloreado}
\end{figure}

 To conclude, let us translate this discussion into Haros graphs. Consider the Haros graph $G_x$ associated with $x \in T_n$. It can be proved that if $G_x$
has $\alpha$ nodes of degree $k$, then the Haros graph $G_{F(x)}$, associated with $F(x) \in T_{n+1}$, has $\alpha$ nodes of connectivity $k + 1$. See Appendix \ref{appendix:C} for a complete proof of this fact. For instance, the Haros graph $G_{2/5}$, whose sequential connectivity is $[3,2,5,2,6]$, has a single node with connectivity $k = 5$, whereas the Haros graph $G_{F(2/5)} = G_{2/7}$ has a single node with connectivity $k + 1 = 6$  (see also Fig. \ref{fig:Familias_23/n} for a visual proof for the three first elements of the families depicted in green, magenta and orange). Altogether, these facts provide the following scaling equation:

\begin{equation}
P(k,x) = (1 + x)\cdot P\left(k+1, \frac{x}{1 + x}\right).  
\label{Pscaling1}
\end{equation}
In general, $\forall m \geq 1$:

\begin{equation}
P(k,x) = (1 + mx)\cdot P\left(k+m, \frac{x}{1 + mx}\right).  
\label{Pscaling}
\end{equation}

\begin{figure}[tbh!]
\includegraphics[width=0.7\columnwidth, trim=0cm 2cm 0cm 0cm,clip=true]{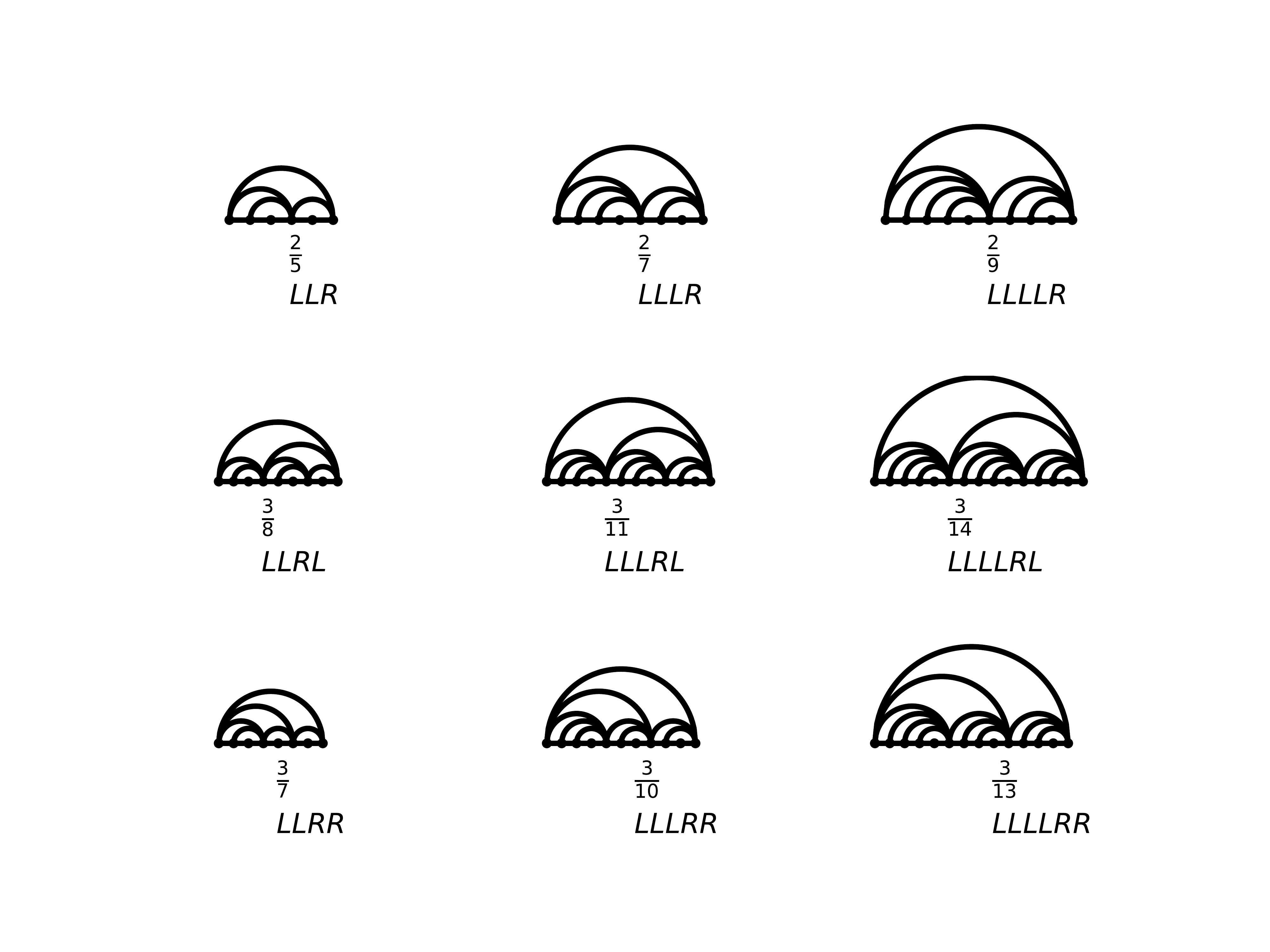}
\caption{First elements of the Haros graph families $G_{2/(2n+1)}$, $G_{3/(3n+2)}$ and $G_{3/(3n+1)}$, $n=2,3,4$. The graphs of each family are reached by symbolic paths of the shape $L^n R \mathcal{P}$. The connectivity structure is similar in each family: the number of nodes with degrees $k \geq 5$ is preserved, even though the degrees change. For example, the family $G_{2/(2n+1)}$ has two degrees $k = n+3$ and $k = n+4$.}
\label{fig:Familias_23/n}
\end{figure}

{The scaling property reported in Eq.\ref{Pscaling} will be exploited later when we explore the structure of the degree distribution entropy.}

%\sout{Figure \ref{fig:P5_aparicion} illustrates the process of appearance of connectivity $k = 5$. It can be proved by induction that  $P(k = 5) = 3x - 1$ at  $x \in (1/3, 1/2)$. However, we can see in Fig \ref{fig:Fig_xvsP(k)} how $P_{1/3}(5) = 1/3$ (the extreme node). Haros graph $G_{1/3}$ is too the limit $n \to \infty$ of the graph succession given by the paths $L^{3}R^n$ where connectivity $k = 5$ not appear. Therefore,  we have that the left limit $\lim_{x\to 1/3^{-}} P_{x}(5) = 0$. Analogously, following the paths $L^2RL^n$ we have that $\lim_{x\to 1/3^{+}} P_{x}(5) = 0$. Then,  $P_{1/3}( 5)$ presents a removable discontinuity. The same happens at $x = 1/2$ and $x = 2/3$. 
%A similar reasoning demonstrates the continuity  of $P_{x}(k) = \alpha$. In that case, taking the graph sequence starting at $G_{p/q}$ node and descending $LR^n$ or $RL^n$, we have $\lim_{x\to x^{-}} P_{x}(k) = \lim_{x\to x^{+}} P_{x}(k) = \alpha$. In general, every $P_{x}(k)$ has removable discontinuities at the rational numbers located at the Farey level where the $k$ connectivity appears and at the previous level, and $P_{x}(k)$ is continuous over $x$  in another case. \textcolor{red}{HASTA AQUI  LA PARTE DIFUSA}} \\

\subsubsection{On the continuity of \texorpdfstring{$P(k,x)$}{Pkx}}
%\subsubsection{On the continuity of $P(k,x)$} \

To conclude our study of the behaviour of $P(k,x)$, we focus on the continuity over the variable $x$, when $k$ is fixed. Consider again Figure \ref{fig:P5_aparicion}, where red dots illustrate how degree $k = 5$ `emerges' along the tree. Although $P(5,x)$ follows Eq.\ref{eqk5}, nothing is yet said about the boundaries $x=1/3,1/2,2/3$. First, notice in Fig. \ref{fig:Fig_xvsP(k)} that at $x=1/3$, we have $P(5,1/3) = 1/3$. Now, $G_{x\to 1/3-}$ is reached asymptotically through the paths $L^{3}R^n$, where the degree $k = 5$ does not appear by construction, and thus $\lim_{x\to 1/3^{-}} P(5,x) = 0$. Analogously, following the paths $L^2RL^n$ we have $\lim_{x\to 1/3^{+}} P(5,x) = 0$. We conclude that $P(5,1/3)$ constitutes a removable discontinuity. A similar phenomenon occurs at $x = 1/2$ and $x = 2/3$. 
A similar reasoning demonstrates the continuity of the degree distribution close to discontinuities in $x=\xi$: taking the sequence starting at $G_{p/q}$ and descending $LR^n$ or $RL^n$, we have $\lim_{x\to \xi^{-}} P(k,\xi) = \lim_{x\to \xi^{+}} P(k,\xi)$. In general, $P(k,x)$ has removable discontinuities at the rational numbers located at the level of the Farey binary tree where degree $k$ first appears ($\ell_{k-3}$) and at the level immediately above ($\ell_{k-2}$), and is continuous otherwise.

\subsection{Arithmetic and geometric mean degree: Khinchin constant}
Once we have established the properties of $P(k,x)$, in this section we further study two additional aspects of the degree sequence: its arithmetic mean degree ${\bar k}(x)$ and the geometric mean degree ${\bar k}_g(x)$. For a $G_x$ where $k_i$ correspond to the connectivity of the node $i$ we have:

\begin{eqnarray*}
\bar k(x) = \langle k\rangle=\sum_{k=2}^\infty kP(k,x)=\sum_{i=1}^q k_i/q \end{eqnarray*}
\begin{eqnarray*}
\bar k_g(x) = \exp(\langle \ln(k)\rangle)=\exp\bigg(\sum_{k=2}^\infty \ln k P(k,x)\bigg)=\bigg(\prod_{i=1}^q k_i\bigg)^{1/q}
\end{eqnarray*}

In his introduction to continued fractions, A. Ya. Khinchin \cite{Khinchin} proves that, for almost all real numbers, the set of coefficients $\{a_i\}$ of the continued fraction expansion of $x = [a_1, a_2, ... ]$ have a finite geometric mean independent of the value of $x$. Such a constant was coined as the Khinchin constant:

$$K_0 = \prod_{r=1}^{\infty} \left( 1 + \frac{1}{r(r+2)} \right)^{\log_{2} r} = 2.68545...$$. 

While this is true for almost all numbers, not much is known about specific numbers fulfilling the theorem, other than the fact that both rational and quadratic irrational numbers (with infinite periodic continued fractions) are sets of null measure that do not verify the result. On the other hand, it is also well known that the arithmetic mean of a number's continued fraction expansion diverges for almost all numbers in the unit interval.\\

\noindent  Here we study whether the structure of Haros graphs $G_x$ inherits such universality by interpreting the graph degree sequence $(k_1,k_2,...,k_q)$ as the graph analogue of the coefficients of the continued fraction expansion $[a_1, a_2, ... ]$ of $x$.
\begin{figure}[htb]
\includegraphics[width=0.8\columnwidth]{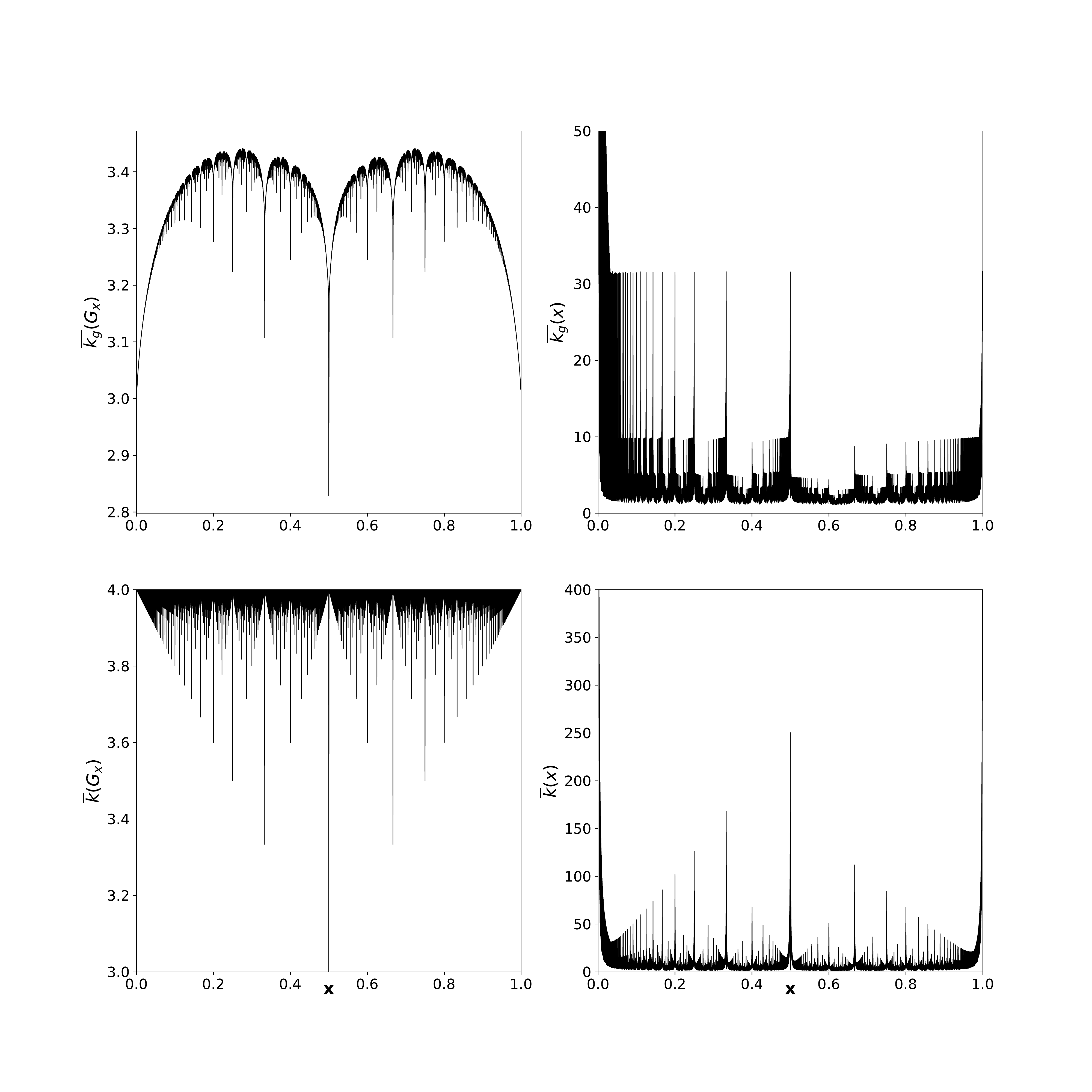}
\caption{(Upper panels) Geometric mean connectivity function $\overline{k}_{g}(G_x)$ computed for all Haros graphs $G_x$ with $x \in {\cal F}_{1000}$ (left) and geometric mean $\overline{k}_g(x)$ for continued fraction terms $x = [a_1, ... a_n], \forall x \in {\cal F}_{1000} $ (right). The upper right panel is unbounded when $x = 1/n \to 0$. (Bottom panels) Arithmetic mean connectivity of $G_x$ (left) and $x = [a_1, ... a_n]$ (right), for the same values of $x$. The asymmetric behaviour of the right panels occurs because the continued fraction expansions are different for elements $x$ and $1-x$.}
\label{fig:FigMeans}
\end{figure}
%\noindent We consider two different types of average for the elements of the degree sequence of $G_{p/q}$: the arithmetic mean $\bar k = \sum_{i=1}^q k_i/q$ and the geometric mean $\bar k_g = (\prod_{i=1}^q k_i)^{1/q}$. 
In the top panels of Fig. \ref{fig:FigMeans} we plot the geometric mean of the continued fraction expansion of $x$ for all $x\in {\cal F}_{1000}$ (right panel) and $\overline{k}_g(G_x)$, the geometric mean of the degree sequence Haros graph $G_x$, for the same set (left panel).
Notice that the former fluctuates around $K_0$, which is an indication that, while in rigour ${\cal F}_{1000}$ does not contain irrationals (and thus $K_0$ is not supposed to show up), the geometric mean of approximants to rational numbers seems to converge to $K_0$. On the other hand, the structure of such convergence is not clear, and there are no other noticeable patterns. The results for the Haros graphs reveal in turn a self-affine shape, similar to a Takagi curve \cite{Takagi}. This curve does not seem to concentrate its measure close to any particular constant, but shows a richer structure.\\
In the bottom panels of Fig. \ref{fig:FigMeans}, we then plot the arithmetic mean of the continued fraction expansion of $x$ for all $x\in {\cal F}_{1000}$ (right panel), and $\overline{k}(x)$, the arithmetic mean of the degree sequence of the Haros graph $G_x$, for the same set (left panel). The former does not show the exact mirror simmetry of the latter as the continued fraction expansion of $x$ and $1-x$ are different. Furthermore, the arithmetic mean of continued fraction expansion is known to diverge for almost all real numbers, however in the case of Haros graphs we find that their arithmetic mean accumulates at $\bar k=4$, with self-affine fluctuations. To make sense of this finding, we resort to a result in the HVG theory \cite{FeigenbaumGraphs, AnalyticalFeigenbaum} that states that an aperiodic HVG with $q$ nodes --and thus an Haros graph $G_{p/q}$-- has an arithmetic mean degree $\bar k = 4-2/q$. Since for irrationals $q\to \infty$, we then have
\begin{equation}
\overline{k}(x)=\left\{
\begin{array}{ll}
4 -\frac{2}{q},  & \text{if $x=p/q$ a reduced fraction}\\
4,               & \text{if $x$ irrational} \\
\end{array}%
\right.
\label{eq:arith}
\end{equation}
%Haros graphs are a subset of the Horizontal Visibility Graphs (HVG), in turn, are a subset of the Natural Visibility Graphs (NVG) \cite{FromTime}. Also a recently proved theorem \cite{canonical} guarantees that the sequential connectivity of a HVG totally determines its adjacency matrix. Hence, we can ask whether the geometric mean of these connectivity is like the Khinchin's constant. 
i.e. a Thomae's function that reflects the fractal structure of the partitions of the interval $[0, 1]$ that generate the successive Farey sequences \cite{Trifonov}.\\

To conclude, we found that the universality of Khinchin's constant is not found in Haros graphs --a self-affine structure with no clear concentration of the measure is found instead, pointing to a richer structure. Such universality is, however, found in the arithmetic mean, which concentrates at $\bar k_g=4$, a quantity that allows us to discriminate between rational and irrational numbers but fails to further distinguish e.g. between quadratic and non-quadratic irrationals as $K_0$ does.
In the next section, we proceed to construct yet another property of the degree sequence which we will show to possess a richer structure capable of classifying different families of rationals and irrationals.

%Note that the arithmetic mean connectivity does allow us to distinguish between rational and irrational numbers. However, this mean has the same value for all irrational numbers. Therefore, the role of the Khinchin's constant in our model is played by the value $\overline{k} = 4$, the mean connectivity for all irrational numbers.
%\cite{Trifonov} 

\section{Entropy}

It is suggestive to explore the topological properties of the Haros graphs in relation to the real numbers to which they correspond. 
The previous section suggests that the structure of $G_x$ --in particular, its degree sequence-- is inheriting the properties of the number $x$ to which it corresponds, i.e., the properties associated with the path followed in the Farey binary tree. To further explore in more detail such structure, we propose to interpret an Haros graph's $G_x$ degree sequence as a `signal', and study its informational content by computing the entropy $S(x)$:
\begin{equation}
S(x)=-\sum_{k \geq 2} P(k,x)\log P(k,x).
\label{def:Def3}
\end{equation}
In information-theoretic terms, $S(x)$ is the block-1 approximant to the Shannon entropy of the signal \cite{entropy}. In graph-theoretic terms, it quantifies the heterogeneity of the graph degree distribution, that is, it measures the amount of disorder contained in the wiring architecture of $G_x$. Our contention is that the rich and intertwined structure of rational and irrational numbers can be unveiled by studying the structure of $S(x)$.\\

\noindent To give a flavour of the shape of $S(x)$, in Fig. \ref{fig:Fig12_Entropy} we show the results from a numerical computation for all elements of the Farey sequence  ${\cal F}_{1000}$, that includes all irreducible fractions $p/q$ with $q \geq 1000$. As $P(k,x) = P(k,1-x)$, the curve is symmetric around $x=1/2$ and $S(x) = S(1-x)$.
At first sight, the entropy curve seems self-affine. For instance, $S(x)$ over $(1/4, 1/3]$ seems to be a rescaled copy of the function at $(1/3, 1/2]$ and, in general, every interval $\mathcal{I}_n = \left(\frac{1}{n+1}, \frac{1}{n}\right]$, $n=2,3,\dots$ shows a similar shape. A numerical estimation of its box-counting dimension \cite{IGTorre} suggests $D_0 \approx 1.43$. In the following subsections, we dig deeper into the shape of $S(x)$ and aim to relate the structure with the properties of $x$.

\begin{figure}[tbh!]
\includegraphics[width=0.5\columnwidth, trim=0cm 1cm 0cm 2cm,clip=true]{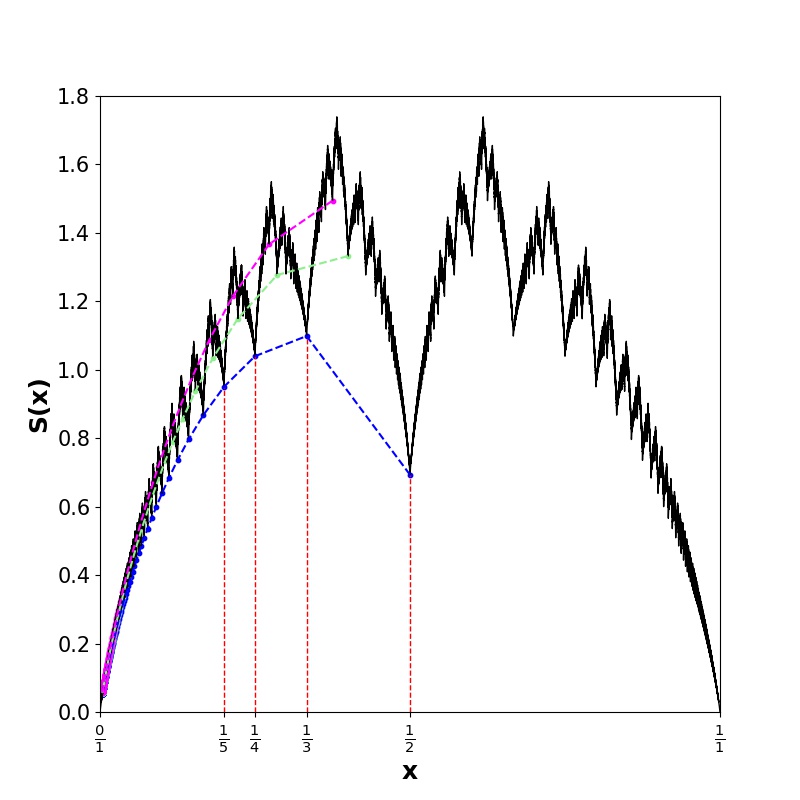}
\caption{Entropy function $S(x)$, computed for all Haros graphs $G_x$ with $x \in {\cal F}_{1000}$. Blue dots highlight the values $S(1/n)$ and define the intervals $\mathcal{I}_n = \left(\frac{1}{n+1}, \frac{1}{n}\right]$ for $n\geq 2$ (vertical red lines). $S(x)$ appears to have self-affine behaviour in every interval $\mathcal{I}_n$. Two examples of different families of rational numbers are represented in which each element appears in a single subinterval $\mathcal{I}_n$, for $n \geq 2$: green dots are the fractions $\frac{2}{2n + 1}$, and magenta dots are the fractions $\frac{3}{3n + 2}$.}
\label{fig:Fig12_Entropy}
\end{figure}

%Let $G_x$ be the Haros graph associated to real number $x$ and let $P(x,k)$ be its degree distribution \cite{entropy} . Then its graph entropy is defined as 

%Note that other possibilities includes the consideration of local clustering $c$ instead of degree for the definition of graph's entropy. However clustering is computationally harder to obtain. 

%\noindent Under these definition, some questions naturally arise: which is the most Farey-entropic number? Is this entropy function continue or smooth in the interval $[0,1]$? Can we relate such function to other number-theoretical properties of the real numbers? In what follows we explore these questions.
\subsection{Continuity of \texorpdfstring{$S(x)$}{Sx}}
%\subsection{Continuity of $S(x)$} \ 

Visually, it is not clear whether $S(x)$ is continuous on $[0,1]$. Observe that the arithmetic mean of the degree sequence is a Thomae's function (Eq. \ref{eq:arith}) and thus discontinuous at all rationals. This, together with the fact that $P(k,x)$ shows discontinuities, might give the false impression that $S(x)$ might also be discontinuous at all rational numbers. However, a subtler analysis shows that this is not the case. The main reason is based on the fact that the discontinuities found in $P(k,x)$ are systematically compensated for by the values of $P(k + 1,x)$. To showcase this, consider again Fig. \ref{fig:P234}: for $k=4$ and close to $x=1/2$ we have a discontinuity, where $P(4, 1/2) = 1/2$ and $\lim_{x\to 1/2^{-}} P(4,x) = \lim_{x\to 1/2^{+}} P(4,x) = 0$. This discontinuity is indeed `balanced-out' with the discontinuity at $k=5$ and close to $x=1/2$, where we have $P(5, 1/2) = 0$ and $\lim_{x\to 1/2^{-}} P(5,x) = \lim_{x\to 1/2^{+}} P(5,x) = 1/2$ (see Fig. \ref{fig:Fig_xvsP(k)}). Accordingly, $P(4,x)+P(5,x)$ is continuous at $x=1/2$, and so is $P(4,x)\log P(4,x) + P(5,x)\log P(4,5)$ (where we use the convention $0\log 0 =0$).
Similarly, close to any $x$ that shows a removable discontinuity in $P(k,x)$, such a discontinuity is exactly balanced with one at $P(k+1,x)$. When summing over $k$ in the computation of $S(x)$ these discontinuities are systematically removed, and the resulting function $S(x)$ is continuous.

\subsection{Disentangling Entropy: from rational Haros graphs being families of local minima to a generalised de Rham curve}
\label{disentan}

In this subsection, we shall produce two results: first, by exploring families of Haros graphs emerging naturally in the entropy function, we unveil a hierarchical structure of local minima associated with a partition of rational numbers into different families. Leveraging on some of the patterns that will emerge in the first part, the second objective of this subsection is to obtain a functional equation for the entropy function, which we will show takes the form of a generalised de Rham curve, i.e., certifying that $S(x)$ is indeed a self-affine function, continuous but not differentiable at any point.\\

\noindent Let us start by considering rationals of the form $x=1/n, \ n\in\mathbb{N}^+$. In Fig.\ref{fig:Familia_1n} we show the first Haros graphs $G_{1/n}$ for $n=2,3,4$, and their respective symbolic paths in the Farey binary tree.
Their degree distributions $P(k, 1/n)$ have the following expression for $n\geq 2$:
\begin{equation*}
P(k, 1/n)=\left\{
\begin{array}{ll}
1/n  & ,  k=2 \\
1-2/n& ,  k=3 \\
1/n  & ,  k=n + 2\\
0 &  , \text{otherwise}%
\end{array}%
\right.
\end{equation*} 

\begin{figure}[tbh]
\includegraphics[width=0.5\columnwidth, trim=0cm 2cm 0cm 0cm,clip=true]{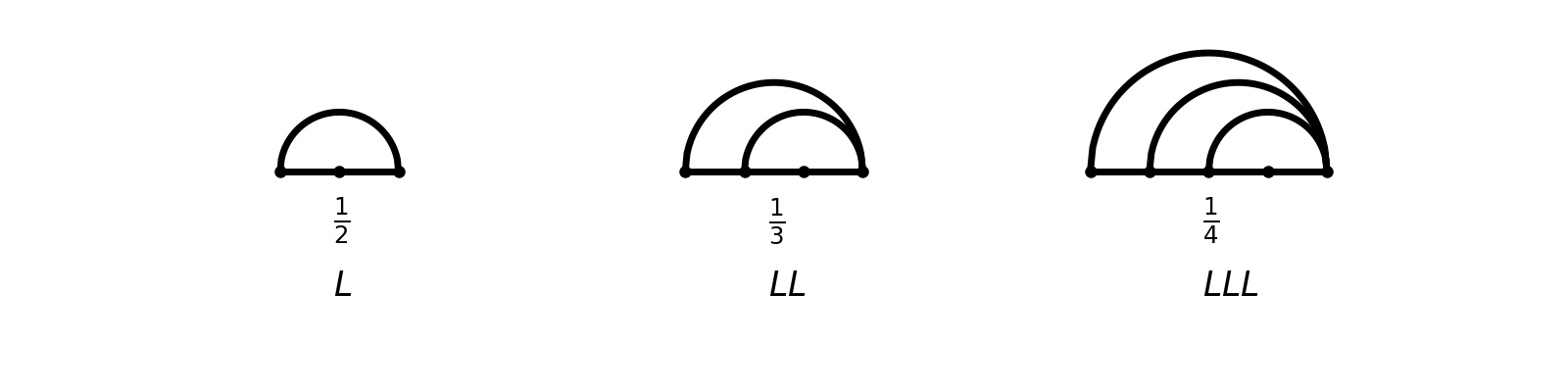}
\caption{The first three elements of the Haros graph family $G_{1/n}$. These Haros graphs, represented in blue in figures \ref{fig:Fig12_Entropy}, \ref{fig:Fig_entropiaSuelo} and \ref{fig:Arbol_coloreado}, are reached by the symbolic paths $L^{n-1}$. Each element of this family maintains the connectivity structure $[3,...,3, 2, n+2]$, where the number of nodes with connectivity $k =3$ increases as $n$ grows.}
\label{fig:Familia_1n}
\end{figure}
Accordingly, their graph entropy values are: $S(1/n) = - 2/n\cdot \log( 1/n ) -  (1-2/n)\cdot \log(1-2/n)$, and these values appear as local minima (blue dots in Fig.\ref{fig:Fig12_Entropy}). Taking advantage of the analytical expression, we can then define a reduced graph entropy $H(x)$ as follows:
\begin{equation}
H(x) \equiv\left\{
\begin{array}{ll}
S(x) + 2x\cdot \log(x) + (1-2x)\cdot \log(1-2x)  & , \; x \in [0,1/2] \\
S(x) + 2(1-x)\cdot \log(1-x) + (2x-1)\cdot \log(2x-1)  & , \; x \in [1/2,1]
\label{def:EntH}
\end{array}%
\right.
\end{equation} 
The reduced entropy preserves the original symmetry $H(x)=H(1-x)$, and is identically null for all rationals of the form $x=1/n, n>1$: see Fig.\ref{fig:Fig_entropiaSuelo} for a plot of $H(x)$, where $H(1/n)$ is highlighted as a dashed blue line connecting all rationals of the form $1/n$. In the same plot, other sets of minima of the reduced graph entropy seem to emerge naturally and connected by straight lines of different slopes, e.g. $G_{2/(2n+1)}$ (green dots connected by a green dashed line), $G_{3/(3n+2)}$ (magenta dots connected by a magenta dashed line), or $G_{3/(3n+1)}$ (orange dots connected by an orange dashed line). The first elements of these respective sets of Haros graphs, along with their respective symbolic sequence, are depicted in Fig. \ref{fig:Familias_23/n}. We now show that this is indeed the case and that such families are indeed aligned local minima of $H(x)$.  It can indeed be proved that, in the reduced entropy space, every such family can be interpolated by a straight line $H(x) = ax$ with a different slope $a$ (a proof is presented in Appendix \ref{appendix:D}).\\

\begin{figure}[htb]
\includegraphics[width=0.5\columnwidth, trim=0cm 0cm 0cm 2cm,clip=true]{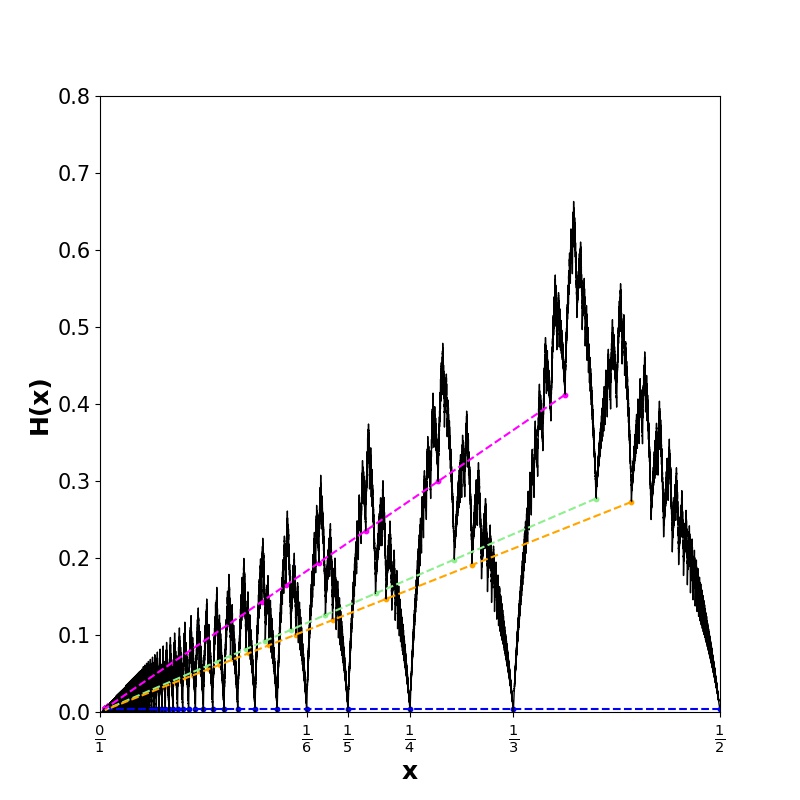}
\caption{Reduced entropy function $H(x) = S(x) + 2x\cdot \log(x) + (1-2x)\cdot \log(1-2x) $ for $x \geq 1/2$. Different families of rational numbers are highlighted: the blue dots are the fractions $1/n$ reached by the paths $L^{n-1}$; blue dashed line has a null slope. The green dots are the fractions $\frac{2}{2n + 1}$ reached by paths $L^{n}R$; the slope of the green line is $\log(2)$. Magenta dots are the fractions $\frac{3}{3n + 2}$ reached by paths $L^{n}RL$; the slope of the magenta line is $\log(3)$. The orange dots are the fractions $\frac{3}{3n + 1}$ reached by paths $L^{n}R^2$; the slope of the orange line is $\log(3) - \frac{2}{3}\cdot \log(2)$.}
\label{fig:Fig_entropiaSuelo}
\end{figure}

Furthermore, if we now define a reduced entropy density $H(x)/x$, the numbers families are now classified in constant levels of reduced entropy density, since $H(x)/x = H(F^{(m)}(x))/F^{(m)}(x)$, with $x\in {\cal I}_2$ and $F^{(m)}(x) \in {\cal I}_{m+2}$. Accordingly and with Eq. \ref{Pscaling}, we have that the reduced entropy fulfils a scaling equation: 
\begin{equation}
H(x) = (1 + mx)\cdot H\left(\frac{x}{1 + mx} \right).
\label{Hscaling}
\end{equation}
Note however that this scaling holds for $x<1/2$, and the scaling is indeed different if $x > 1/2$ as shown in Appendix \ref{appendix:C}. 
If we now apply the change of variable $z = \frac{1}{x} - 2$, in Eq. \ref{Hscaling}, the interval $[0,1/2]$ is deformed into $[0, \infty)$, and every rational number $x= 1/n$ is transformed into a natural number $z = n - 2$ for $n\geq 2$. Thus, the intervals $\mathcal{I}_n$ deform into intervals of equal size $[n-2, n -1)$, $H(z)/z$ becomes periodic (see Fig. \ref{fig:Fig_racniveles}) and we can write for $z \in [0, 1)$ and $m=0, 1, 2,...$: 

\begin{equation}
(z + 2) \cdot H \left( \frac{1}{z + 2} \right) = (z + m + 2) \cdot H \left( \frac{1}{z + m + 2} \right).
\end{equation}

Moreover, from Fig. \ref{fig:Eq_entropia}, we visually observe that $(z+2)H(1/(z+2)$ is exactly the graph entropy defined in Eq. \ref{def:Def3} for $z \in [1/2, 1]:$  

\begin{equation}
S(z) = (z + 2) \cdot H\left(\frac{1}{z + 2} \right).
\label{SHscaling}
\end{equation}
or, equivalently, by the definition of reduced entropy \ref{def:EntH}: 

\begin{equation}
S(z) = (z + 2) \cdot  S\left(\frac{1}{z + 2} \right)+z \log z -(z + 2)\log(z + 2),  \text{for } z \in [1/2, 1].
\end{equation}
A rigorous proof of this scaling equation is given in the proof of Theorem \ref{Tma12} in Appendix \ref{appendix:C}.\\

\begin{figure}[h!]
\includegraphics[scale=0.6]{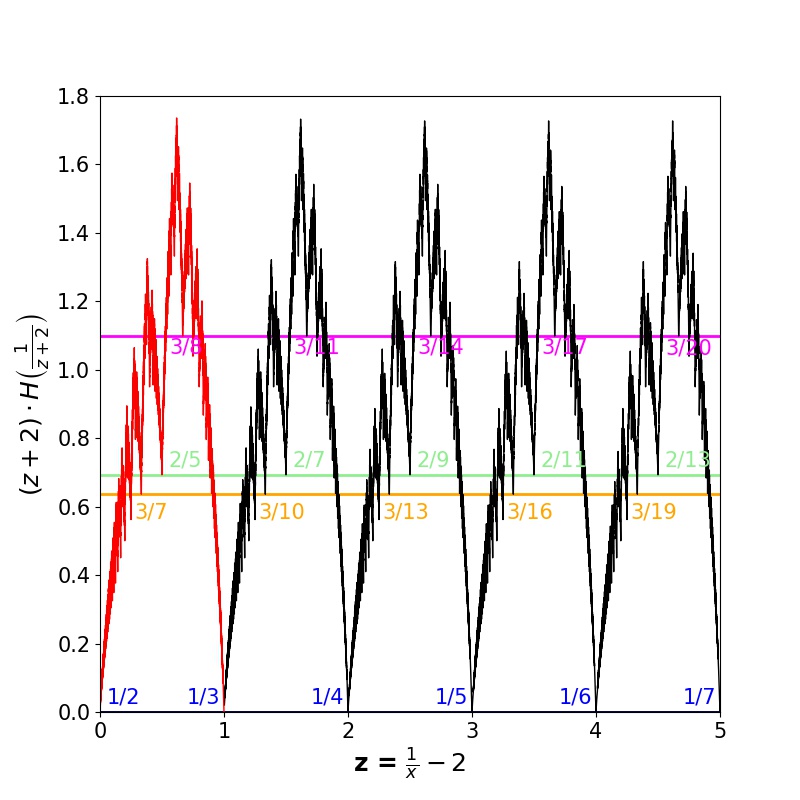}
\caption{Plot of $(z+2)\cdot H(1/(z+2))$. The change of variable $z = 1/x - 2$ transforms the interval $[0, 1/2]$ into $[0, \infty)$. Furthermore, the fractions $1/n$ are equispaced and we have a periodic pattern in the intervals $[n, n + 1]$. The interval $\mathcal{I}_2$ is highlighted in red in correspondence with the red-coloured subtree in Fig. \ref{fig:Arbol_coloreado}.}
\label{fig:Fig_racniveles}
\end{figure}

This equation has the form of a generalised Rham curve \cite{Takagi}, where the graph of $S\left(\frac{1}{z + 2} \right)$ in $[1/3,2/5]$ is a self-affine copy of $S(z)$ in $[1/2,1$] scaled by a factor $(z + 2)$ and displaced $z \log z -(z + 2)\log(z + 2)$. In Fig. \ref{fig:Eq_entropia}, we numerically illustrate this self-affinity. Observe that in the last step, the interval $[1/3, 2/5]$ is the interval $[1/2, 1]$ transformed by $1/(z + 2)$. As a family of de Rham curves, $S(x)$ is indeed a continuous function with no derivative at any point.\\

To conclude this section, we observe how Figure \ref{fig:Fig12_Entropy} supports visually the idea that the rational numbers form envelope curves under the graph entropy function. We sketch a proof of local minima as follows: Clearly we have that the global minima are reached in $x = 0,1$. The local maxima in a given interval will be reached at a point obtained as a finite composition of affine transformations $\frac{1}{z + 2}$, where $z$ are rational numbers, so the local minima will be rational numbers.

\begin{figure}[htbp]
\includegraphics[width=0.9\columnwidth]{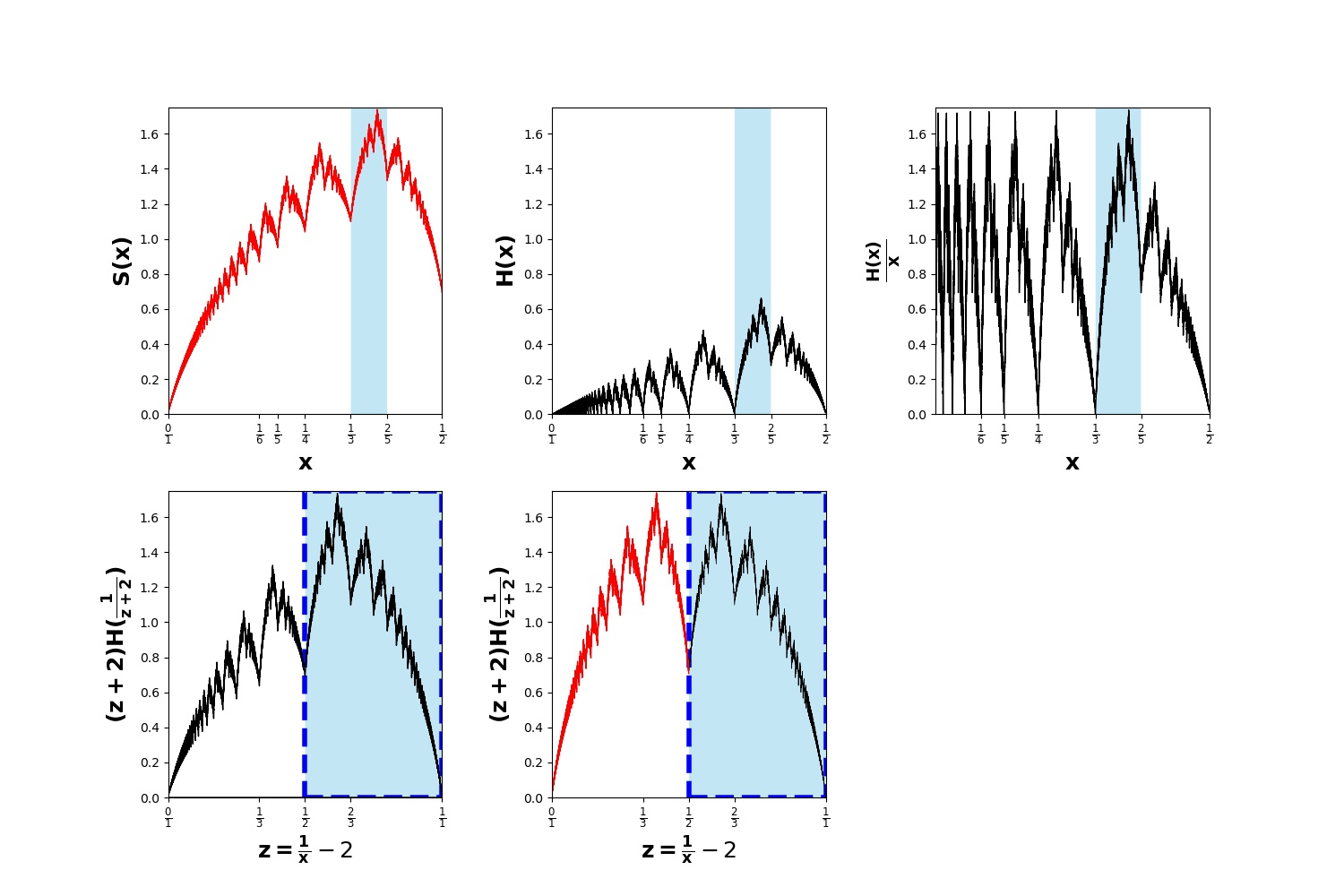}
\caption{Compositions of functions to visually describe the self-affine structure of $S(x)$. The blue shaded zone is the interval $[\frac{1}{3}, \frac{2}{5}]$. (Upper, left panel) Graph entropy $S(x)$.  (Upper, middle panel) The transformation $H(x) = S(x) + 2x\cdot \log(x) + (1-2x)\cdot \log(1-2x)$ removes the entropy of rational numbers ${1/n}$. (Upper, right panel) Transformation $H(x)/x$. (Bottom, left panel) Change of variable $z = 1/x - 2$. (Bottom, middle panel) The resulting function matches the original function $S(x)$ over $[1/2, 1]$.}
\label{fig:Eq_entropia}
\end{figure}

\subsection{Entropy maxima: from the Golden mean to noble numbers}

As we showed in the previous section, the entropy function reaches local minima at every rational number, and such minima are organised into families. Since the global minimum value of $S(x)$ is trivially reached for $x=0/1, 1/1$, where $S(0) = S(1) = 0$, the analysis of minima is concluded and now we turn to explore the structure of the maxima of $S(x)$.\\  

\subsubsection{Global maximum: \texorpdfstring{$\phi^{-1}$}{Phi} } 
%\subsubsection{Global maximum: $\phi^{-1}$} \ 

Here we make use of technique of Lagrange multipliers to prove that $S(x)$ reaches its global maximum at $x= \phi^{-1}=0.618\dots$, and by symmetry at $x'=1- \phi^{-1}$, where $\phi=(1+\sqrt{5})/2=1.618033...$ is the so-called Golden number. To fix the constraints of the Lagrange multipliers technique, we consider again the degree distribution $P(k,x)$. 
Since almost all Haros graphs fulfil $P(2,x)=x, P(3,x)=1-2x$ and $P(4,x)=0$ (Eq.\ref{eq1}), normalisation implies ${\cal Q}_0 = \sum_{k=5}^{\infty} P(k,x) = 1 - x$. 
Second, observe that $S(x)$ reaches minima at the rationals, and thus the global maximum should take place at an irrational number. According to Eq.\ref{eq:arith}, all irrational Haros graphs have arithmetic mean degree $\overline{k} = 4$, hence the second constraint is ${\cal Q}_1 = \sum_{k=5}^{\infty} k\cdot P(k,x) = 5 - 4x$. Altogether, for irrational numbers, the following Lagrangian functional can be defined:
\begin{equation*}
\mathcal{L}[\left\lbrace P(k,x) \right\rbrace]=-\sum_{k=5}^{\infty}{P(k,x)\log{P(k,x)}}-(\lambda_0-1)\left(\sum_{k=5}^{\infty}{P(k,x)}-{\cal Q}_0 \right)
-\lambda_1\left(\sum_{k=5}^{\infty}{kP(k,x)}- {\cal Q}_1 \right)
\end{equation*}
which will take an extreme in the solution $\delta {\cal L}[P^*(k,x)]=0$ that also extremizes the entropy $S(x)$. One can prove that such extrema are reached for a specific shape of the degree distribution $P(k,x)$ (see Appendix \ref{appendix:E} for details):

\begin{equation}
P(k,x)=\left\{
\begin{array}{ll}
1-x & k=2 \\
2x-1 & k=3 \\
0 			& k=4\\
x^{k-1} & k\geq 5%
\end{array}%
\right.
\label{GoldenDistribution}
\end{equation}

Note that this technique does not determine the value of $x$ for us, only the shape of the degree distribution. But according to Eq.\ref{GoldenDistribution}, whatever $x$ is,  $G_x$  will necessarily contain nodes spanning all possible degrees (except $k=4$), i.e., the degree distribution does not have any holes.  According to the analysis in Section III, this implies that the symbolic path that identifies $G_x$ in the binary tree tree cannot have adjacent repeated symbols, i.e. it needs to follow an infinite zigzag, hence $x=\phi^{-1}$, the reciprocal of the Golden number, and by symmetry $1-x=1- \phi^{-1}$.\\

%cannotThis graph is necessary $G_{\phi^{-1}}$ associated to the irrational $x = \phi^{-1}$, the reciprocal of the Golden number, associated at the only graph where all the connectivities appear due to the infinite zigzag of its path in the binary Farey tree. 

\noindent We can also independently determine $P(k,\phi^{-1})$ from the successive degree distributions of Haros graphs $G_{\mathfrak{F}_{n-1}/\mathfrak{F}_n}$ associated with rational convergents $\mathfrak{F}_{n-1}/\mathfrak{F}_n$ of $\phi^{-1}$ obtained via the Fibonacci sequence $(\mathfrak{F}_n)$, see Fig. \ref{fig:Ruta_aurea}. From a previous work \cite{QuasiperiodicGraphs}, we have:

\begin{equation}
P\left(k,\frac{\mathfrak{F}_{n-1}}{\mathfrak{F}_n}\right)=\left\{
\begin{array}{lll}
1 - \frac{\mathfrak{F}_{n-1}}{\mathfrak{F}_n} = \frac{\mathfrak{F}_{n-2}}{\mathfrak{F}_n} &  & k=2 \\
2\cdot \frac{\mathfrak{F}_{n-1}}{\mathfrak{F}_n} - 1 = \frac{\mathfrak{F}_{n-3}}{\mathfrak{F}_n} & & k=3 \\
0 			& & k=4\\
\frac{\mathfrak{F}_{n+k-1}}{\mathfrak{F}_n} & &5 \leq k \leq n+1%
\end{array}%
\right.
\end{equation}

Now, using Binet's formula and taking the limit $\lim_{n\to\infty} \mathfrak{F}_{n-1}/\mathfrak{F}_n = \phi^{-1}$, we recover Eq.\ref{GoldenDistribution} with $x=\phi^{-1}$.

\begin{figure}[htbp]
\includegraphics[width=0.45\columnwidth, trim=0cm 0cm 1cm 0cm,clip=true]{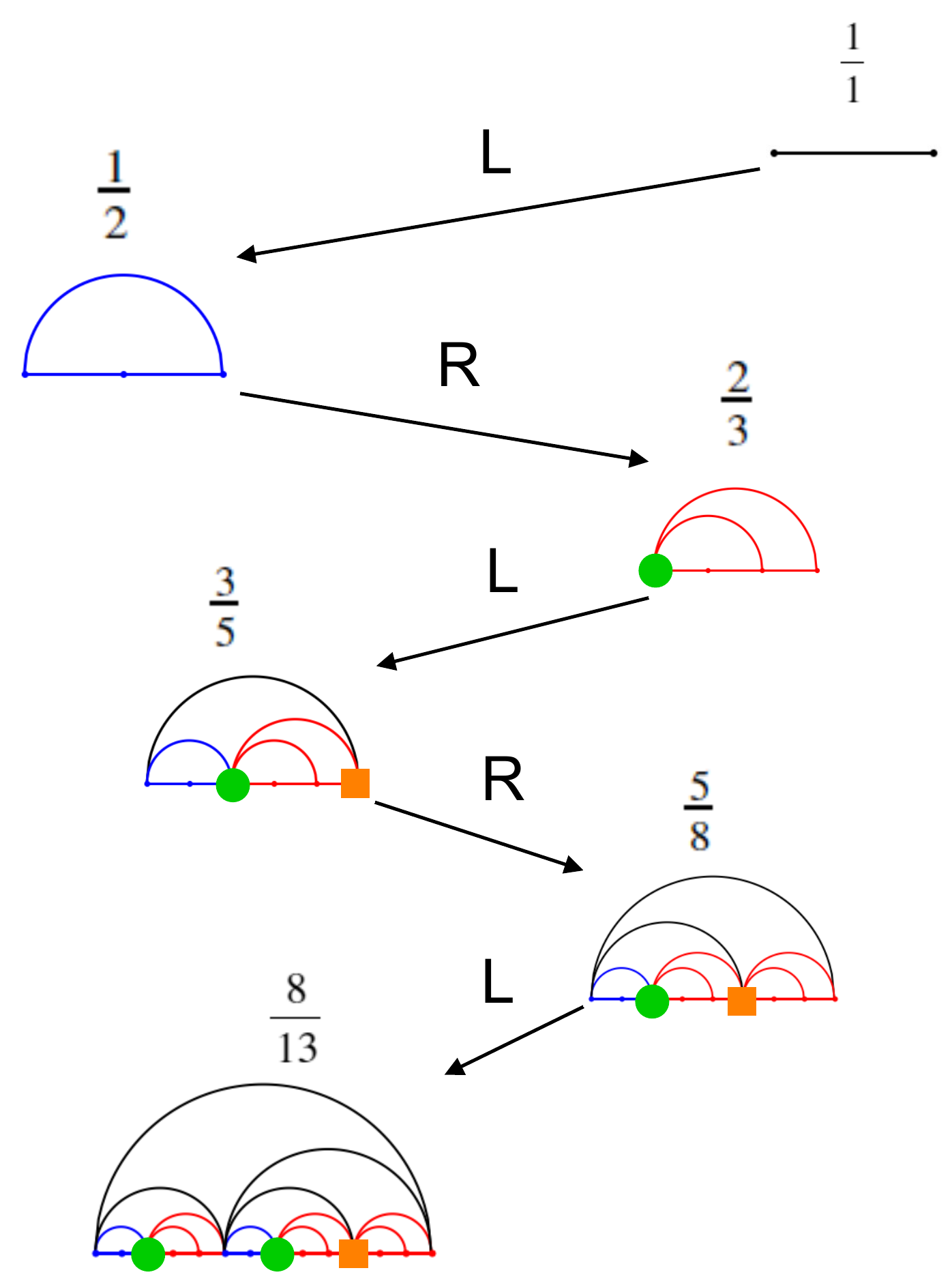}
\caption{We depict the first six steps of the path of the Haros graph tree to $G_{\phi^{-1}}$. The rational numbers $1/1, 1/2$, $2/3$, $3/5$, $5/8$ and $8/13$ are quotients $\frac{\mathfrak{F}_{n-1}}{\mathfrak{F}_{n}}$ of Fibonacci elements. The green points highlight the nodes with degree $k = 5$ and illustrate how the zigzag descent increase the degree frequency following a Fibonacci sequence, whereas the orange square points remark the nodes with degree $k = 6$. Due to lack of space, only six of these convergents are shown. However, it is easy to verify that the next convergent $8/13 \oplus 5/8 = 8/13 = 13/21$ has $3$ green nodes and $2$ orange square node, and, in general, the frequency of green and orange square nodes follow a Fibonacci sequence.
%The concatenation graph emulates the Fibonacci sequence if the path is the infinite $LR$ downstream.
}
\label{fig:Ruta_aurea}
\end{figure}

If additional restrictions are imposed in the degree distribution, which Haros graphs will be maximally entropic accordingly? In what follows, we show that further restrictions can be parametrised by holes in the degree distribution, i.e. $P(\kappa,x)=0$ for $4\leq \kappa \leq f(n)$ (where $n$ denotes an ordinal).
We will prove that if the restrictions are a finite number of zeros of $P(k,x)$, the emergent entropy local maxima coincide with a family of affine transformations of the Golden number. 

\subsubsection{Local maxima }

We have proved that the entropy function $S(x)$ finds its global maximum at $x=\phi^{-1}$ (and by symmetry at $1-x=1-\phi^{-1}= 1/(2+\phi^{-1})$). Here, we further explore the hierarchy of local maxima that emerge in the rich structure of $S(x)$. For simplicity, from now on we focus only on $x \in [0,1/2]$, since the results are then extended by symmetry to $[0,1]$.
%As proved in the preceding section, the global maximum of the entropy function in $[0,1/2]$ is attained at $x=1-\phi^{-1}= 1/(2+\phi^{-1})$.
Now, let us again consider the graphic representation of $S(x)$, see Fig.\ref{fig:Entropia_nobles_suelo}. For $x \in [0,1/2]$, one can extract a set of local maxima's positions ${\cal M}_1$ that can be connected by a certain envelope curve ${\cal C}_1(n)$, highlighted by red dots connected by a red dashed line in \ref{fig:Entropia_nobles_suelo}). Numerically, ${\cal M}_1$ appears to coincide with the set of numbers $${\cal M}_1=\bigg \{\ \frac{1}{2+\phi^{-1}}, \frac{1}{3+\phi^{-1}}, \frac{1}{4+\phi^{-1}}, \frac{1}{5+\phi^{-1}},\dots \bigg \},$$
whose elements are naturally ordered by the discrete function ${\cal C}_1(n)$:
$${\cal C}_1(n)=\frac{1}{n+\phi^{-1}}, \ n\geq 2.$$

\begin{figure}[htbp]
\includegraphics[width=1.1\columnwidth, trim=4cm 0cm 0cm 0cm,clip=true]{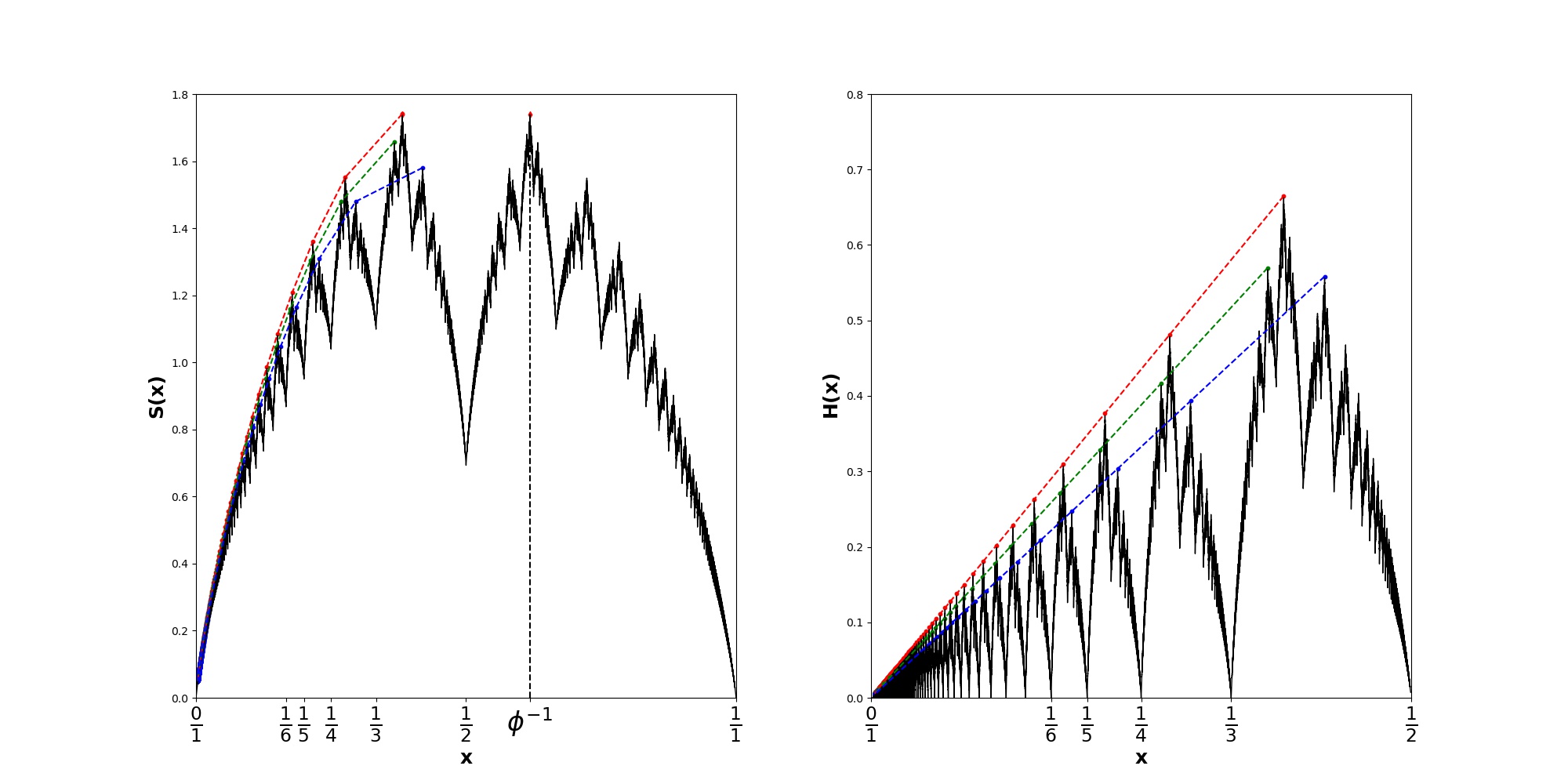}
\caption{ Entropy $S(x)$ (left panel) and reduced entropy $H(x)$ (right panel), computed for all Haros graphs $G_x$ with $x \in {\cal F}_{1000}$. In both panels, the red dots correspond to the local maxima ${\cal C}_1(n)=\frac{1}{n+\phi^{-1}}$ for $n \geq 2$. The global maxima are ${\cal C}_1(2) = 1 - \phi^{-1}$ and its mirror symmetry $\phi^{-1}$. Green and blue dots represent other families of noble numbers. The right panel helps to visualise how these families display a linearly increasing reduced entropy with a slope that depends on the family and where each element appears only once in each interval ${\cal I}_n$.}
\label{fig:Entropia_nobles_suelo}
\end{figure}

First, we study in some detail the elements of ${\cal M}_1$. Elements in this set are a subset of the so-called \textit{noble numbers}: irrationals whose continued fraction expansion eventually reaches an infinite sequence of $1$'s:  $[a_1, a_2, \dots,a_m ,1,1,1,\dots]$. We denote the initial sequence of symbols $[a_1, a_2, \dots, a_m]$ as the transient sequence. The reciprocal of the Golden number $\phi^{-1}$ is a very particular case of a noble number with no transient sequence ($m=0$). In particular, the elements in ${\cal M}_1$ (and ordered by ${\cal C}_1(n)$) have a transient sequence $[a_1]$, with $m=1$ and $a_1=n$.\\
The noble numbers can also be coded as an infinite path in the Farey binary tree, which, after a transient sequence of $L$'s and $R$'s, eventually reaches an infinite zigzag $(\text{LR})^{\infty}$.
%incidentally an Haros graph whose properties are exactly known as one can find  exactly $G_{p/q}$  . 
For example, ${\cal C}_1(3)$  has a continued fraction expansion $[3,\overline{1}]$, i.e., its transient sequence therefore consists only of the number $3$, and its symbolic path is $L^2 (LR)^{\infty}$. The convergents in the transient are thus $0/1$ and $1/3$, and once the infinite zigzag sets in, the path continues and leads to the successive approximants $1/4, 2/7, 3/11, 5/18, ...$. Observe that this sequence of approximants can also be generated using suitable combinations of Fibonacci numbers, e.g. for this case 
\begin{equation}
    \bigg(\frac{1}{4}, \frac{2}{7}, \frac{3}{11}, \frac{5}{18}, ...\bigg) = \bigg(\frac{\mathfrak{F_{n}}}{ \mathfrak{F_{n-1}} + 3 \mathfrak{F_{n}}}\bigg)_{n\geq1},
    \label{eq:0113}
\end{equation}
where the right-hand side of Eq.\ref{eq:0113} is just the particular case of the general expression
\begin{equation}
    \bigg(\frac{A\mathfrak{F_{n-1}} + B\mathfrak{F_{n}}}{C\mathfrak{F_{n-1}} + D\mathfrak{F_{n}}}\bigg)_{n\geq1},
    \label{eq:abcd}
\end{equation}
where $A/C$ and $B/D$ are the last two convergents of the transient sequence, i.e. $A=0,\ C=1,\ B=1,\ D=3$ in our example $x={\cal C}_1(3)$.
%($G_1 = 0$, $G_2 = 1$ for the numerators  and $G_1 = 1$, $G_2 = 3$ for denominators). The sequences can be expressed as $G_n = G_1 \cdot F_{n-1} + G_2 \cdot F_{n}$. In the example, $\frac{F_{n}}{ F_{n-1} + 3F_{n}} $. 
Taking the limit $n\to \infty$ in the rhs of Eq.\ref{eq:0113}, we correctly recover the irrational number $\mathcal{C}_1(3) = \frac{1}{3 + \phi^{-1}}$. 
Likewise, taking the limits in Eq.\ref{eq:abcd} provides an expression for all noble numbers as $$\frac{A\phi^{-1} + B}{C\phi^{-1}+D}$$
where it is easy to see that for a general $\mathcal{C}_1(n)$, the convergents are $0/1$ and $1/n$ i.e. $A=0, \ B=1, \ C=1, \ D=n$.\\ 

%\textcolor{red}{Para discutir: El hecho de que la frac. continua $[a_1, a_2, ... , a_n]$ tenga su traslado al camino $L^{a_1} R^{a_2} ... (L/R)^{a_n -1}$ trae problemas en el caso infinito de los numeros nobles. Ejemplo tomando el numero $[2,1,2,\overline{1}]$. La parte finita $[2,1,2] = 3/8$ consiste en seguir el camino $LLRL$. Sin embargo, al tomar el numero con la parte finita e infinita, el camino que seguria es $LLRLL(RL)^{\infty}$, donde tenemos $2$ opciones: o considerar $LLRLL(RL)^{\infty}$, de modo que el camino se situa en $4/11 = [2,1,3]$ y luego hace el zig-zag empezando por $R$; o bien considerar $LLRL(LR)^{\infty}$, donde el camino se situa en $3/8$ y luego hace el zig-zag empezando por $L$. La primera representacion hace el zig-zag con los padres $4/11$ y $3/8$, dando lugar a la representacion $\frac{4 + 3\phi^{-1}}{11 + 8\phi^{-1}}$, mientras que la segunda hace el zig-zag con $1/3$ y $3/8$, dando lugar a la representacion $\frac{3 + \phi^{-1}}{8 + 3\phi^{-1}}$. Sea $m = \sum a_i$, entonces el grafo $p/q$ estara en el nivel $m$.
Let us consider then the topological structure of Haros graphs $G_{{\cal C}_1(n)}$. Since noble numbers ${\cal C}_1(n)$  have paths in the Farey binary tree that start with a sequence of $n$ consecutive $L$ s, by theorem \ref{Tma:Rep}, it turns out that degrees $4 \leq k \leq n + 2$ do not appear in $G_{{\cal C}_1(n)}$. The degrees that do appear can be calculated using $F({\cal C}_{1}(n)) = {\cal C}_{1}(n+1)$, the expression of $P(k, \phi^{-1}) = P(k, 1 - \phi^{-1} = {\cal C}_{1}(2))$ and the scaling equation \ref{Pscaling}. In the same manner, taking as additional the restriction over the interval $[0, 1/n]$ into consideration and using similar techniques as those leading to Eq.\ref{GoldenDistribution}, after a bit of algebra and using Eqs. \ref{eq:0113} and  \ref{eq:abcd}, it can be proved that

\begin{equation}
P(k,{\cal C}_1(n))=\left\{
\begin{array}{ll}
{\cal C}_1(n) & k=2 \\
1 - 2{\cal C}_1(n) & k=3 \\
0 			&  4 \leq k \leq 2 + n\\
{\cal C}_1(n)(\phi^{-1})^{k-(n+1)} & k \geq 3 + n%
\end{array}%
\right.
\label{eq:Pk_C1}
\end{equation}

Interestingly, the degree distribution for all ${\cal C}_1(n)$ has a universal exponential tail with the same exponent, based on the Golden number. It turns out that in the intervals $[0,\frac{1}{n}]$, the graph entropy function $S(x)$ necessarily reaches its maximum precisely at $x = {\cal C}_{1}(n)$ (see Fig. \ref{fig:Entropia_nobles_suelo} for an illustration and Appendix \ref{appendix:F} for a proof). 
%The MaxEnt principle is applied with additional constrains of $P(k)$ $. 
We can also prove that the reduced entropy $H(x)$ linearly interpolates the elements of ${\cal M}_1$, more concretely:
\begin{equation}
    H\left( {\cal C}_1(n) \right) = \left( -\log (\phi^{-1})\cdot(3 + \phi^{-1}) \right)\cdot {\cal C}_1(n) 
    \label{eq:H_C1}
\end{equation}

Hence, the slope of the red dashed line in Fig. \ref{fig:Entropia_nobles_suelo} is precisely $-\log (\phi^{-1})\cdot(3 + \phi^{-1})$ (See Appendix \ref{appendix:G} for more details).\\

\noindent One can now define the full set of Noble numbers via continued fractions as:

$${\cal C}_{m}(n_1, ... , n_m) = [n_1, n_2, ..., n_m, \overline{1}],$$
that is, these are numbers with an arbitrary transient, eventually followed by an infinite period-1 pattern (or, in terms of the binary sequence, followed by an infinite zigzag). Noble numbers are those quadratic irrationals which are quotients of affine functions of the Golden number, i.e.
$$
{\cal C}_{m}(n_1, ... , n_m) = \frac{p_m + \phi^{-1} p_{m-1} }{q_m + \phi^{-1} q_{m-1}},
$$
{where $p_i /q_i$ is the i-th convergent of the rational reached after the transient sequence $p/q = [n_1, ... , n_m]$.} In Fig.\ref{fig:Entropia_nobles_suelo} we highlight two specific families: ${\cal C}_{3}(n,1,2)$ (green line) and ${\cal C}_2(n,2)$ (blue line) (we remind that the red line interpolates ${\cal C}_1(n)$). For example, it can be proved (the proof is however cumbersome and not reported) that the degree distribution $P(k,x\in \mathcal{C}_{3}(n,1,2))$, with $x = \frac{3 + \phi^{-1}}{(3n + 2) + (n+1)\cdot \phi^{-1}}$ is:

\begin{equation}
\label{Eq:C3n12}
P\left(k,x\right)=\left\{
\begin{array}{ll}
x  & k=2 \\
1 - 2x & k=3 \\
0 &  4 \leq k \leq n + 2 \\
\frac{1}{3 + \phi^{-1}}\cdot x & k =  n+3 \\
\frac{1+ \phi^{-1}}{3+ \phi^{-1}}\cdot x  & k =  n+4 \\
0 & k = n + 5 \\
\frac{1}{3 + \phi^{-1}}\cdot x\cdot(\phi^{-1})^{k+n-4}  & k \geq  n+6 \\
\end{array}%
\right.  
\end{equation}%
In the same way, with the theoretical degree distribution, we can calculate the reduced entropy of $H(\mathcal{C}_{3}(n,1,2)) = \mathcal{C}_3(n,1,2) \cdot \left( \log(3 + \phi^{-1}) - \frac{1}{3+\phi^{-1}}\cdot \log(1-\phi^{-1}) \right)$. Again, the reduced entropy has the form $H(x) = a\cdot x$, as we have shown visually with the green dashed line in Fig.\ref{fig:Entropia_nobles_suelo}.\\
\noindent Interestingly, we can see that while the specific transient of the continued fraction has a specific echo in the first values of the degree distribution, its tail is still exponential with the same base as for the family ${\cal C}_1(n)$:
\begin{theorem}
Haros graphs $G_{x}$ where $x$ is an arbitrary noble number ${\cal C}_{m}(n_1, ... , n_m)$ have a degree distribution with an exponential tail, with the Golden number as a base.
\end{theorem}

%In general case, given two neighbours in a Farey sequence and the interval $I = [p_1/q_1, p_2/q_2]$, the local maximum over $I$ is the noble number $[a_1, ... , a_m, \overline{1}]$, where $p/q = [a_1, ... , a_m]$ is the rational transient defined by $p/q = p_1/q_1 \oplus p_2/q_2$ . The local maximum can be expressed as:

%$$\frac{p_n + \phi^{-1} p_{n-1} }{q_n + \phi^{-1} q_{n-1}} $$

Here we report a sketch of the proof: The degree distribution $P(k, x = {\cal C}_{m}(n_1, ... , n_m))$ is constructed in a zigzag process starting at the degree distribution of $p/q$. After the transient sequence, the infinite zigzag reaches the values:

$$\frac{p + p_{m-1}}{q + q_{m-1}}, \frac{2p + p_{m-1}}{2q + q_{m-1}}, \frac{3p + 2p_{m-1}}{3q + 2q_{m-1}} \to \frac{p + \phi^{-1} p_{m-1} }{q + \phi^{-1} q_{m-1}},$$

where $p_{m-1}/q_{m-1}$ is the $m-1$ convergent of $p/q$. At this point, the generation of new degrees follows the process illustrated by the green and orange points in Fig. \ref{fig:Ruta_aurea}. In each descent, a new degree appears at the boundary node, which becomes an inner node in the next two descents. After that, its value increases as a Fibonacci sequence: $2, 3, 5, 8, 13, ...$, whereas the denominators increase as $q + q_{m-1}, 2q + q_{m-1},3q + 2q_{m-1}, ... $ Hence, the tail of the degree distribution of the noble number will have follow:
$$
\frac{\mathfrak{F}_{n - (k - \alpha)}}{\mathfrak{F}_n \cdot q + \mathfrak{F}_{n-1}\cdot q_{m-1}},
$$
where the value $\alpha$ depends on the concrete noble number $x$. Using Binet's formula, from the last expression we easily obtain an exponential tail, where the base is indeed the reciprocal of the Golden number (see Fig. \ref{fig:Fig_log_dist}).\\

To recap, we have proved the Haros graphs associated with Noble numbers --a subset of the family of quadratic irrationals-- have a degree distribution with an exponential tail of base $\phi^{-1}$. This result has been obtained via an entropic-maximization process, where the infinite zigzag period-1 expansion $[\bar 1]$ of Noble numbers plays an important role. The question arises naturally: What happens if we then consider a periodic patterns of larger period $\geq 2$? As a matter of fact, the so-called family of Metallic ratios $\phi_{b}^{-1}$ are the positive solutions of the equation $x^2 +bx -1 = 0$, and their continued fraction expression is $[\overline{b}]$, i.e., metallic ratios have a period-$b$ pattern. For example, the reciprocal of the Golden number is the Metallic ratio for $b=1$, while $\sqrt{2} -1$ is the so-called Silver ratio, or Metallic ratio for $b = 2$. From \cite{QuasiperiodicGraphs}, the theoretical expression of $P(k,\phi_{b}^{-1})$ is:  

\begin{equation}
P(k,\phi_{b}^{-1})=\left\{
\begin{array}{ll}
\phi_{b}^{-1} & k=2 \\
1 - 2\phi_{b}^{-1} & k=3 \\
0 			&  k = 4\\
(1 - \phi_{b}^{-1})\cdot (\phi_{b}^{-1})^{\frac{k - 3}{b}}& k = bn + 3, n\geq 1 \\
0 & \text{otherwise}
\end{array}%
\right.
\label{eq:Metallic_ratio}
\end{equation}

The proof of Eq.\ref{eq:Metallic_ratio} is similar to others presented in this paper and proceeds to use Lagrange multiplier techniques with the constraints $P(k,x) = 0$ for values $k \neq bn + 3$, i.e., we search for local maxima establishing an infinite number of periodic restrictions. We conclude that Haros graphs associated with metallic ratios have a degree distribution with a periodic pattern of zeros and an exponential tail with base the metallic ratio $\phi_{b}^{-1}$, and are also local maxima of $S(x)$.\\

\noindent We close this section by further generalizing our previous results: we can define the full set of generalised metallic ratios (note that this family contains all Noble numbers) as
$${\cal C}_{b,m}(n_1, ... , n_m) = [n_1, n_2, ..., n_m, \overline{b}],$$
We can state the following theorem (the proof is cumberstone and is not presented, but its essence is the same as that seen in previous theorems):

\begin{theorem}
 Haros graphs associated with generalised metallic ratios ${\cal C}_{b,m}(n_1, ... , n_m)$ have a degree distribution with a periodic pattern of zeros and an exponential envelope with a base $\phi_{b}^{-1}$, and are thus also local maxima of $S(x)$.
\end{theorem}

\begin{figure}[h!]
\includegraphics[scale=0.6]{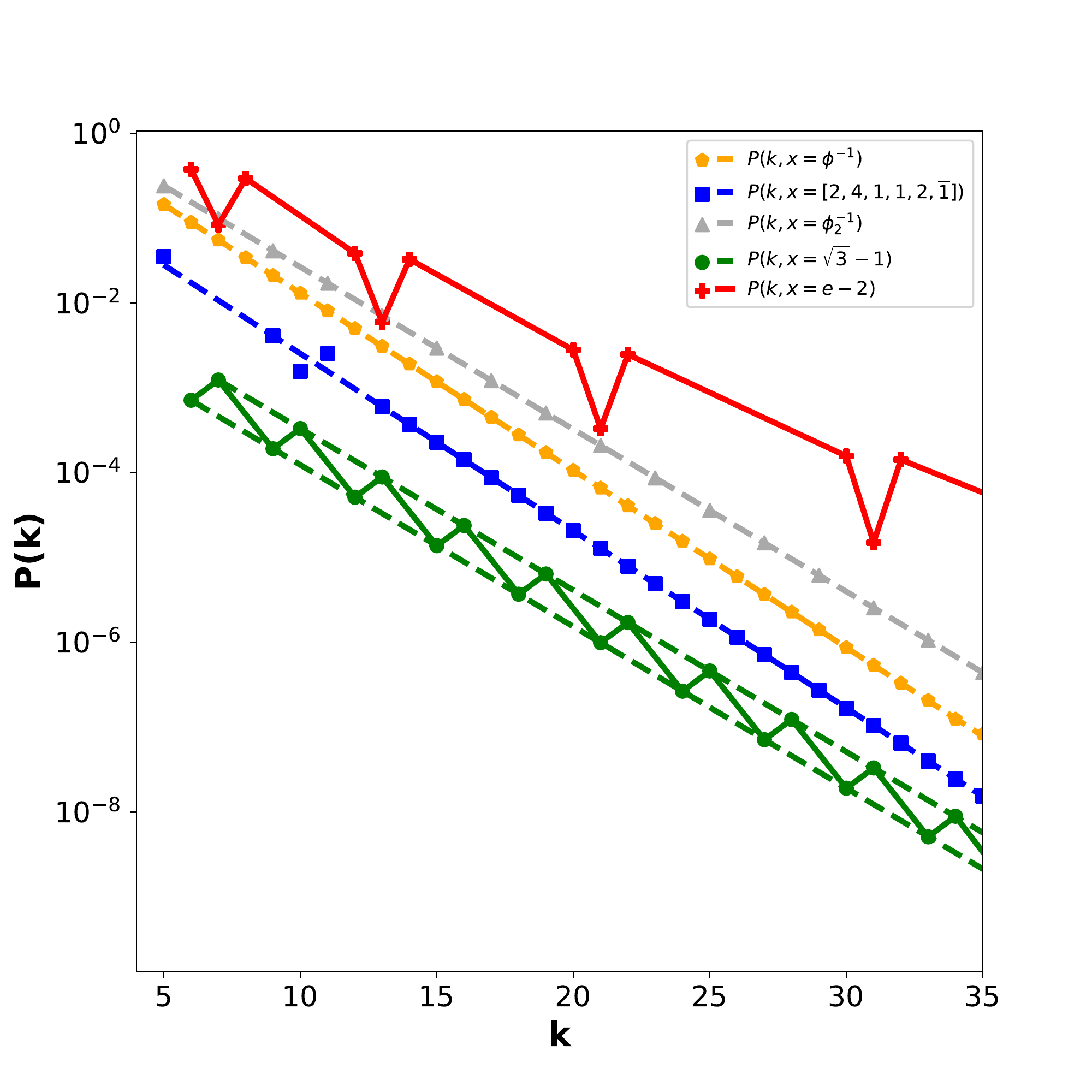}
\caption{Semi-log plot of the numerical degree distributions $P(k,x)$ for values $5\leq k \leq 30$ of Haros graphs associated with irrational numbers of different families: the reciprocal of the Golden number (golden circles), the general noble number $[0; 2, 4, 1, 1, 2, 2, \overline{1}]$ (blue squares), the silver mean or metallic ratio with $b = 2$, $\sqrt{3} - 1 = [\overline{1,2}]$,  a quadratic irrational with continued fraction of period greater than $1$, and the number $e$, a non-algebraic irrational. The dots represent the numerical $P(k,x)$ obtained by the large convergents of each irrational number. The dashed lines correspond to theoretical tails of the degree distribution of each quadratic irrational. In the case of $\sqrt{3}-1$, the degrees $k_1 = 3n +1$ and $k_2 = 3n+2$ form two parallel exponential tails. It can be observed that the tails of the degree distribution of the reciprocal of Golden number and the noble numbers are parallels and their slope matches the the reciprocal of the Golden number.}
\label{fig:Fig_log_dist}
\end{figure}

\section{Discussion and Open Problems}
Let us first summarise the paper's highlights. In this work, we have introduced Haros graphs $G_x$, that provide a graph-theoretic representation of real numbers $x\in[0,1]$. The structure of the Farey binary tree is replicated using the set of Haros graphs and a concatenation operator that plays the role of the mediant sum operator in graph space. The degree sequence of Haros graphs encapsulates rich information of the number they represent. We have provided analytical results on the degree distribution $P(k,x)$ of the Haros graph $G_x$ associated to $x\in[0,1]$ and provide insight into such topological structure in relation to the continued fraction representation of $x$. We have further explored the arithmetic and geometric mean degree of Haros graphs for both rational and irrational numbers, revealing an intricate fractal structure which is absent in the continued fraction expansion. \\
In a second step, we have defined a graph entropy $S(x)$ in terms of the entropy over the degree distribution of the Haros graph $G_x$. We have proved that such an entropy is a self-affine function, which can be expressed in terms of a generalised de Rham curve. The local minima of such a function relate to rational numbers, whereas the local maxima are related to specific families of irrationals.
We have proved that the subset of noble numbers ${\cal C}_1(n)$ are local maxima of $S(x)$ and have an Haros graph whose degree distribution has an exponential tail with base equal to the reciprocal of the Golden number, and additional analytical and numerical evidence supports the fact that this is also the case for the whole set of noble numbers ${\cal C}_{m}(n_1,..., n_m)$.
Moreover, the Noble numbers and metallic ratios are also local maxima of $S(x)$, and their Haros graphs have a degree distribution with an exponential tail with base the reciprocal of the Golden number or metallic ratio, respectively. In other words, the infinite zigzag period-$b$ expansion $[\overline{b}]$ that emerges after the transient of the continued fraction expansion of any Noble number or metallic ratio is mapped into an Haros graph with a degree distribution with an exponential tail of base $\phi_{b}^{-1}$. Observe that Noble numbers and metallic ratios are subsets of quadratic irrationals, this larger set being formed by numbers whose continued fraction --after a generic transient-- reaches a generic periodic pattern of period $b$ (Nobles have period 1 pattern). At this point we can state the following conjectures, that will form the basis of future work:

%In other words, the infinite zigzag period-$1$ expansion $[\overline{1}]$ that emerges after the transient of the continued fraction expansion of any Noble number is mapped into an Haros graph with a degree distribution with an exponential tail of base $\phi^{-1}$. Observe that Noble numbers is a subset of quadratic irrationals, this larger set being formed by numbers
%whose continued fraction --after a generic transient-- reaches a generic periodic pattern of period $\geq 2$ (Nobles have period 1 pattern). For instance the so-called metallic ratios have a period-2 pattern. At this point we can state the following conjectures, that will form the basis of future work:

\begin{conjecture}
 All quadratic irrationals have associated Haros graphs with a periodic pattern of zeros in the degree distribution and an exponential envelope in the tail whose base is related to the periodic pattern of zeros (i.e., to the periodic pattern of the continued fraction expansion).
\end{conjecture}

\begin{conjecture}
 Only quadratic irrationals have the properties displayed in the previous conjecture. Accordingly, non-quadratic irrationals have associated Haros graphs with a degree distribution which is not maximally entropic, i.e. does not have an exponential tail.
\end{conjecture}

For illustration of these conjectures, in Fig. \ref{fig:Fig_log_dist} we plot $P(k,x)$ for $x=\phi^{-1}$ (Golden number, period $1$), $x = [0; 2, 4, 1, 1, 2, 2, \overline{1}]$ (Noble number, period $1$), $x=\phi_{2}^{-1}$ (Silver ratio, period $2$), $x=\sqrt{3} - 1$ (non-noble, nonmetallic, quadratic, period $1,2$) and $x=e - 2$ (non-algebraic). A proof of conjecture 3 would naturally yield a classification of real numbers into quadratic and non-quadratic irrationals.\\
Furthermore, additional research should also be undertaken to explore the ability of Haros graphs to provide constructive classifications of other sets of irrationals, e.g., transcendental or normal numbers, this being an important open problem in mathematics. In this sense, see Fig. \ref{fig:Trascen_alg_numbers} where we plot the degree distributions of Haros graphs associated with different transcendental, algebraic and unclassified numbers,  giving a preliminary hint of the apparent differences between algebraic and trascendental numbers in the Haros graph representation.\\ 

\begin{figure}[h!]
\centering
\includegraphics[scale=0.4,clip=true]{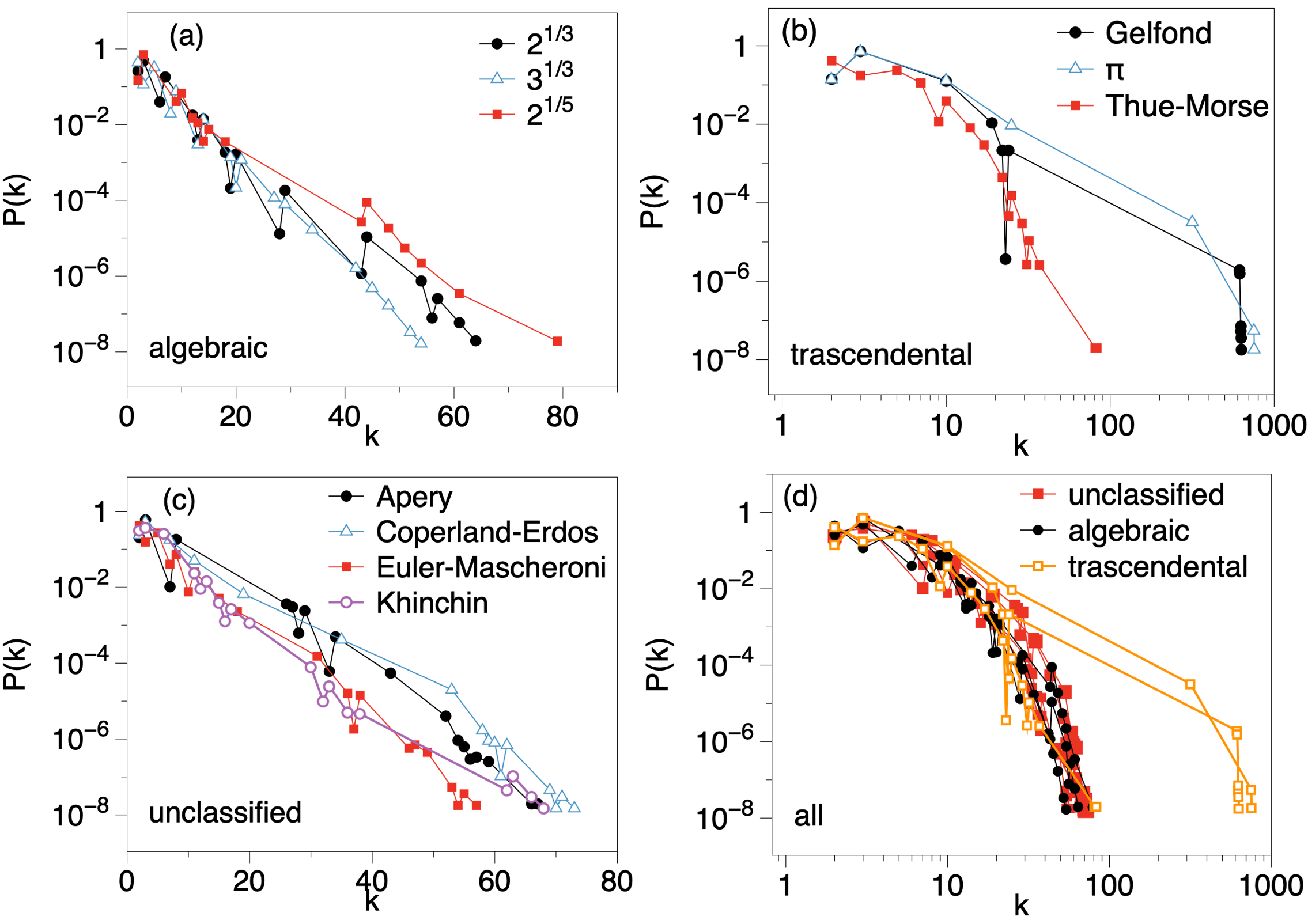}
\caption{Degree distributions $P(k,x)$ of Haros graphs associated with irrational numbers of different families: (a) (non-quadratic) algebraics (fractional parts of $\sqrt[3]{2}, \sqrt[5]{2}$  and $\sqrt[3]{3}$), (b) trascendentals (Gelfond constant $e^{\pi}$, the fractional part of $\pi$ and Prouhet–Thue–Morse $\tau$), (c) unclassified (Euler-Mascheroni constant $\gamma$, fractional part of Apery constant $\zeta(3)$, fractional part of Khinchin $K_0$ and Copeland-Erd\"{o}s constant $C_{CE} = 0.2357111317...$). Some plots are in semi-log, revealing approximate yet not perfect exponential trends for (non quadratic) algebraic numbers, while others are in log-log, revealing longer tails. Panel (d) merges all of them, and suggest that trascendentals and algebraics typically display different structural signatures.}
%Semi-log dot plot of the numerical degree distributions $P(k,x)$ for values $2\leq k \leq 80$ of Haros graphs associated with irrational numbers of different families: transcendental numbers, algebraic numbers, and unclassified numbers, i.e, numbers for which the classification is still unknown. The transcendental number are: the Gelfond constant $e^{\pi}$, the fractional part of $\pi$ and Prouhet–Thue–Morse $\tau$. Algebraic numbers are the fractional part of $\sqrt[3]{2}, \sqrt[5]{2}$  and $\sqrt[3]{3}$. Unknown numbers are: Euler-Mascheroni constant $\gamma$, fractional part of Apery constant $\zeta(3)$, fractional part of Khinchin $K_0$ and Copeland-Erd\"{o}s constant $C_{CE} = 0.2357111317...$.}
\label{fig:Trascen_alg_numbers}
\end{figure}

The above-mentioned open problems and conjectures relate to a proposed research programme on the \textit{structural} properties of Haros graphs. Complementary to this, we envision two additional research avenues. The first is about defining \textit{dynamical} rules on the set of Haros graphs by means of a graph renormalization operator \citep{AnalyticalFeigenbaum, QuasiperiodicGraphs} and accordingly studying how the dynamics attractors relate to the underlying number system.
%dynamic can be defined in the system. This graph dynamic has a translation to real numbers: the well-known function modified Farey map. Moreover, in \cite{AMM}, a dynamical system is introduced on the continued fraction.  This dynamic has its correlation through the renormalization of graphs \citep{AnalyticalFeigenbaum, QuasiperiodicGraphs}. It can be explored how this dynamics affect to different families of real numbers and the relationship with the entropy.
The second stems from the fact that Noble numbers fulfil a maximum entropy principle, which suggests the development of a \textit{statistical-mechanical} formalism underlying the set of Haros graphs.
All in all, we sincerely hope that the community will find these open problems of interest and that the research programme on Haros graphs will get momentum in the years to come.

%although only the subfamily of noble numbers has been detailed, it can be observed how all quadratic irrationals have an exponential connectivity distribution. The characteristics of this type of distribution resemble a Boltzmann distribution. This description is defined by an energy level, the Boltzmann constant and a temperature.

\newpage
\section{Appendix}
\addcontentsline{toc}{section}{Appendix}
\renewcommand{\thesubsection}{\Alph{subsection}}

\subsection{Proof of the three first values of  \texorpdfstring{$P(k,x)$}{Pkx}}
%\subsection{Proof of the three first values of $P(k,x)$}

\label{appendix:A}

\noindent  In what follows, we assume $p/q<1/2$ without loss of generality. We first prove that, out of a total of $q$ nodes (the period), $G_{p/q}$ will have $p$ nodes with degree $k=2$  and $q - 2p$ nodes with degree $k=3$, i.e. $P(2,p/q)=p/q$ and $P(3) = 1 - 2\cdot p/q$.  We prove the results by induction over the level $n$ of  $\cup_{n} \ell_n$:\\

\noindent \textbf{Theorem: } ${P(2,p/q) = p/q}$.\\
\textit{Proof. }The root of the Haros graph tree $G_{1/2}$ has $p=1$ nodes with degree $k=2$. Let us assume now that two given graphs  are {\it adjacents} in the Farey binary tree at a given level $n$, with the property that their degree distributions fulfill the induction hypothesis $P(2,p_1/q_1)= p_1/q_1$ and $P(2,p_2/q_2)=  p_2/q_2$ for $G_{p_1/q_1}$ and $G_{p_2/q_2}$, respectively. By virtue of the properties of the concatenation operation, we have $G_{p_1/q_1} \oplus G_{p_2/q_2}=G_{p/q}$, where $p=p_1+p_2$ y $q=q_1+q_2$. Since, by construction in all Haros graphs, but $G_{0/1}$, the initial and final nodes have degree $ k \geq 2$ (Fig. \ref{fig:FareyGraphTree} illustrates this matter in the first level of Haros graph Tree), an Haros graph resulting from the concatenation of two other Haros graphs does not produce news nodes with degree $k=2$, hence the total number of nodes with degree $k=2$ in $G_{p/q}$  is simply $p=p_1+p_2$, hence $P(2,p/q)=p/q$. \qedsymbol\\

\noindent  \textbf{Theorem: }${P(3,p/q) = 1 - 2\cdot p/q}$.\\
\textit{Proof. }We must distinguish the special case $p/q = 1/q$. Trivially, we find that $G_{1/2}$ has $2-2\cdot1=0$ nodes with degree $k=3$ in agreement with the statement. Let us assume that $G_{1/(q-1)}$ has $(q-1)-2=q-3$ nodes of degree $k=3$. Then, the graph $G_{1/q} = G_{0/1} \oplus G_{1/(q-1)}$ will have $q-3$ nodes with $k = 3$ plus one node obtained by the concatenation of $G_{0/1}$; i.e, $P(3,1/q) = (q-2)/q = 1 - 2\cdot(1/q)$. For all other cases, since by construction, the initial and final nodes in each Haros graph have degrees $k \geq 2$, the concatenation of two graphs ($G_{p_1/q_1} \oplus G_{p_2/q_2}=G_{p/q}$) does not produce new nodes with $k=3$, hence the total number of nodes with $k=3$ in the concatenated graph is simply the sum $(q_1 -2p_1) + (q_2 - 2p_2) = (q_1 + q_2)  - 2(p_1 +p_2) = q - 2p$. \qedsymbol \\

\noindent  \textbf{Theorem: }${P(4,p/q) = 0 }$.\\
\textit{Proof. }Finally, every Haros graph $G_{p/q}$ with $p/q \neq 1/2$ has no nodes with $k=4$. The proof is based on an earlier fact: It is not possible for nodes of order 4 to appear, as none of them appear in the interior of any graph. It does not appear in the union node, as there are no unions of nodes with connectivity $1$ and $3$, nor $2$ and $2$ that can generate it. \qedsymbol \\

\noindent Observe, to conclude, that we have focused on $G_{p/q}$ for $p/q<1/2$. In the symmetric case $p/q>1/2$, we end up with the same result if we change $p\to q-p$: $P(2, p/q)=1-p/q, P(3,p/q)=2\cdot p/q-1$ and $P(4,p/q)=0$.  %$\ \ \ \ \ $   \qedsymbol

\subsection{Proof of distribution of holes in \texorpdfstring{$P(k,x)$}{Pkx}}
%\subsection{Proof of distribution of holes in $P(k,x)$}

\label{appendix:B}
In order to proof the theorem, we want to study how do the Haros graphs acquire their structure as they grow during the descent through the Haros graph Tree. The proof presents a slight technical drawback. The boundary node convention expressed in Section III.A is applied after the construction of the whole Haros graph tree, i.e., to obtain a new Haros graph, the ancestors do not have the boundary node identification.

\begin{lemma}
\label{lemma:5}
The boundary node has connectivity $\kappa +2$ at the level $\kappa$, $\forall \kappa \geq 3$.
\end{lemma}
\textit{Proof of Lemma B.\ref{lemma:5}.} By induction over $\kappa$: \\

For $\kappa = 3$, the graphs $G_{1/3}$ and $G_{2/3}$ have connectivity $\kappa = 5$ at the boundary nodes.

Induction hypothesis: Suppose the result for $\kappa \leq n$. Let be an Haros graph $G = G_L \oplus G_R$, where $G, G_L$ and $G_R$ are at level $\kappa \leq n$ as we see in the figure \ref{fig:Fig17}.\\

\begin{center}
\begin{figure}
\includegraphics[width=0.7\textwidth, trim=0cm 20cm 0cm 0cm,clip=true]{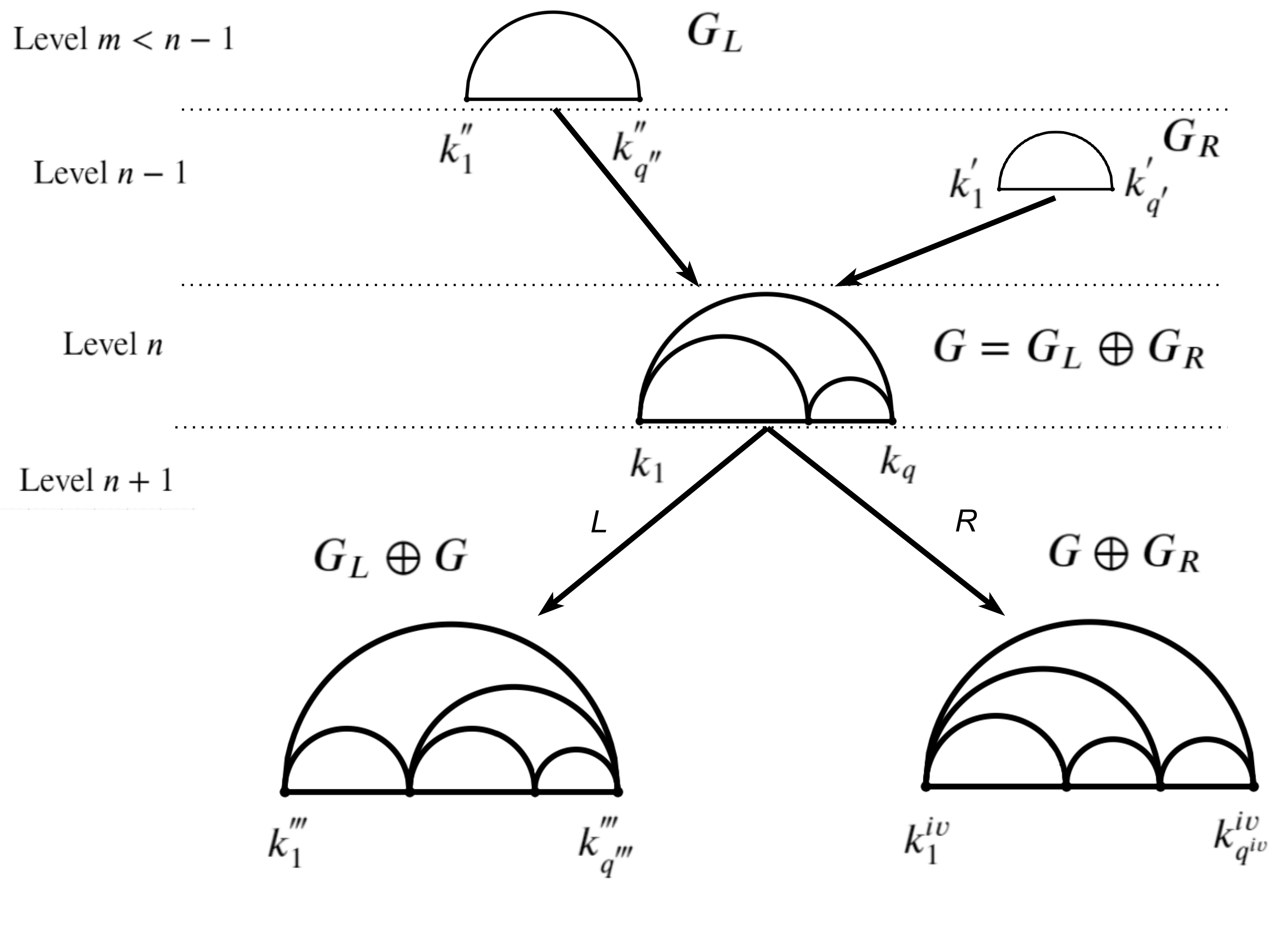}
\caption{Diagram of boundary node generation. Given the boundary node of Haros graph $G$ located at level $n$ and the boundary nodes of their ancestors $G_L$ and $G_R$, the diagram shows the boundary node of descendants of $G$ at level $n+1$.}
\label{fig:Fig17}
\end{figure}
\end{center}

The induction hypothesis gives us the result that the boundary nodes (the joining of two extreme nodes in the picture, as we mentioned in \textit{Boundary node convention}) are:

\begin{eqnarray}
k_{1}^{''} + k_{q''}^{''}& = & m + 2 \label{eqn:C1}\\
k_{1}^{'} + k_{q'}^{'}& = & (n-1) + 2 = n + 1 \label{eqn:C2} \\
k_{1} + k_{q}& = & n + 2 \label{eqn:C3}
\end{eqnarray}

The construction of Haros graphs implies the following relationships:

\begin{eqnarray}
k_{1} = k_{1}^{''} + 1  \label{eqn:C4}\\
k_{q} = k_{q}^{'} + 1 \label{eqn:C5}
\end{eqnarray}

Therefore, the descendant Haros graphs of $G$ are $G_L \oplus G$ and $G \oplus G_R$, located at the level $\kappa + 1$. The boundary node of $G_L \oplus G_L$ has connectivity:

$$k_{1}^{'''} +  k_{q'''}^{'''} = k_1 + (k_q + 1) = n + 3$$

whereas for $G \oplus G_R$:

$$k_{1}^{iv} +  k_{q^{iv}}^{iv} = (k_1 +  1) + k_q  = n + 3$$ \qedsymbol

\begin{corollary}
\label{cor:2}
The merging node has a smaller connectivity than the boundary node.
\end{corollary}

\textit{Proof of Corollary B.\ref{cor:2}.} The merging node of $G$ has connectivity (using equations [\ref{eqn:C1}, \ref{eqn:C3}, \ref{eqn:C4}, \ref{eqn:C5}]):

\begin{eqnarray*}
k_M &  = & k_{q''}^{''} + k_{1}^{'} =  (m + 2 - k_{1}^{''}) + (n + 1 - k_{q'}^{'}) \\
& = & (m - 2 - (k_1 -1)) + (n+1- (k_q -1)) \\
& = & m + 3 + n + 2 - (k_1 + k_q) = m + 3 \\
& < & n + 2 = k_1 + k_q
\end{eqnarray*}
\qedsymbol

It is easy to verify the following result:

\begin{corollary}
\label{cor:3}
The boundary node has the highest connectivity in the Haros graph.
\end{corollary}

\noindent The lemma \ref{lemma:5} and the corollaries \ref{cor:2} and \ref{cor:3} lead us to the following conclusion: The connectivity $\kappa$ may appear, at first, at the level $\kappa - 2$.
\begin{lemma}
\label{lem:4}
The symbol repetition in a path in the Haros graph Tree (either $LL$ or $RR$) will not generate a new connectivity except in the boundary node.
\end{lemma}

\textit{Proof of Lemma B.\ref{lem:4}. } We start by considering the diagram of the lemma \ref{lemma:5}. If $G = G_L \oplus G_R$ is located at level $\kappa = n$ and $G_R$ is located at $\kappa = n - 1$, then $G$ is the left descendant of $G_R$. Hence, the last symbol is $L$ (the proof is analogous for $R$). \\

The right descendant $G \oplus G_R$ has a merging node with connectivity:

$$
k_{M}^{iv} = k_q + k_1^{'} = k_{q'}^{'} + 1 + k_{1}^{'} = n + 2
$$
where we use $k_{1}^{'} + k_{q'}^{'} = n + 1$ according to the lemma \ref{lemma:5}. Hence, the connectivity $n + 2$ is enclosed in the Haros graph, and the subtree with root $G \oplus G_R$ will contain this connectivity in each Haros graph.\\

The left descendant $G_L \oplus G$ has a merging node with connectivity:

$$
k_{M}^{'''} = k_{q''}^{''} + k_1 = k_{q''}^{''} + k_{1}^{''}+ 1 = m + 3
$$

Using again the lemma \ref{lemma:5} to highlight $k_{q''}^{''} + k_{1}^{''} = m + 2$. However, this is the same connectivity as the merging node of $G$. Using equations \ref{eqn:C1}, \ref{eqn:C2} and $k_1 + k_q = n + 2$:

$$
k_{q''}^{''} + k_{1}^{'} = (m + 2 - k_{1}^{''}) + (n + 1 - k_{q'}^{'} = (m + 2 - k_1 + 1) + (n + 1 - k_q + 1) =  m + 3 + n + 2 - (k_1 + k_q) = m + 3
$$

To conclude the proof of the theorem, we need to ensure that connectivity $\kappa = n + 2$ cannot appear. But the merging node in $(G_L \oplus G) \oplus G$ has connectivity:

$$
k_{q'''}^{'''} + k_1 = k_q + 1 + k_1 = n + 3
$$.

Then, the connectivities generated in this subtree will be greater than $n + 2$. Analogously, it is easy to verify that $G_L \oplus (G_L \oplus G)$ will also not contain connectivity $n + 2$.
\qedsymbol

\subsection{Proof of scaling equations of \texorpdfstring{$P(k,x)$}, \texorpdfstring{$H(x)$} and \texorpdfstring{$S(x)$} (equations \ref{Pscaling1},  \ref{Pscaling}, \ref{Hscaling} and \ref{SHscaling})}
\label{appendix:C}
In order to prove the different scaling equations, we first analyse the operator $F(x) = \frac{x}{1+x}$ as real number function and as a Haros graph operator. Therefore, the first result is easy to prove:

\begin{lemma}
$x \in \ell_n \Rightarrow F(x) \in \ell_{n+1}.$
\label{lemma_Fx_niv}
\end{lemma}

\noindent Lemma \ref{lemma_Fx_niv} tells us that the image through the operator $F(x)$ is one level lower than $x$. The following lemma shows that the frequencies of degrees $k \geq 5$ of $G_{F(x)}$ are the same as the frequencies of degrees of $G_x$ but with displacement:
\begin{lemma}
\label{lemma:10}
Let be an Haros graph $G_x$ and let $m_{k,x}$ be the number of nodes with degree $k \geq 5$. Then, $m_{k,x} = m_{k+1,F(x)}$ where $F(x) = \frac{x}{1 + x}.$
\end{lemma}

\textit{Proof of Lemma B.\ref{lemma:10}. }
By induction over the levels of the Farey binary tree $\ell_n$:
\begin{itemize}
    \item In case $n = 3$, the proof is trivially true because there is only one degree $k$ greater than $5$, where we have $m_{5,1/3} = m_{6,1/4}$.
    
    \item Induction hypothesis: The result is true for every $G_{p/q}$ with $p/q \in \cup_{i\leq n} \ell_i$.
    
    \item Let us consider $G_{p/q}$ an Haros graph with $p/q \in \ell_{n+1}$. Then $G_{p/q} = G_{p_1 /q_1} \oplus G_{p_2/ q_2}$, with $p_1 /q_1, p_2/q_2 \in \cup_{i\leq n} \ell_i$. Consider an arbitrary non-null $k$ in the degree distribution of $G_{p/q}$. Therefore, we have $3$ possibilities:
    \begin{itemize}
        \item The boundary node has degree $k$: The lemma \ref{lemma:5} and the corollary \ref{cor:3} imply that $k = n + 3$ and $m_{k,p/q} = 1$. Hence, by the lemma \ref{lemma_Fx_niv} we get $F(p/q) \in \ell_{n+2}$ and its boundary node has degree $k = n + 4$. Again, by lemma \ref{lemma:5} and corollary \ref{cor:3}, we get $m_{k + 1, F(p/q)} = 1$ and the result is verified.
        
        \item The merging node has degree $k$: To clarify this case, we must distinguish between two cases:
        
        \begin{itemize}
            \item $m_{k, p/q} = 1$, i.e, the merging node is the only node with degree $k$. Therefore, an ancestor has the boundary node of degree $k$. We can assume, without loss of generality, that it is the left ancestor $p_1/q_1$. It is easy to see that we are in the case where $m_{k,p_1/q_1} = 1$, and even more so, this node of degree $k$ is the boundary node. Therefore, it is easy to verify that this the right ancestor $p_2/q_2$ cannot have nodes of degree $k$. Reasoning like in the previous case, we have: $$m_{k + 1, F(p/q)} = m_{k + 1, F(p_1/q_1)} + m_{k + 1, F(p_2/q_2)} = 1 + 0 = 1 = m_{k, p/q}.$$
            
            \item $m_{k, p/q} > 1$, i.e, other inner nodes have a degree $k$ apart from the merging node. In that subcase, we again find that one ancestor has no degree $k$. The other ancestor, supposing again $p_1/q_1$, has exactly the same number of nodes with degree $k$ as inner nodes of degree $k$ in $p/q$, i.e, $m_{k, p_1/q_1} = m_{k, p/q} - 1$. Applying the induction hypothesis and the concatenation of $F(p_1/q_1) \oplus F(p_2/q_2)$, we have $m_{k,p/q} = m_{k+1,F(p/q)}.$ \\
        \end{itemize}
        
        \item Only some inner nodes have degree $k$. In that case, we must distinguish two subcases:\\
        \begin{itemize}
        
            \item The number of nodes with degree $k$ in $G_{p/q}$ is the sum of nodes with degree $k$ in his ancestors, that is: $$m_{k, p/q} = m_{k, p_1/q_1} + m_{k, p_2/q_2}.$$ 
            Therefore, by hypothesis induction: $$m_{k, p_1/q_1} = m_{k+1, F(p_1/q_1)}$$
            $$m_{k, p_2/q_2} = m_{k+1, F(p_2/q_2)}$$ and we clearly have: $$m_{k+1, F(p/q)} = m_{k+1, F(p_1/q_1) } + m_{k+1, F(p_2/q_2)} = m_{k, p_1/q_1}.$$
            
            \item The other case occurs when $m_{k, p/q} = m_{k, p_1/q_1}$ and $m_{k, p_2/q_2} = 1$ (analogously if $m_{k, p/q} = m_{k, p_2/q_2}$ and $m_{k, p_1/q_1} = 1$). Hence, $m_{k+1, F(p_2/q_2)} = 1$ and $m_{k+1, F(p_1/q_1)} = m_{k, p_1/q_1}$, concluding that: $$m_{k+1, F(p/q)} = m_{k+1, F(p_1/q_1)} = m_{k, p_1/q_1} = m_{k, p/q}.$$
        \end{itemize}
    \end{itemize}
\end{itemize}
\qedsymbol

\noindent The lemma \ref{lemma:10} leads us to the first scaling equation. This self-similar behaviour can be observed in Fig. \ref{fig:Fig_xvsP(k)}
%\begin{theorem}
%For $k \geq 5$, we have the following scaling equation:
%$$
% P\left(k+1, F(x) = \frac{x}{x+1}\right) = \frac{1}{x+1} \cdot P(k,x) 
%$$
%\label{Tma10}
%\end{theorem}
\begin{theorem}
For all $k \geq 5$, we have the following scaling equation:
$$
 P\left(k+1, F(x) = \frac{x}{x+1}\right) = \frac{1}{x+1} \cdot P(k,x) 
$$
\label{Tma10}
\end{theorem}

\begin{proof}
Let us $x = p/q$ and $m_{k,p/q}$ be the number of nodes with degree $k$ in $G_{p/q}$. In virtue of Lemma \ref{lemma:10}, we have $m_{k,x} = m_{k+1, F(p/q)}$. Using that $F\left(\frac{p}{q}\right) = \frac{p}{p+q} $ then:

$$
P\left(k,\frac{p}{q}\right) = \frac{m_{k,p/q}}{q} = \frac{m_{k+1,F(p/q)}}{q} = \frac{p+q}{q}\cdot\frac{m_{k+1,F(p/q)}}{p+q} = \left(1 + \frac{p}{q}\right) \cdot P\left(k+1,F(p/q)\right). 
$$
\end{proof}

Theorem \ref{Tma10} shows that if we know the following degree values of $G_x$: $P(5,x), P(6,x), ... $, we know the degree values of $G_{F(x)}$: $P(6,F(x)), P(7,F(x)), ...$. Then, the only unknown value of $P(k, F(x))$ is $k = 5$. The following corollary shows this value to complete the degree distribution of $G_{F(x)}$:
 
\begin{corollary}

Given the degree distribution $P(k,x)$ of the Haros graph $G_x$, the degree distribution $P(k,F(x))$ of the Haros graph $G_{F(x)}$ is known. In particular, if $x > 1/2$, then $P(k = 5, F(x)) = \frac{2x-1}{x+1}$ and if $x < 1/2$, then $P(k = 5, F(x)) = 0$.
\label{cor:12}
\end{corollary}

\textit{Proof of Corollary B.\ref{cor:12}. } \\
If $x < 1/2$, we know that: $$\sum_{k \geq 5} P(k,x) = 1 - P(2) - P(3) - P(4) = 1 - x - (1-2x) = x.$$ 
Moreover, as $F(x) < 1/2$, we have: 

$$\sum_{k \geq 5} P(k,F(x)) = \frac{x}{x+1}.$$ 

By Theorem \ref{Tma10}, we note the following:

\begin{eqnarray*}
\frac{x}{x+1} = \sum_{k \geq 5} P(k,F(x)) = P(5, F(x)) + \sum_{k \geq 6} P(k,F(x)) =  \\
= P(5, F(x)) + \frac{1}{x+1}\sum_{k \geq 5} P(k,x) = P(5, F(x)) + \frac{x}{x+1}.
\end{eqnarray*}

\noindent Hence, $P(5, F(x)) = 0$.\\

\noindent Now, if $x > 1/2$, then: 
$$\sum_{k \geq 5} P(k,x) = 1-x$$ 
while again we have: 
$$\sum_{k \geq 5} P(k,F(x)) = \frac{x}{x+1}.$$ 
Using Theorem \ref{Tma10}, we observe:

\begin{eqnarray*}
\frac{x}{x+1} =  \sum_{k \geq 5} P(k,F(x)) = P(5, F(x)) + \sum_{k \geq 6} P(k,F(x)) = \\
= P(5, F(x)) + \frac{1}{x+1}\sum_{k \geq 5} P(k,x) = P(5, F(x)) + \frac{1 - x}{x+1}.
\end{eqnarray*}

Hence, $P(5, F(x)) = \frac{2x-1}{x+1}$.

\qedsymbol
\\

To illustrate the meaning of the corollary, we first show the degree distribution of the Haros graph $G_{2/5}$ and the Haros graph of $F(2/5) = 2/7$:

\begin{equation*}
    P(k, 2/5) =\begin{dcases}
        2/5  \; \; ,k = 2 \\
        1/5  \; \; ,k = 3 \\
        0  \; \; \; \; \; \; ,k = 4 \\ 
        1/5  \; \; ,k = 5,6 \\ 
        \end{dcases}
\, \, \, \,
    P(k, 2/7) = \begin{dcases}
        2/7  \; \; ,k = 2 \\
        3/7  \; \; ,k = 3 \\
        0  \; \; \; \; \; \; ,k = 4,5 \\ 
        1/7  \; \; ,k = 6,7 \\ 
    \end{dcases}
\end{equation*}

It can be observed that the degrees $k = 6,7$ of $G_{2/7}$ that appear are the same degrees as those of $G_{2/5}$ that displaced one position, i.e., $k = 5,6$. Moreover, the degree distribution follows the scaling equation given in Theorem \ref{Tma10}. However, if we now consider $x = 3/5 > 1/2$ and his image by $F(x)$, $3/8$, we have the degree distributions:

\begin{equation*}
    P(k, 3/5) =\begin{dcases}
        2/5  \; \; ,k = 2 \\
        1/5  \; \; ,k = 3 \\
        0  \; \; \; \; \; \; ,k = 4 \\ 
        1/5  \; \; ,k = 5,6 \\ 
        \end{dcases}
\, \, \, \,
    P(k, 3/8) = \begin{dcases}
        3/8  \; \; ,k = 2 \\
        2/8  \; \; ,k = 3 \\
        0  \; \; \; \; \; \; ,k = 4 \\ 
        1/8  \; \; ,k = 5,6,7 \\ 
    \end{dcases}
\end{equation*}

In that case, the degrees $k= 5,6$ appearing in $G_{3/5}$ correspond to the degrees $k = 6,7$ in $G_{3/8}$ through the scaling property. However, also the degree $k = 5$ is not generated via the scaling equation given in Theorem \ref{Tma10}.
\begin{theorem}
For $x \in [0,1]$, we have
\begin{equation*}
H(x)=\left\{
\begin{array}{ll}
(1+x)\cdot H(F(x)) & x \in [0,1/2] \\
(1+x)\cdot H(F(x)) - x\log x +  (2x-1)\log(2x-1)  +  (1-x)\log(1-x) & x \in [1/2, 1]  \\
\end{array}%
\right.
\end{equation*}

that is, the scaling of $H(x)$ is different depending on if $x < 1/2$ or $x > 1/2$.
\end{theorem}

\begin{proof}

Let us $x \in [0,1/2]$. By the definition of $H(x)$ given in \ref{def:EntH}, it is clear that $H(x) = -\sum_{k\geq 4} P(k,x)\log(P(k,x)) + x\cdot \log x$ (remember that $P(4,x)=0$). Therefore, using Theorem \ref{Tma10} :

\begin{eqnarray*}
H(x) &  = & -\sum_{k \geq 4} P(k,x)\log(P(k,x)) + x\cdot \log x \\
& = & -\sum_{k \geq 4} (1+x)P(k+1,F(x))\log((1+x)P(k+1,F(x))) + x\cdot \log x  \\
& = & (1+x)\left(-\sum_{k \geq 5}P(k,F(x))\log(1+x) -\sum_{k \geq 5}P(k,F(x))\log(P(k,F(x))) + \frac{x}{1+x}\cdot \log x \right)\\
& = & (1+x)\left(-\frac{x}{1 + x}\log(1+x) -\sum_{k \geq 5}P(k,F(x))\log(P(k,F(x))) + \frac{x}{1+x}\cdot \log x \right) \\
& = & (1+x)\left(-\sum_{k\geq 5}P(k,F(x))\log(P(k,F(x)))  + \frac{x}{1 + x} \cdot \log \frac{x}{1 + x} \right)\\
& = & (1+x)\left(-\sum_{k \geq 5}P(k,F(x))\log(P(k,F(x)))  + F(x) \cdot \log F(x)  \right)\\
& =&  (1+x)\cdot H(F(x))
\end{eqnarray*}

Let us consider $x \in [1/2, 1]$. In that case, we have $H(x) = -\sum_{k\geq 4} P(k) \cdot  \log(P(k)) + (1-x)\cdot \log(1-x)$. Moreover, using the scaling of $P(k,x)$ provided in Theorem \ref{Tma10} and the degree distribution of $P(k,F(x))$ given in Corollary \ref{cor:12}, we have the following:

\begin{eqnarray*}
H(x) &  = & -\sum_{k \geq 4} P(k,x)\cdot \log(P(k,x)) + (1-x)\cdot \log (1-x). \\
&&\text{Adding and subtracting the term $P(k_0, F(x)) = \frac{2x-1}{x+1}$:} \\
& = & -\sum_{k \geq 5} (1+x)\cdot P(k,F(x))\cdot \log((1+x)\cdot P(k,F(x))) + (1+x)\cdot \frac{2x-1}{x+1}\cdot \log \left( (1+x)\cdot \frac{2x-1}{x+1} \right)+ \\
 & + &(1-x)\cdot \log (1-x) \\
& = & (1+x)\cdot \left[-\sum_{k \geq 5}P(k,F(x))\cdot\log(1+x)  -\sum_{k \geq 5}P(k,F(x))\cdot \log(P(k,F(x))) + \frac{2x-1}{x+1}\cdot \log(2x-1) + \right. \\
& + & \left. \frac{1-x}{1+x}\cdot \log(1-x) \right]\\
& =  &(1+x)\cdot \left[-\frac{x}{1 + x}\cdot \log(1+x)  -\sum_{k \geq 5}P(k,F(x))\cdot  \log(P(k,F(x))) +
\frac{2x-1}{x+1}\cdot  \log(2x-1) + \right. \\
& + & \left. \frac{1-x}{1+x}\cdot \log (1-x) \right] \\
& = & (1+x)\cdot  \left[-\sum_{k \geq 5}P(k,F(x))\cdot  \log(P(k,F(x)))  + \frac{x}{1 + x} \cdot \log\left( \frac{x}{1 + x} \right) - \frac{x}{1+x}\cdot \log(x)+ \right.  \\
& + & \left. \frac{2x-1}{x+1}\cdot  \log(2x-1)+ \frac{1-x}{1+x}\cdot \log (1-x) \right] \\
&= & (1+x)\cdot H(F(x)) - x\cdot  \log (x) + (2x-1)\cdot  \log(2x-1) +  (1-x)\cdot  \log(1-x)
\end{eqnarray*}
\end{proof}
\newpage
\begin{theorem}
For $x \in [1/2, 1]$, we have the following functional scaling equation:
$$
S(x) = (x+2)\cdot H\left( \frac{1}{x+2} \right)
$$
\label{Tma12}
\end{theorem}
\begin{proof}
For $x > 1/2$, we have:
\begin{eqnarray*}
S(x) &  = & H(x) - 2\cdot (1-x) \cdot \log (1-x) - (2x - 1)\cdot \log(2x-1). \\
&&\textit{Applying the scaling equation of $H(x)$ when $x > 1/2$:} \\
& S(x) = & (1 + x)\cdot H\left( \frac{x}{x+1} \right) - x\cdot \log(x) + \underbrace{(2x-1)\cdot \log(2x-1)}_\text{I)} + (1-x) \cdot \log(1-x) - \\
& - & 2(1-x)\cdot \log(1-x) -  \underbrace{(2x-1)\cdot \log(2x-1)}_\text{I)} \\
& = & (1 + x)\cdot H\left( \frac{x}{x+1} \right) - x\cdot \log(x) - (1-x)\cdot \log(1-x) \\
&&\textit{By symmetry of $H$, we have $H(x) = H(1-x)$:} \\
& = & (1 + x)\cdot H\left(\frac{1}{x+1} \right) - x\cdot \log(x) - (1-x)\cdot \log(1-x) \\
&&\text{As $F(x) = \frac{x}{1+x} < \frac{1}{2}$, then $\frac{1}{x+1} = 1 - \frac{x}{x+1} > \frac{1}{2}$.} \\
&&\textit{Therefore, we apply again the scaling equation of $H(x)$ when $x > 1/2$:} \\
& = & (1 + x)\cdot \left[ \left( 1 + \frac{1}{x+1}\right)\cdot  H\left(\frac{1}{x+2} \right) - \frac{1}{x+1} \cdot \log \left( \frac{1}{x+1} \right) + \left( \frac{2}{x+1}-1\right) \cdot \log \left( \frac{2}{x+1}-1\right) \right.  \\
& + & \left. \left( 1 - \frac{1}{x+1}\right) \cdot \log \left( 1 - \frac{1}{x+1}\right) \right] - x\cdot \log(x) - (1-x)\cdot \log(1-x) \\
& = & (x + 2)\cdot H\left(\frac{1}{x+2} \right) + \log(x+1) + (1-x)\cdot \log \left( \frac{1-x}{1+x} \right) + x\cdot \log \left( \frac{x}{1+x} \right) - x\cdot \log(x) - (1-x)\cdot \log(1-x) \\
& = & (x + 2)\cdot H\left(\frac{1}{x+2} \right) + \underbrace{\log(x+1)}_\text{II)} + \underbrace{(1-2x) \cdot \log(1-x)}_\text{III)} - \underbrace{(1-x) \cdot \log(x+1)}_\text{II)} + \underbrace{x \cdot \log(x)}_\text{I)} -  \\
& - &\underbrace{x \cdot \log(x+1)}_\text{II)} -  \underbrace{x \cdot \log(x)}_\text{I)} -  \underbrace{(1-x) \cdot \log(1-x)}_\text{III)} \\
& = & (x + 2)\cdot H\left(\frac{1}{x+2} \right)
\end{eqnarray*}
\end{proof}
%\underbrace{(x + 2)^3}_\text{text 1}

\subsection{Some aspects of  \texorpdfstring{$H(x)$}{Hx}}
%\subsection{Some aspects of $H(x)$}
\label{appendix:D}
The entropy function $H(x)$ evaluated in $x = p/q$ is:

\begin{eqnarray*}
H(p/q) = - \sum_{i \geq 5} \frac{n_i}{q}\cdot \log \left( \frac{n_i}{q} \right) + \frac{p}{q} \cdot \log \left( \frac{p}{q} \right) \\
= - \frac{1}{q} \sum_{i \geq 5} n_i \cdot \log (n_i) + \frac{\log (q)}{q} \cdot \sum_{i \geq 5} n_i  + \frac{p}{q} \cdot \log \left( \frac{p}{q} \right) 
\end{eqnarray*}

Using that 
$$\sum_{i \geq 5} n_i = q - p - q + 2p = p$$

we obtain the following

\begin{eqnarray*}
H(p/q) = - \frac{1}{q} \sum_{i \geq 5} n_i \cdot \log (n_i) + \frac{p}{q} \log (q) \ + \frac{p}{q} \cdot \log \left( \frac{p}{q} \right) = - \frac{1}{q} \sum_{i \geq 5} n_i \cdot \log (n_i) + \frac{p}{q} \cdot \log (p) 
\end{eqnarray*}

Then:

\begin{eqnarray*}
\frac{q}{p} \cdot H(p/q) = log (p)  - \frac{1}{p} \sum_{i \geq 5} n_i \cdot \log (n_i) 
\end{eqnarray*}

For example, fractions $\frac{2}{q} = \frac{2}{2n + 1} < \frac{1}{2}$ with $n \geq 2$ always have a node $k_1$ with frequency $n_1 = 1$ and the outer node $k_2$ with frequency $n_2 = 1$. This leads to the fact that $\frac{q}{2} \cdot H\left(\frac{2}{q}\right)= \log(2)$, as corroborated in Fig. \ref{fig:Familias_23/n}.

%Besides that, Fig \ref{fig:Fig_racniveles} illustrates how the function $H(x)/x$ assigns null entropy to the rational numbers $1/n$. The connectivity distribution of these graphs are always $P(2) = 1/n$, $P(3) = 1 - 2/n$ and $P(n + 2) = 1/n$. The transformation $H(x)$ avoid the nodes with connectivity $2$, $3$ and the external node . Moreover, it seems that the numbers $\mathcal{C}_1(n)$ all have the same entropy. After an algebra exercise, it can be seen that:

If we now take the fractions of the numerator $p = 3$, we get two distinct possibilities: $\left\lbrace \frac{3}{3n + 1} \right\rbrace_{n \geq 2}$ and $\left\lbrace \frac{3}{3n + 2} \right\rbrace_{n \geq 2}$. It can be checked, as Fig. \ref{fig:Familias_23/n} illustrates for the $3$ first elements,  how the rational numbers $\frac{3}{3n + 1}$ always have $n_1 = 1$ and $n_2 = 2$, while $\frac{3}{3n + 2}$ have $n_1 = n_2 = n_3 = 1$. We obtain that this rational family always follows:

\begin{align*}
\frac{3n + 1}{3} \cdot H\left( \frac{3}{3n + 1} \right) &  = \log( 3 ) - \frac{2}{3}\cdot \log (2) \\
\frac{3n + 2}{3} \cdot H\left( \frac{3}{3n + 2} \right)&  = \log(3)
\end{align*}

The explanation for this fact is as follows: the roots of the intervals $\mathcal{I}_n$ are the concatenation of $G_{\frac{1}{n+1}}$ and $G_{\frac{1}{n}}$. The connectivity distribution is made up of $2$ nodes with $k = 2$, one internal link of $k = n + 3$, the extreme node $k = n + 4$. The other nodes have connectivity $k = 3$. In summary, each graph reached by paths $L^n R \mathcal{P}$ will have the same nodes with connectivity $p = 2$ and identical frequencies $n_i$, even if their connectivity values are different.

\subsection{Proof of global maxima of \texorpdfstring{$S(x)$}{Sx}}
%\subsection{Proof of global maxima of $S(x)$}
\label{appendix:E}
We shall now prove that $S$ defined over the set of Haros graphs is indeed maximal when the degree distribution coincides with the degree distribution of $G_{\phi^{-1}}$. %Before applying the technique of Lagrange multipliers to compute the degree distribution with maximal graph entropy, we need an intermediate result summarised in the following lemma:\\
For symmetry, we shall then consider the interval $(1/2,1]$ (we proceed analogously for $x'=1-\phi^{-1} \in[0,1/2]$ ). We will use the technique of Lagrange multipliers with two constraints for $P(k)$: normalisation of the degree distribution and arithmetic mean degree distribution (that is closely related to $q$ by construction). The Lagrangian functional has the form 
$${\cal L}[P(k)]=S(x) - (\text{normalization}) - (\text{arithmetic mean \ degree})$$
which will take an extremum at the solution $\delta {\cal L}[P^*(k)]=0$.

The normalisation constraints read
$${\cal P}=1-P(2)-P(3)-P(4)=1-(1-x)-(2x-1)-0 = 1-x$$
We know that if the number $x$ is irrational, $G_{x}$ will necessarily have an arithmetic mean degree $\overline{k} =4$. In such a case, the second constraint is enclosed in the constant

$${\cal Q}=4 - 2P(2)-3P(3)-4P(4)= 5-4x$$ 

If $x$ is instead a rational number $x=p/q$, the arithmetic mean degree $\overline{k}=4(1-1/2q)$ and the restriction on the mean degree are slightly different.

For irrational numbers, the Lagrangian functional reads
\begin{equation*}
\mathcal{L}[P(k)]=-\sum_{k=5}^{\infty}{P(k)\log{P(k)}}-(\lambda_0-1)\left(\sum_{k=5}^{\infty}{P(k)}-{\cal P}\right)
-\lambda_1\left(\sum_{k=5}^{\infty}{kP(k)}-{\cal Q}\right)
\end{equation*}
for which the extreme condition is
\begin{equation*}
\frac{\partial \mathcal{L}}{\partial P(k)}=-\log{P(k)}-\lambda_0-\lambda_1k=0
\end{equation*}
has the solution
\begin{equation*}
P(k)=e^{-\lambda_0-\lambda_1k}
\end{equation*}
From this, and using the definition of $\cal P$ we get
\begin{equation*}
\mathcal{P} = \sum_{k \geq 5}P(k)=\sum_{k=5}^{\infty}{e^{-\lambda_{0}-\lambda_{1}k}}=e^{-\lambda_0}\sum_{k=5}^{\infty}{e^{-\lambda_{1}k}}=1-x
\end{equation*}
As the infinite sum satisfies
\begin{equation}
\sum_{k=5}^{\infty}e^{-\lambda_{1}k}={\frac{e^{-4\lambda_{1}}}{e^{\lambda_{1}}-1}}
\tag{1} \label{suma}
\end{equation}
we have the following relationship between the Lagrange multipliers
\begin{equation}
e^{-\lambda_{0}}=\frac{(1-x)(e^{\lambda_{1}}-1)}{e^{-4\lambda_{1}}}
\tag{2} \label{elambda}
\end{equation}
For the reduced mean connectivity $\mathcal{Q}$ we get:
\begin{equation}
\mathcal{Q} =\sum_{k\geq 5}k P(k)=\sum_{k=5}^{\infty}k{e^{-\lambda_0-\lambda_1k}}=e^{-\lambda_0}\sum_{k=5}^{\infty}{k e^{\lambda_{1}k}}=5-4x
\tag{3} \label{Q}
\end{equation}
To calculate the sum, we can differenciate Eq. \ref{suma} with respect to $\lambda_1$, to obtain:
\begin{equation*}
\sum_{k=5}^{\infty}{ke^{-\lambda_{1}k}}=e^{-4\lambda_1} \frac{-4+5e^{\lambda_1}}{(e^{\lambda_{1}}-1)^2}
\end{equation*}
Substituting this sum into Eq. \ref{Q} and using Eq. \ref{elambda} we have
\begin{equation*}
5-4x =\left(\frac{(1-x)(e^{\lambda_{1}}-1)}{e^{-4\lambda_{1}}}\right)\left(e^{-4\lambda_1} \frac{-4+5e^{\lambda_1}}{(e^{\lambda_{1}}-1)^2}\right) = \frac{(1-x)(-4+5e^{\lambda_{1}})}{e^{\lambda_{1}}-1}
\end{equation*}
which, after some algebra, yields for the second Langrange multiplier
\begin{equation*}
x = e^{-\lambda_{1}}
\end{equation*}

\noindent Introducing the above result in Eq. \ref{elambda} we obtain for the first Lagrange multiplier
\begin{equation*}
e^{-\lambda_0} =\frac{(1-x)(x^{-1}-1)}{x^{4}} = (1-x)^{2}x^{-5}
\end{equation*}
Therefore, the degree distribution maximising the graph entropy is given by
\begin{equation*}
P(k,x)=\left\{
\begin{array}{ll}
1-x & k=2 \\
2x-1 & k=3 \\
0 			& k=4\\
(1-x)^{2}x^{k-5} & k\geq 5%
\end{array}%
\right.
\end{equation*}
But this is the distribution of $G_{\phi^{-1}}$, since
\begin{align*}
(1-\phi^{-1})^{2}(\phi^{-1})^{k-5} & = (\phi^{-1})^{k-5} \left(\frac{\phi-1}{\phi}\right)^2 \\
                                   & =  (\phi^{-1})^{k-5}(\phi^{-1})^{2}(1-\phi^{-1}) \\
 & = (\phi^{-1})^{k-5}(\phi^{-1})^{2}(\phi^{-1})^{2} \\ 
 & = (\phi^{-1})^{k-5+4} \\
 & = \phi^{1-k}
\end{align*}
Therefore, we conclude the proof.\qedsymbol

\subsection{Proof of local maxima}
\label{appendix:F}
We now prove that $S$ defined over the set of Haros graphs with $x \leq 1/n$ is maximal when the degree distribution coincides with the degree distribution of the graphs associated with noble numbers ${\cal C}_1(n)=\frac{1}{n+\phi^{-1}}$. As before, we establish two constraints for $P(k)$: normalisation of the degree distribution and the arithmetic mean degree distribution. An immediate consequence of Theorem \ref{Tma:Rep} tells us that if $x \leq 1/n$, then $P(x,k) = 0$ for $4 \leq k \leq n+2$. The normalisation constraint reads

\begin{equation*}
{\cal P} = 1 - P(2) - P(3) - \sum_{k \geq 4}^{n+2}P(k) = 1 - x - (1 - 2x) - 0 = x
\end{equation*}

\noindent and so, the arithmetic mean degree constraint now is

\begin{equation*}
{\cal Q} = 4 - 2P(2) - 3P(3) - \sum_{k \geq 4}^{n+2}kP(k) = 4 - 2x - 3(1 - 2x) - 0 = 1 + 4x
\end{equation*}
Then, for irrational numbers, the Lagrangian functional

\begin{equation*}
\mathcal{L}[P(k)]=-\sum_{k=n+3}^{\infty}{P(k)\log{P(k)}}-(\lambda_0-1)\left(\sum_{k=n+3}^{\infty}{P(k)}-{\cal P}\right)
-\lambda_1\left(\sum_{k=n+3}^{\infty}{kP(k)}-{\cal Q}\right)
\end{equation*}
for which the extremum condition reads

\begin{equation*}
\frac{\partial \mathcal{L}}{\partial P(k)}=-\log{P(k)}-\lambda_0-\lambda_1k=0
\end{equation*}
has the solution
\begin{equation*}
P(k)=e^{-\lambda_0-\lambda_1k}
\end{equation*}
For constraint $\cal P$ we have

\begin{equation*}
{\cal P} = \sum_{k = n + 3}^{\infty} P(k) = \sum_{k = n + 3}^{\infty} e^{-\lambda_0-\lambda_1k} =  e^{-\lambda_0} \sum_{k = n + 3}^{\infty} e^{-\lambda_1k} = x
\end{equation*}

\noindent As the infinite sum gives,

\begin{equation}
\label{lambda1}
\sum_{k = n + 3}^{\infty} e^{-\lambda_1k} = \frac{e^{-\lambda_1(n+2)}}{e^{\lambda_1}-1}
\end{equation}

\noindent we obtain

\begin{equation}
\label{lambda0}
e^{-\lambda_0} = x(e^{\lambda_1}-1)e^{\lambda_1(n+2)}
\end{equation}

\noindent The second constraint satisfies

\begin{equation}
\label{Q2}
{\cal Q} = \sum_{k = n + 3}^{\infty} kP(k) = \sum_{k = n + 3}^{\infty} k e^{-\lambda_0 - \lambda_{1}k} = e^{-\lambda_0} \sum_{k = n + 3}^{\infty} k e^{- \lambda_{1}k} = 1 + 4x
\end{equation}

Again, to calculate the infinite sum, we can differentiate equation \ref{lambda1} with respect to $\lambda_1$ and obtain the following result

\begin{equation*}
\sum_{k = n + 3}^{\infty} k e^{-\lambda_1k} = e^{-\lambda_1(n+2)} \cdot \frac{(n+3)e^{\lambda_1} - (n+2)}{(e^{\lambda_1}-1)^2}
\end{equation*}

And from equations \ref{Q2}  and \ref{lambda0}

\begin{equation*}
1 + 4x = e^{-\lambda_0} \sum_{n+3}^{\infty} k e^{- \lambda_{1}k} =
\end{equation*}

\begin{equation*}
\left( \frac{x(e^{\lambda_1}-1)}{e^{-\lambda_1(n+2)}} \right) \cdot \left(e^{-\lambda_1(n+2)}  \frac{(n+3)e^{\lambda_1} - (n+2)}{(e^{\lambda_1}-1)^2} \right) = \frac{x((n+3)e^{\lambda_1} - (n+2))}{e^{\lambda_1}-1}
\end{equation*}

Simplification of the above yields for the second Lagrange multiplier

\begin{equation*}
e^{\lambda_1} = \frac{1 - (n-2)x}{1 - (n-1)x}
\end{equation*}

Introducing the above equality in equation \ref{lambda0}, we get for the first Lagrange multiplier

\begin{equation*}
e^{-\lambda_0} = \frac{x^2}{1-(n-1)x} \left( \frac{1 - (n-2)x}{1 - (n-1)x} \right)^{n+2}
\end{equation*}

Therefore, the degree distribution maximising graph entropy is given by: 

\begin{equation}
P(k,x)=\left\{
\begin{array}{ll}
x  & k=2 \\
1 - 2x & k=3 \\
0 &  4 \leq k \leq n + 2 \\
\frac{x^2}{1-(n-1)x} \left( \frac{1 - (n-2)x}{1 - (n-1)x} \right)^{n+2-k} & k\geq n+3
\end{array}%
\right.  
\end{equation}%

This is the probability distribution of ${\cal C}_1(n)$ and using these three identities

\begin{align*}
\frac{{\cal C}_1(n)}{{\cal C}_1(n+1)} & = 1 + {\cal C}_1(n) \\
\frac{1 - (n-2){\cal C}_1(n)}{1 - (n-1){\cal C}_1(n)} &= \phi  \\
\frac{{\cal C}_1(n)^2}{1- (n-1){\cal C}_1(n)}  & = {\cal C}_1(n)\phi^{-1} \\ 
\end{align*}

the expression has simplified: $P(k) = {\cal C}_1(n)(\phi^{-1})^{k - (n+1)}$ \qedsymbol

\subsection{Slope of \texorpdfstring{$H(\mathcal{C}_1(n))$}{HC1} }
%\subsection{Slope of $H(\mathcal{C}_1(n))$}
\label{appendix:G}

\begin{eqnarray*}
H(\mathcal{C}_1 (n)) & = &  - \sum_{k = n + 3}^{  \infty} \mathcal{C}_1 (n) ( \phi^{-1} )^{k - (n+1)} \cdot \log \left( \mathcal{C}_1 (n) ( \phi^{-1} )^{k - (n+1)} \right) + \mathcal{C}_1 (n)\cdot \log(\mathcal{C}_1 (n)) \\
& = & -\mathcal{C}_1 (n)\left[ \sum_{k = n + 3}^{  \infty} ( \phi^{-1} )^{k - (n+1)} \cdot \log\left( \mathcal{C}_1 (n)\right) + \sum_{k = n + 3}^{  \infty} ( \phi^{-1} )^{k - (n+1)} \cdot \left(k - (n+1)\right)\cdot \log \left( \phi^{-1}\right) \right] + \\
& + & \mathcal{C}_1 (n)\cdot \log(\mathcal{C}_1 (n)) \\
& = & -\mathcal{C}_1 (n)\left[ \log(\mathcal{C}_1 (n))\cdot \frac{(\phi^{-1})^2}{1 - \phi^{-1}} + \log(\phi^{-1})\cdot \frac{(2 - \phi^{-1})\cdot (\phi^{-1})^2 }{(\phi^{-1} - 1)^2} \right] + \mathcal{C}_1 (n)\cdot \log(\mathcal{C}_1 (n)) \\
& = & -\mathcal{C}_1 (n)\left[ \log(\mathcal{C}_1 (n)) + \log(\phi^{-1})\cdot (3 + \phi^{-1}) \right] + \mathcal{C}_1 (n)\cdot \log(\mathcal{C}_1 (n)) \\ 
& = & -\mathcal{C}_1 (n)\cdot \log(\phi^{-1})\cdot (3 + \phi^{-1})
\end{eqnarray*}

Hence:

$$
\frac{H(\mathcal{C}_1 (n))}{\mathcal{C}_1 (n)} = - \log(\phi^{-1})\cdot (3 + \phi^{-1})
$$

\subsection{Slope of \texorpdfstring{$H(\mathcal{C}_{3}(n,1,2))$}{HC3} }
%\subsection{Slope of $H(\mathcal{C}_{3}(n,1,2))$ }
\label{appendix:H}

Let us $x = [n, 1, 2, \overline{1}] = \frac{3 + \phi^{-1}}{(3n + 2) + (n+1)\cdot \phi^{-1}}$ for $n\geq 2$, where:

To shorten the equations, we denote $\mathcal Q = \frac{1}{(3n + 2) + (n+1)\cdot \phi^{-1}}$. We have the following:

\begin{eqnarray*}
H(\mathcal{C}_3(n,1,2)) & = & - \frac{1}{\mathcal Q}\cdot \log \left( \frac{1}{\mathcal Q}\right) - \frac{1 + \phi^{-1}}{\mathcal Q}\cdot \log \left( \frac{1 + \phi^{-1}}{\mathcal Q} \right) - \\
& - & \sum_{k = n + 6}^{\infty}\left[ \frac{1 - \phi^{-1}}{\mathcal Q}\cdot(\phi^{-1})^{k - (n+6)}\cdot\log\left( \frac{1 - \phi^{-1}}{\mathcal Q}\cdot(\phi^{-1})^{k - (n+6)} \right) \right] + \frac{3 + \phi^{-1}}{\mathcal Q}\cdot \log\left( \frac{3 + \phi^{-1}}{\mathcal Q} \right)  \\
& = & -\frac{1}{\mathcal Q} \left[ (1 + \phi^{-1})\cdot\log(1 + \phi^{-1}) -(2+\phi^{-1})\cdot \log(\mathcal Q)- \right. \\
& + & \left. (1 + \phi^{-1})\cdot \left(\log\left(\frac{1-\phi^{-1}}{\mathcal Q}\right)\cdot(2 + \phi^{-1}) +  \log(\phi^{-1})\cdot\left(1 + \frac{2}{\phi^{-1}}\right) \right) - \right. \\
& - & \left. (3+ \phi^{-1})\cdot\log\left( \frac{3 + \phi^{-1}}{\mathcal Q} \right) \right] \\
& = & \frac{-1}{\mathcal Q}\left[ (1+\phi^{-1})\cdot \log(1 + \phi^{-1}) - (2 + \phi^{-1})\cdot \log(\mathcal Q) + \log\left(\frac{1-\phi^{-1}}{\mathcal Q}\right) \right. \\
& + &  \left. \phi\cdot\log(\phi^{-1}) - (3+\phi^{-1})\cdot\log\left(\frac{3 + \phi^{-1}}{\mathcal Q} \right) \right] \\
& = & \frac{-1}{\mathcal Q}\left[ (1 + \phi^{-1})\cdot\log(1 + \phi^{-1}) + \log(1-\phi^{-1}) + \phi\log(\phi^{-1}) + (3+\phi^{-1})\cdot\log(3+ \phi^{-1}) \right. \\
& - & \left. \log(\mathcal Q)\cdot(-2-\phi^{-1}-1+3 + \phi^{-1}) \right] = \\
& = & \frac{-1}{\mathcal Q}\left[ (1+\phi^{-1})\cdot\log(1 + \phi^{-1}) + \log (1 - \phi^{-1}) + \phi\log(\phi^{-1}) + (3+\phi^{-1})\cdot \log(3 + \phi^{-1}) \right] =  \\
& = & \frac{-1}{\mathcal Q}\left[\log(1 -\phi^{-1}) -(3 + \phi^{-1})\cdot\log(3 + \phi^{-1}) \right]
\end{eqnarray*}
Thus:
$$
H(\mathcal{C}_3(n,1,2)) = \mathcal{C}_3(n,1,2) \cdot \left( \log(3 + \phi^{-1}) - \frac{1}{3+\phi^{-1}}\cdot \log(1-\phi^{-1}) \right)
$$
\newpage
\subsection*{Acknowledgments}
JC and BL acknowledge funding from Spanish Ministry of Science and Innovation under project M2505 (PID2020-113737GB-I00). LL acknowledges funding from projects DYNDEEP (EUR2021-122007), MISLAND (PID2020-114324GB-C22) and through the Severo Ochoa and María de Maeztu Program for Centers and Units of Excellence in R\&D (MDM-2017-0711), all of them funded by the Spanish Ministry of Science and Innovation via AEI.

\bibliographystyle{ieeetr}
\bibliography{On_a_graph_representation_of_real_numbers.bib}
\end{document}